\documentclass[11pt,leqno]{amsart}

\usepackage[a4paper, right=20mm, left=20mm, top=20mm]{geometry}

\usepackage{amsthm}
\usepackage[english]{babel}
\usepackage[utf8]{inputenc} 
\usepackage{hyperref}
\usepackage{amsmath}
\usepackage{amssymb}
\usepackage{mathtools}
\usepackage{mathrsfs}
\usepackage{enumerate}

\usepackage{caption}
\usepackage{tikz}
\usetikzlibrary{matrix,arrows}

\numberwithin{equation}{section}

\theoremstyle{plain}
\newtheorem{Theorem}{Theorem}[section]
\newtheorem{Lemma}[Theorem]{Lemma}
\newtheorem{Proposition}[Theorem]{Proposition}
\newtheorem{Corollary}[Theorem]{Corollary}
\newtheorem{thmalph}{Theorem}

\theoremstyle{definition}
\newtheorem{Definition}[Theorem]{Definition}
\newtheorem{Remark}[Theorem]{Remark}
\newtheorem{Example}[Theorem]{Example}
\newtheorem*{Definition*}{Definition}

\newcommand{\C}{\mathbb{C}}
\newcommand{\R}{\mathbb{R}}

\newcommand{\Z}{\mathbb{Z}}
\newcommand{\kk}{\Bbbk}

\newcommand{\Cdot}{\boldsymbol{\cdot}}
\newcommand{\Hom}{\mathrm{Hom}}
\newcommand{\Ext}{\mathrm{Ext}}
\newcommand{\Stab}{\mathrm{Stab}}
\newcommand{\Rep}{\mathrm{Rep}}
\newcommand{\Tilt}{\mathrm{Tilt}}

\newcommand{\SL}{\mathrm{SL}}
\newcommand{\gfd}{\mathrm{gfd}}
\newcommand{\pr}{\mathrm{pr}}
\newcommand{\om}{\/\textsuperscript{$\omega$}}

\DeclareMathOperator{\ch}{ch}

\providecommand{\abs}[1]{\lvert#1\rvert}

\title[Multiplicity free and completely reducible tensor products]{Multiplicity free and completely reducible\\tensor products for $\mathrm{SL}_3(\Bbbk)$ and $\mathrm{Sp}_4(\Bbbk)$}

\author{Jonathan Gruber}
\address{Department of Mathematics, FAU Erlangen-Nuremberg, Cauerstr.\ 11, 91058 Erlangen, Germany}
\email{jonathan.gruber@fau.de}

\author{Gaëtan Mancini}
\address{School of Mathematics and Natural Sciences, University of Wuppertal, Gaußstr.\ 20, 42119 Wuppertal, Germany}
\email{mancini@uni-wuppertal.de}

\date{\today}

\begin{document}

\begin{abstract}
	Let $G$ be a simple algebraic group over an algebraically closed field $\kk$ of positive characteristic.
	We consider the questions of when the tensor product of two simple $G$-modules is multiplicity free or completely reducible.
	We develop tools for answering these questions in general, and we use them to provide complete answers for the groups $G = \SL_3(\kk)$ and $G = \mathrm{Sp}_4(\kk)$.
\end{abstract}

\maketitle

\section*{Introduction}

Let $G$ be a simply connected simple algebraic group over an algebraically closed field $\kk$ of characteristic $p>0$.
The category $\Rep(G)$ of finite-dimensional rational $G$-modules is a non-semisimple tensor category, and the composition multiplicities and submodule structure of tensor products of simple $G$-modules are a priori very hard to understand.
In this article, we propose to gain a better understanding of tensor products in $\Rep(G)$ by classifying the pairs of simple modules for which the composition multiplicities or the submodule structure of the tensor product are as elementary as possible.
More concretely, we ask when the tensor product of two simple $G$-modules is \emph{multiplicity free} (i.e.\ all composition multiplicities are bounded by $1$) or \emph{completely reducible} (i.e.\ a direct sum of simple $G$-modules).
These questions are closely related to one another:
If the tensor product of two simple $G$-modules is multiplicity free then it is also completely reducible by \cite[Lemma 4.12]{GruberCompleteReducibility}, and the converse is true under additional assumptions (e.g.\ if all weight spaces of one of the simple $G$-modules are at most one-dimensional), as we observe in Subsection \ref{subsec:weightspaces}.
We establish a number of new tools and techniques to decide whether a tensor product of simple $G$-modules is completely reducible or multiplicity free, and we show that these questions are intimately linked with
the \emph{alcove geometry} which governs the decomposition of $\Rep(G)$ into blocks.
Using these tools, we achieve a complete classification of all pairs of simple $G$-modules whose tensor product is completely reducible or multiplicity free, for the groups $G=\SL_3(\kk)$ and $G = \mathrm{Sp}_4(\kk)$.
In order to formulate our results more precisely, we next introduce some more notation.

The group $G$ is equipped with a root system $\Phi$, and the isomorphism classes of simple $G$-modules in $\Rep(G)$ are parametrized by the set $X^+$ of dominant weights in the weight lattice $X$ of $\Phi$, with respect to a fixed choice of simple roots $\Pi \subseteq \Phi$.
The unique (up to isomorphism) simple $G$-module of highest weight $\lambda \in X^+$ is denoted by $L(\lambda)$.
We write $W_\mathrm{fin}$ for the Weyl group of $\Phi$ and $W_\mathrm{aff} = W_\mathrm{fin} \rtimes \Z\Phi$ for the affine Weyl group.
Both $W_\mathrm{fin}$ and $W_\mathrm{aff}$ are Coxeter groups, with simple reflections $S_\mathrm{fin}$ and $S_\mathrm{aff} = S_\mathrm{fin} \sqcup \{ s_0 \}$, respectively, where $s_0$ denotes the affine simple reflection (see Subsection \ref{subsec:alcoves} for more details).
Furthermore, $W_\mathrm{aff}$ acts on $X$ via the so-called ($p$-dilated) dot action $(x,\lambda) \mapsto x \Cdot \lambda$, and two simple $G$-modules $L(\lambda)$ and $L(\mu)$ belong to the same block of $\Rep(G)$ only if their highest weights $\lambda$ and $\mu$ belong to the same $W_\mathrm{aff}$-orbit.

The fixed points of a reflection in $s \in W_\mathrm{aff}$ under the dot action form an affine hyperplane $H_s$ in the Euclidean space $X_\R = X \otimes_\Z \R$, and the connected components of $X_\R \setminus \bigcup_s H_s$ are called \emph{alcoves}.
The closure of any alcove is a fundamental domain for the dot action of $W_\mathrm{aff}$ on $X_\R$, and every point $x \in X_\R$ belongs to the `upper closure' of a unique alcove.
We say that a reflection hyperplane $H = H_s$ is a \emph{wall} of an alcove $C \subseteq X_\R$ if $H$ is the unique reflection hyperplane separating the alcoves $C$ and $s \Cdot C$.
With these conventions in place, we can define our key alcove geometric criterion for the complete reducibility of tensor products.

\begin{Definition*}
	Let  $\lambda , \mu \in X^+$ and let $C \subseteq X_\R$ be the unique alcove whose upper closure contains $\lambda$.
	We call $\mu$ \emph{reflection small} with respect to $\lambda$ if, for every reflection $s \in W_\mathrm{aff}$ such that $H_s$ is a wall of $C$ and $\lambda \leq s \Cdot \lambda$, we have $\lambda+w(\mu) \leq s\Cdot(\lambda + w(\mu))$ for all $w \in W_\mathrm{fin}$.
\end{Definition*}
 
This seemingly technical condition is in fact easy to verify in any given example:
One only needs to check that all weights of the form $\lambda + w(\mu)$ lie below certain reflection hyperplanes determined by $\lambda$, and this amounts to a simple system of linear inequalities (see Examples \ref{ex:A2reflectionsmall} and \ref{ex:B2reflectionsmall}).
Our motivation for introducing the notion of reflection small weights is the following result; see Theorem \ref{thm:reflectionsmalltensorproduct} below.

\begin{thmalph} \label{thm:introreflectionsmallCR}
	Let $\lambda,\mu \in X^+$ such that $\mu$ is reflection small with respect to $\lambda$.
	Then $L(\lambda) \otimes L(\mu)$ is completely reducible.
\end{thmalph}

If $\mu \in X^+$ is reflection small with respect to $\lambda \in X^+$, we further show in Theorem \ref{thm:reflectionsmalltensorproduct} that all composition factors of $L(\lambda) \otimes L(\mu)$ have highest weights in the upper closure of the unique alcove whose upper closure contains $\lambda$, and we give explicit formulas for the composition multiplicities of the simple $G$-modules with highest weight in the upper closure of this alcove, in terms of dimensions of weight spaces in $L(\mu)$.
These formulas are a useful tool for determining the pairs of simple $G$-modules whose tensor product is multiplicity free, among the pairs of simple $G$-modules whose tensor product is completely reducible.

Another key component of our strategy for studying complete reduciblity and multiplicity freeness of tensor products is a reduction to pairs of simple modules with \emph{$p$-restricted} highest weights, based on results from \cite{GruberCompleteReducibility}.
A weight $\lambda \in X^+$ is called \emph{$p$-restricted} if $(\lambda,\alpha^\vee) < p$ for all simple roots $\alpha \in \Pi$, and we write $X_1 \subseteq X^+$ for the set of $p$-restricted weights.
Every weight $\mu \in X^+$ admits an expansion $\mu = \mu_0 + p \mu_1 + \cdots + p^m \mu_m$ with $\mu_i \in X_1$ for $i = 0 , \ldots , m$, and Steinberg's tensor product theorem gives a tensor product decomposition
\[ L(\mu) \cong L(\mu_0)^{[0]} \otimes L(\mu_1)^{[1]} \otimes \cdots \otimes L(\mu_m)^{[m]} , \]
where $M \mapsto M^{[r]}$ denotes the $r$-th Frobenius twist functor on $\Rep(G)$.
Now for weights $\lambda,\mu \in X^+$ with expansions $\lambda = \lambda_0 + \cdots + p^m \lambda_m$ and $\mu = \mu_0 + \cdots + p^m \mu_m$ such that $\lambda_i,\mu_i \in X_1$ for $i = 0 , \ldots , m$, it is shown in \cite[Theorem C]{GruberCompleteReducibility} that the tensor product $L(\lambda) \otimes L(\mu)$ is completely reducible if and only if the tensor products $L(\lambda_i) \otimes L(\mu_i)$ are completely reducible for $i = 0 , \ldots , m$.
Using this result, and the fact that multiplicity freeness implies complete reducibility, we show in Subsection \ref{subsec:prestricted} that an analogous reduction result holds for multiplicity freeness:
The tensor product $L(\lambda) \otimes L(\mu)$ is multiplicity free if and only if $L(\lambda_i) \otimes L(\mu_i)$ is multiplicity free for $i = 0 , \ldots , m$.
Therefore, in order to obtain a complete classification of the pairs of weights $\lambda,\mu \in X^+$ such that $L(\lambda) \otimes L(\mu)$ is completely reducible (or multiplicity free), it suffices to classify the pairs of $p$-restricted weights $\lambda^\prime,\mu^\prime \in X_1$ such that $L(\lambda^\prime) \otimes L(\mu^\prime)$ is completely reducible (or multiplicity free). 
This is the task that we accomplish for the groups $G = \SL_3(\kk)$ and $G = \mathrm{Sp}_4(\kk)$.

To conclude the introduction, we state our classification results for the group $G = \SL_3(\kk)$.
(The results for $G = \mathrm{Sp}_4(\kk)$ are stated in Theorems \ref{thm:B2CR} and \ref{thm:B2MF}.)
The weight lattice $X$ of $G = \SL_3(\kk)$ is spanned by the fundamental dominant weights $\varpi_1$, $\varpi_2$, and every $p$-restricted weight $\lambda \in X_1$ can be written uniquely as $\lambda=a\varpi_1+b\varpi_2$, with $0 \leq a,b < p$.
Furthermore, there are precisely two alcoves $C_0$ and $C_1$ whose upper closure contains $p$-restricted weights.

\begin{thmalph} \label{thm:introA2CRMF}
	Let $G = \SL_3(\kk)$ and let $\lambda,\mu \in X_1$ be $p$-restricted weights.
	\begin{enumerate}
		\item The tensor product $L(\lambda) \otimes L(\mu)$ is completely reducible if and only if $\lambda$ and $\mu$ satisfy one of the conditions 1--4 or 1*--3* in Table \ref{tab:introA2CRMF}, up to interchanging $\lambda$ and $\mu$.
		\item The tensor product $L(\lambda) \otimes L(\mu)$ is multiplicity free if and only if $\lambda$ and $\mu$ satisfy one of the conditions 1--3, 1*--3*, or 4a, 4a* or 4b in Table \ref{tab:introA2CRMF}, up to interchanging $\lambda$ and $\mu$.
	\end{enumerate}
	\begin{table}[htbp]
	\centering
	\caption{Weights $\lambda,\mu \in X_1 \setminus \{ 0 \}$ such that $L(\lambda) \otimes L(\mu)$ is c.r.\ or m.f., for $G = \SL_3(\kk)$.}
	\label{tab:introA2CRMF}
	\begin{tabular}{|l||l|}
	\hline
	$\mathbf{\#}$ & \textbf{conditions} \\
	\hline\hline
	1 & $\lambda = a\varpi_1$ and $\mu = a^\prime \varpi_1$, where $a+a^\prime = p-1$ \\
	\hline
	1* & $\lambda = b\varpi_2$ and $\mu = b^\prime \varpi_2$, where $b+b^\prime = p-1$ \\
	\hline
	2 & $\lambda = (p-1) \cdot \varpi_1$ and $\mu = \varpi_2$ \\
	\hline
	2* & $\lambda = (p-1) \cdot \varpi_2$ and $\mu = \varpi_1$ \\
	\hline
	3 & $\lambda = a \varpi_1 + b \varpi_2 \in C_1$ and $\mu = (b+1) \cdot \varpi_2$, where $a+b = p-1$ and $b<a$ \\
	\hline
	3* & $\lambda = a \varpi_1 + b \varpi_2 \in C_1$ and $\mu = (a+1) \cdot \varpi_1$, where $a+b = p-1$ and $a<b$ \\
	\hline
	4 & $\lambda \in C_0 \cup C_1$ and $\mu$ is reflection small with respect to $\lambda$ \\
	\hline \hline
	4a & $\lambda \in C_0 \cup C_1$ and $\mu = a^\prime \varpi_1$ is reflection small with respect to $\lambda$; \\
	\hline
	4a* & $\lambda \in C_0 \cup C_1$ and $\mu = b^\prime \varpi_2$ is reflection small with respect to $\lambda$; \\
	\hline
	4b & $\lambda = a \varpi_1 + b \varpi_2 \in C_1$, where $a+b=p-1$, and $\mu$ is reflection small with respect to $\lambda$; \\
	\hline
	\end{tabular}
	\end{table}
\end{thmalph}

We observe that in Theorem \ref{thm:introA2CRMF}, the tensor product $L(\lambda) \otimes L(\mu)$ is completely reducible if and only if either $L(\lambda) \otimes L(\mu)$ is multiplicity free or $\mu$ is reflection small with respect to $\lambda$, and the same observation can be made in our results for $G = \mathrm{Sp}_4(\kk)$ (see Section \ref{sec:B2}).
This, in our view, underlines both the importance of the reflection smallness condition for the study of complete reducibility of tensor products, and the fact that complete reducibility and multiplicity freeness of tensor products are closely related phenomena.

The article is organized as follows:
Section \ref{sec:preliminaries} contains definitions and preliminary results about representations of algebraic groups and the associated alcove geometry.
In Section \ref{sec:alcovegeometryCR}, we establish the relation between the reflection smallness condition and complete reducibility of tensor products, and we give the proof of Theorem \ref{thm:introreflectionsmallCR} (see Theorem \ref{thm:reflectionsmalltensorproduct}).
Section \ref{sec:CRMFcriteria} introduces further tools and criteria for deciding whether a tensor product of simple $G$-modules is completely reducible or multiplicity free.
For instance, we prove the results that enable us to reduce these questions to pairs of simple $G$-modules with $p$-restricted highest weights (Subsection \ref{subsec:prestricted}), and we establish complete reducibility and multiplicity freeness criteria using the notion of \emph{good filtration dimension} from \cite{FriedlanderParshallGFD} (Subsection \ref{subsec:GFD}) and the tensor ideal of \emph{singular $G$-modules} from \cite{GruberLinkageTranslation} (Subsection \ref{subsec:singular}).
The two final Sections \ref{sec:A2} and \ref{sec:B2} are devoted to our classification results for the groups $G = \SL_3(\kk)$ and $G = \mathrm{Sp}_4(\kk)$, respectively.
See Theorems \ref{thm:A2CR} and \ref{thm:A2MF} for the classification of completely reducible and multiplicity free tensor products of simple $G$-modules for $G = \SL_3(\kk)$ (as stated in Theorem \ref{thm:introA2CRMF}), and Theorems \ref{thm:B2CR} and \ref{thm:B2MF} for the analogous results for $G = \mathrm{Sp}_4(\kk)$.

\subsubsection*{Acknowledgments}
	This project has originated from the master's thesis of G.M.\ \cite{ManciniMF}, which contains a classification of the pairs of simple $G$-modules whose tensor product is multiplicity free for $G = \SL_3(\kk)$ and partial results for $G = \mathrm{Sp}_4(\kk)$, using different methods than the present article.
	The thesis was written under the direction of Donna Testerman, and we would like to express our deepest gratitude for her valuable advice and for encouraging our collaboration.
	We acknowledge support by the Swiss National Science Foundation via the grants FNS 200020\_207730 and P500PT\_206751.

\section{Preliminaries}
\label{sec:preliminaries}

\subsection{Roots and weights}

Let $\Phi$ be a simple root system in a Euclidean space $X_\R$ with scalar product $(-\,,-)$.
Let $\Phi^+ \subseteq \Phi$ be a positive system, corresponding to a base $\Pi = \{ \alpha_1, \ldots ,\alpha_n \}$ of $\Phi$, and write $\alpha_\mathrm{hs}$ and $\alpha_\mathrm{h}$ for the highest short root and the highest root in $\Phi$, respectively (with the convention that $\alpha_\mathrm{hs} = \alpha_\mathrm{h}$ if $\Phi$ is simply laced).
For $\alpha \in \Phi$, let $\alpha^\vee = 2\alpha/(\alpha,\alpha)$ be the coroot of $\alpha$, and let $s_\alpha$ be the reflection corresponding to $\alpha$, with $s_\alpha(x) = x - (x,\alpha^\vee) \cdot \alpha$ for $x \in X_\R$.
We write
\[ X = \{ \lambda \in X_\R \mid (\lambda,\alpha^\vee) \in \Z \text{ for all } \alpha \in \Phi \} \]
for the weight lattice and $W_\mathrm{fin} = \langle s_\alpha \mid \alpha \in \Phi \rangle$ for the (finite) Weyl group of $\Phi$, with simple reflections $S_\mathrm{fin} = \{ s_\alpha \mid \alpha \in \Pi \}$ and length function $\ell \colon W_\mathrm{fin} \to \Z_{\geq 0}$.
Furthermore, we write $\rho = \frac12 \cdot \sum_{\alpha \in \Phi^+} \alpha$ and denote by $h = (\rho,\alpha_\mathrm{hs}^\vee)+1$ the Coxeter number of $\Phi$.
The set of \emph{dominant weights} is
\[ X^+ = \{ \lambda \in X \mid (\lambda,\alpha^\vee) \geq 0 \text{ for all } \alpha \in \Phi^+ \} , \]
and the \emph{fundamental dominant weights} $\{ \varpi_1 , \ldots , \varpi_n \}$ are defined by $(\varpi_i,\alpha_j^\vee) = \delta_{ij}$ for $1 \leq i,j \leq n$.
We consider the partial order $\leq$ on $X$ such that $\lambda \leq \mu$ if and only if $\mu - \lambda \in \sum_{\alpha \in \Pi} \Z_{\geq 0} \cdot \alpha$.

\subsection{Representations and characters}
\label{subsec:representationscharacters}

Let $\kk$ be an algebraically closed field of characteristic $p>0$ and let $G$ be the unique (up to isomorphism) simply connected simple algebraic group over $\kk$ with root system $\Phi$.
We identify $X$ with the character lattice of a fixed maximal torus of $G$ and write $\Rep(G)$ for the category of finite-dimensional rational $G$-modules.
In the following, we will simply refer to the objects of $\Rep(G)$ as \emph{$G$-modules} (omitting the mention of rationality and finite-dimensionality).

Every $G$-module $M$ admits a weight space decomposition $M = \bigoplus_{\mu \in X} M_\mu$, and the character of $M$ is defined as $\ch M = \sum_\mu \dim M_\mu \cdot e^\mu \in \Z[X]$.
The simple $G$-modules are determined up to isomorphism by their highest weight in $X^+$, and we write $L(\lambda)$ for the simple $G$-module of highest weight~$\lambda$.
Every $G$-module $M$ has a finite composition series, and we write $[M:L(\lambda)]$ for the multiplicity of the simple $G$-module $L(\lambda)$ in a composition series of $M$.
We also denote by $\nabla(\lambda)$ and $\Delta(\lambda)$ the induced module and the Weyl module of highest weight $\lambda \in X^+$ (see Subsections II.2.1--2 and II.2.13 in \cite{Jantzen}), so that $L(\lambda)$ is the unique simple submodule of $\nabla(\lambda)$ and the unique simple quotient of $\Delta(\lambda)$.
Writing $M^*$ for the dual of a $G$-module $M$, we have $L(\lambda)^* \cong L(-w_0\lambda)$ and $\nabla(\lambda)^* \cong \Delta(-w_0\lambda)$ for all $\lambda \in X^+$, where $w_0 \in W_\mathrm{fin}$ denotes the longest element.

The characters of both $\nabla(\lambda)$ and $\Delta(\lambda)$ are given by the well-known Weyl character formula, cf.\ Subsections II.5.10--11 in \cite{Jantzen}.
We write $\chi(\lambda) = \ch \nabla(\lambda) = \ch \Delta(\lambda)$ and extend the definition of $\chi(\lambda)$ to all $\lambda \in X$ (possibly non-dominant) by requiring that $\chi(w \Cdot \mu) = (-1)^{\ell(w)} \cdot \chi(\mu)$ for $\mu \in X$ and $w \in W_\mathrm{fin}$, and $\chi(\mu)=0$ if $(\mu,\alpha^\vee) = -1$ for some $\alpha \in \Pi$.
The characters $\chi(\lambda)$ with $\lambda \in X^+$ form a basis of the ring $\Z[X]^{W_\mathrm{fin}}$ of $W_\mathrm{fin}$-invariants in $\Z[X]$, and so do the characters $\ch L(\lambda)$ of the simple $G$-modules, by Section II.5.8 in \cite{Jantzen}.
We write
\[ \Lambda(\lambda) = \{ \mu \in X \mid \Delta(\lambda)_\mu \neq 0 \} \]
for the set of weights of $\Delta(\lambda)$ and note that $\Lambda(\lambda) = \mathrm{conv}( W_\mathrm{fin} \lambda ) \cap ( \lambda + \Z\Phi )$ by \cite[Theorem 14.18]{FultonHarris}, where $\mathrm{conv}(-)$ denotes the convex hull of a subset of $X_\R$.

A \emph{good filtration} of a $G$-module $M$ is a filtration
\[ 0 = M_0 \subseteq \cdots \subseteq M_r = M \]
such that $M_i / M_{i-1} \cong \nabla(\lambda_i)$ for some $\lambda_i \in X^+$ for $i=1,\ldots,r$, and a \emph{Weyl filtration} of a $G$-module $N$ is a filtration
\[ 0 = N_0 \subseteq \cdots \subseteq N_s = N \]
such that $N_i / N_{i-1} \cong \Delta(\mu_i)$ for some $\mu_i \in X^+$ for $i=1,\ldots,s$.
For a $G$-module $M$ that admits a good filtration (as above), the multiplicity of $\nabla(\lambda)$ in any good filtration of $M$ is defined by
\[ [ M : \nabla(\lambda) ]_\nabla  \coloneqq \abs{ \{ i \mid \lambda_i = \lambda \} } = \dim \Hom_G\big( \Delta(\lambda) , M \big) \]
for all $\lambda \in X^+$ (see Proposition II.4.16 in \cite{Jantzen}), and for a $G$-module $N$ that admits a Weyl filtration, the multiplicity $[N : \Delta(\lambda)]_\Delta$ of $\Delta(\lambda)$ in a Weyl filtration of $N$ is defined analogously.
A $G$-module $M$ is called a \emph{tilting module} if it admits both a good filtration and a Weyl filtration.
For all $\lambda \in X^+$, there is a unique indecomposable tilting module $T(\lambda)$ of highest weight $\lambda$, and every tilting module can be written as a finite direct sum of these indecomposable tilting modules \cite{RingelAlmostSplit,DonkinTiltingModules}.
Furthermore, the class of $G$-modules admitting a good filtration is closed under tensor products \cite{WangGoodFiltration,DonkinGoodFiltration,MathieuGoodFiltration}, and hence so are the class of all $G$-modules admitting a Weyl filtration and the full subcategory $\Tilt(G)$ of tilting modules in $\Rep(G)$.

\subsection{Alcove geometry}
\label{subsec:alcoves}

Let $W_\mathrm{aff} = \Z\Phi \rtimes W_\mathrm{fin}$ be the affine Weyl group of $G$.
We write $\gamma \mapsto t_\gamma$ for the canonical embedding of $\Z\Phi$ into $\Z\Phi \rtimes W_\mathrm{fin} = W_\mathrm{aff}$.
The \emph{($p$-dilated) dot action} of $W_\mathrm{aff}$ on $X_\R$ is defined by
\[ t_\gamma w \Cdot x = w( x + \rho ) - \rho + p \gamma \]
for $\gamma \in \Z\Phi$, $w \in W_\mathrm{fin}$ and $x \in X_\R$.
In the following, we recall some results about the alcove geometry associated with the dot action of $W_\mathrm{aff}$ on $X_\R$; we refer the reader to Subsections II.6.1--5 in \cite{Jantzen} for more details and additional references.

The affine Weyl group is generated by the affine reflections $s_{\alpha,r} = t_{r\alpha} s_\alpha$, for $\alpha \in \Phi^+$ and $r \in \Z$, and the fixed points of $s_{\alpha,r}$ with respect to the $p$-dilated dot action form the affine hyperplane
\[ H_{\alpha,r} = \{ x \in X_\R \mid (x+\rho,\alpha^\vee) = pr \} . \]
The connected components of $X_\R \setminus \big( \bigcup_{\alpha,r} H_{\alpha,r} \big)$ are called \emph{alcoves}, so an alcove is any non-empty set of the form
\[ C = \{ x \in X_\R \mid p n_\alpha < (x+\rho,\alpha^\vee) < p \cdot (n_\alpha + 1) \text{ for all } \alpha \in \Phi^+ \} , \]
for some collection of integers $n_\alpha \in \Z$ with $\alpha \in \Phi^+$.
The \emph{upper closure} of the alcove $C$ is defined by
\[ \widehat C = \{ x \in X_\R \mid p n_\alpha < (x+\rho,\alpha^\vee) \leq p \cdot (n_\alpha + 1) \text{ for all } \alpha \in \Phi^+ \} . \]
For every element $x \in X_\R$, there is a unique alcove $C$ with $x \in \widehat{C}$.
We call
\begin{align*}
	C_0 & = \{ x \in X_\R \mid 0 < (x+\rho,\alpha^\vee) < p \text{ for all } \alpha \in \Phi^+ \} \\
	& = \{ x \in X_\R \mid 0 < (x+\rho,\alpha^\vee) \text{ for all } \alpha \in \Pi \text{ and } (x+\rho,\alpha_\mathrm{hs}^\vee) < p \}
\end{align*}
the \emph{fundamental alcove}.
The affine Weyl group acts simply transitively on the set of alcoves and the closure of every alcove is a fundamental domain.
In particular, $W_\mathrm{aff}$ is in bijection with the set of alcoves via $w \mapsto w \Cdot C_0$.
A weight $\lambda \in X$ is called \emph{$p$-regular} if it belongs to an alcove and \emph{$p$-singular} if it belongs to one of the reflection hyperplanes $H_{\alpha,r}$ with $\alpha \in \Phi^+$ and $r \in \Z$.

The affine Weyl group is a Coxeter group with simple reflections $S_\mathrm{aff} = S_\mathrm{fin} \sqcup \{ s_0 \}$, where $s_0 \coloneqq s_{\alpha_\mathrm{hs},1}$, and we write $\ell \colon W_\mathrm{aff} \to \Z_{\geq 0}$ for the length function.
For all $x \in W_\mathrm{aff}$, the coset $W_\mathrm{fin} x$ contains a unique element of minimal length, and we define
\[ W_\mathrm{aff}^+ = \{ x \in W_\mathrm{aff} \mid x \text{ has minimal length in } W_\mathrm{fin} x \} . \]
If $p \geq h$ then we also have $W_\mathrm{aff}^+ = \{ x \in W_\mathrm{aff} \mid x\Cdot0 \in X^+ \}$.
An alcove $C = w \Cdot C_0$ is called \emph{dominant} if $w \in W_\mathrm{aff}^+$.
We say that a reflection hyperplane $H = H_{\alpha,r}$ separates two points $x,y \in X_\R$ if either
\[ (x+\rho,\alpha^\vee) < pr < (y+\rho,\alpha^\vee) \qquad \text{or} \qquad (y+\rho,\alpha^\vee) < pr < (x+\rho,\alpha^\vee) . \]
Similarly, $H$ separates two alcoves $C , C^\prime \subseteq X_\R$ if there exist points $x \in C$ and $y \in C^\prime$ such that $H$ separates $x$ and~$y$.
(Equivalently, $H$ separates $x$ and $y$ for all $x \in C$ and $y \in C^\prime$.)
The alcoves $C$ and $C^\prime$ are called \emph{adjacent} if there is a unique reflection hyperplane $H = H_{\alpha,r}$ separating $C$ and $C^\prime$, and in that case, we have $C^\prime = s_{\alpha,r} \Cdot C$ and $H_{\alpha,r}$ is called a \emph{wall} of $C$ and $C^\prime$.
For every wall $H = H_{\alpha,r}$ of $C$, we have either $(x+\rho,\alpha^\vee) < pr$ for all $x \in C$ or $(x+\rho,\alpha^\vee) > pr$ for all $x \in C$, and accordingly, we say that $H$ belongs to the upper closure or to the lower closure of $C$, respectively.
For an alcove $C \subseteq X_\R$ and a reflection $s = s_{\alpha,r}$ such that $(x+\rho,\alpha^\vee) < pr$ for all $x \in C$, we write $C \uparrow s \Cdot C$, and we define the \emph{linkage partial order} $\uparrow$ on the set of alcoves as the reflexive and transitive closure of this relation; see Subsection II.6.5 in \cite{Jantzen}.

\subsection{Linkage principle and translation functors}

The linkage principle asserts that two simple $G$-modules $L(\lambda)$ and $L(\mu)$ belong to the same block of $\Rep(G)$ only if the highest weights $\lambda,\mu \in X^+$ belong to the same $W_\mathrm{aff}$-orbit with respect to the $p$-dilated dot action; see Subsections II.7.1--3 in \cite{Jantzen}.
For $\lambda \in \overline{C}_0 \cap X$, we let $\Rep_\lambda(G)$ be the full subcategory of $\Rep(G)$ whose objects are the $G$-modules all of whose composition factors have highest weight in $W_\mathrm{aff} \Cdot \lambda$ (i.e.\ the Serre subcategory generated by the simple $G$-modules $L(w \Cdot \lambda)$, for $w \in W_\mathrm{aff}$ such that $w\Cdot\lambda \in X^+$), and we call $\Rep_\lambda(G)$ the \emph{linkage class} of $\lambda$.
Then by the linkage principle, there is a canonical projection functor
\[ \mathrm{pr}_\lambda \colon \Rep(G) \longrightarrow \Rep_\lambda(G) \]
which sends a $G$-module $M$ to its largest submodule $\pr_\lambda M$ that belongs to $\Rep_\lambda(G)$, and there is a direct sum decomposition $M = \bigoplus_{\lambda \in \overline{C}_0 \cap X} \pr_\lambda M$.
If the weight $\lambda \in \overline{C}_0 \cap X$ is dominant then
\[ L(\lambda) \cong \Delta(\lambda) \cong \nabla(\lambda) \cong T(\lambda) ; \]
see Corollary II.5.6 and Section II.E.1 in \cite{Jantzen}.

For $\lambda,\mu \in \overline{C}_0 \cap X$, let $\nu \in X^+$ be the unique dominant weight in the $W_\mathrm{fin}$-orbit of $\mu-\lambda$, and consider the \emph{translation functor}
\[ T_\lambda^\mu = \pr_\mu\big( L(\nu) \otimes - \big) \colon \quad \Rep_\lambda(G) \longrightarrow \Rep_\mu(G) . \]
In the definition of $T_\lambda^\mu$, we can replace the simple $G$-module $L(\nu)$ by $\nabla(\nu)$, $\Delta(\nu)$ or $T(\nu)$ without changing the functor, up to a natural isomorphism (see Remark 1 in Subsection II.7.6 in \cite{Jantzen}), and in particular, translation functors preserve the class of $G$-modules with Weyl filtrations and the class of tilting modules.
For $p$-regular weights $\lambda,\mu \in C_0$, the translation functor $T_\lambda^\mu$ is an equivalence with quasi-inverse $T_\mu^\lambda$ by \cite[Proposition II.7.9]{Jantzen}, and for $w \in W_\mathrm{aff}^+$, we have
\[ T_\lambda^\mu L(w\Cdot\lambda) \cong L(w\Cdot\mu) , \qquad T_\lambda^\mu \Delta(w\Cdot\lambda) \cong \Delta(w\Cdot\mu) , \qquad T_\lambda^\mu T(w\Cdot\lambda) \cong T(w\Cdot\mu) . \]
Furthermore, if $\lambda \in C_0$ and $\mu \in \overline{C}_0$ then we have
\begin{equation} \label{eq:translationWeylsimple}
	T_\lambda^\mu \Delta(w\Cdot\lambda) \cong \begin{cases} \Delta(w\Cdot\mu) & \text{if } w\Cdot\mu \in X^+ , \\ 0 & \text{otherwise} , \end{cases} \qquad T_\lambda^\mu L(w\Cdot\lambda) \cong \begin{cases} L(w\Cdot\mu) & \text{if } w\Cdot\mu \in \widehat{ w \Cdot C_0 } , \\ 0 & \text{otherwise} , \end{cases}
\end{equation} 
by Subsections II.7.11 and II.7.15 in \cite{Jantzen}.

\section{Alcove geometry and complete reducibility}
\label{sec:alcovegeometryCR}

Our goal in this section is to establish a condition on a pair of weights $\lambda,\mu \in X^+$ that guarantees the complete reducibility of the tensor product $L(\lambda) \otimes L(\mu)$.
Loosely speaking, we require that the weight $\mu$ is sufficiently small with respect to the position of $\lambda$ within the unique alcove whose upper closure contains $\lambda$.
More specifically, we make the following definition:

\begin{Definition}
	Let $\lambda,\mu \in X^+$ and let $C$ be the unique alcove whose upper closure contains $\lambda$.
	We say that $\mu$ is \emph{reflection small} with respect to $\lambda$ if
	\[ \lambda + w(\mu) \leq s \Cdot (\lambda + w(\mu)) \]
	for all $w \in W_\mathrm{fin}$ and all reflections $s$ in walls of $C$ such that $\lambda \leq s \Cdot \lambda$.
\end{Definition}

Some particular cases where this condition is made more explicit can be found in Examples \ref{ex:A2reflectionsmall} and \ref{ex:B2reflectionsmall} below.
We will show in Theorem \ref{thm:reflectionsmalltensorproduct} that for all pairs of weights $\lambda,\mu \in X^+$ such that $\mu$ is reflection small with respect to $\lambda$, the tensor product $L(\lambda) \otimes L(\mu)$ is completely reducible.
Furthermore, we will derive an explicit formula for the composition multiplicities in the tensor product in terms of weight space dimensions in $L(\mu)$.
We first need to establish some additional facts about reflection smallness.

\begin{Remark} \label{rem:reflectionsmallpregular}
	Let $\lambda,\mu \in X^+$ such that $\mu$ is reflection small with respect to $\lambda$.
	We claim that either $\lambda$ is $p$-regular or $\mu = 0$.
	Indeed, suppose that $\mu \neq 0$ and let $C$ be the unique alcove whose upper closure contains $\lambda$.
	If $\lambda$ is $p$-singular then there is a wall $H = H_{\beta,r}$ of $C$ such that $\lambda \in H$, and for the corresponding reflection $s = s_{\beta,r}$, we have $s\Cdot\lambda = \lambda$.
	We can further choose an element $w \in W_\mathrm{fin}$ such that $w^{-1}(\beta) \in \{ \alpha_\mathrm{h} , \alpha_\mathrm{hs} \}$ is either the highest root or the highest short root in $\Phi$, and it follows that
	\[ \big( w(\mu) , \beta^\vee \big) = \big( \mu , w^{-1}(\beta^\vee) \big) > 0 . \]
	This implies that $\big( \lambda+w(\mu)+\rho , \beta^\vee \big)>pr$ and $s\Cdot \big( \lambda+w(\mu) \big) < \lambda+w(\mu)$, contradicting the assumption that $\mu$ is reflection small with respect to $\lambda$.
	In particular, the existence of a pair of weights $\lambda,\mu \in X^+$ such that $\mu$ is reflection small with respect to $\lambda$ and $\mu \neq 0$ implies that $p \geq h$.
\end{Remark}

Recall that for $\mu \in X^+$, we write $\Lambda(\mu) = \mathrm{conv}(W_\mathrm{fin} \mu) \cap ( \mu + \Z\Phi )$ for the set of weights of $\Delta(\mu)$.

\begin{Lemma} \label{lem:reflectionsmallconditionallweights}
	Let $\lambda,\mu \in X^+$ such that $\mu$ is reflection small with respect to $\lambda$, and let $C$ be the unique alcove whose upper closure contains $\lambda$.
	For all $\nu \in \Lambda(\mu)$ and every reflection $s$ in a wall of $C$ such that $\lambda \leq s \Cdot \lambda$, we have $\lambda + \nu \leq s \Cdot (\lambda + \nu)$.
\end{Lemma}
\begin{proof}
	Let $\alpha \in \Phi^+$ and $r \in \Z$ such that $s = s_{\alpha,r}$, and observe that $( \lambda + w(\mu) + \rho , \alpha^\vee ) \leq pr$
	for all $w \in W_\mathrm{fin}$ because $\mu$ is reflection small with respect to $\lambda$.
	As $\nu$ belongs to the convex hull of $W_\mathrm{fin} \mu$, we have $(\nu,\alpha^\vee) \leq ( w(\mu) , \alpha^\vee )$ for some $w \in W_\mathrm{fin}$, and it follows that
	\[ ( \lambda + \nu + \rho , \alpha^\vee ) \leq ( \lambda + w(\mu) + \rho , \alpha^\vee ) \leq pr , \]
	whence $\lambda+\nu \leq s \Cdot (\lambda+\nu)$, as required.
\end{proof}

Motivated by Lemma \ref{lem:reflectionsmallconditionallweights}, we next prove a general result on sets of the form $\lambda + \Lambda(\mu)$ and the action of $W_\mathrm{aff}$, for $\lambda,\mu \in X^+$.
We will apply this to the case where $\mu$ is reflection small with respect to $\lambda$ in Corollaries \ref{cor:reflectionsmallsequence} and \ref{cor:reflectionsmallorbit} below.

\begin{Proposition} \label{prop:shiftedweightsystemsequence}
	Let $\lambda,\mu \in X^+$ and let $C \subseteq X_\R$ be an alcove with $\lambda \in \overline{C}$.
	For all $\nu \in \Lambda(\mu)$, there is a sequence of weights $\nu_1,\ldots,\nu_m \in \Lambda(\mu)$ and reflections $s_1,\ldots,s_m$ in walls of $C$ such that
	\[ s_i s_{i-1} \cdots s_1\Cdot (\lambda+\nu) = \lambda + \nu_i \]
	for $1 \leq i \leq m$ and $\lambda + \nu_m \in \overline{C}$.
\end{Proposition}
\begin{proof}
	If $\lambda + \nu \in \overline{C}$ then there is nothing to show, so assume that $\lambda + \nu \notin \overline{C}$, and let $C^\prime \subseteq X_\R$ be an alcove with $\lambda+\nu \in \overline{C^\prime}$, such that the number of reflection hyperplanes separating $C$ and $C^\prime$ is minimal.
	Let $H = H_{\alpha,r}$ be a wall of $C$ that separates $C^\prime$ from $C$, and define $s_1 = s_{\alpha,r}$ and $\nu_1 = s_1 \Cdot (\lambda+\nu) - \lambda$.
	As $H$ separates $C$ and $C^\prime$, we have either
	\[ (\lambda+\rho,\alpha^\vee) \leq pr < (\lambda+\nu+\rho,\alpha^\vee) \qquad \text{or} \qquad (\lambda+\nu+\rho,\alpha^\vee) < pr \leq (\lambda+\rho,\alpha^\vee) , \]
	according to whether $H$ belongs to the upper closure or to the lower closure of $C$, and the weight
	\[ \nu_1 = s_1 \Cdot (\lambda+\nu) - \lambda = \nu + \big( pr - (\lambda+\nu+\rho,\alpha^\vee) \big) \cdot \alpha = s_\alpha(\nu) + \big( pr - (\lambda+\rho,\alpha^\vee) \big) \cdot \alpha \]
	lies on the $\alpha$-string between $\nu$ and $s_\alpha(\nu)$ because we have either
	\[ pr - (\lambda+\nu+\rho,\alpha^\vee) > 0 \geq pr - (\lambda+\rho,\alpha^\vee) \qquad \text{or} \qquad pr - (\lambda+\nu+\rho,\alpha^\vee) < 0 \leq pr - (\lambda+\rho,\alpha^\vee) . \]
	This implies that $\nu_1 \in \Lambda(\mu)$ and $s_1 \Cdot (\lambda+\nu) = \lambda + \nu_1$, as required.
	Now we have $\lambda+\nu_1 \in s_1 \Cdot \overline{C^\prime}$, and the alcoves $C$ and $s_1\Cdot C^\prime$ are separated by one reflection hyperplane less than $C$ and $C^\prime$ (since $H$ is a wall of $C$ that separates $C$ and $C^\prime$), so the claim follows by induction.
\end{proof}

\begin{Corollary} \label{cor:shiftedweightsystemorbit}
	Let $\lambda,\mu \in X^+$ and let $C \subseteq X_\R$ be an alcove with $\lambda \in \overline{C}$.
	For all $\nu \in \Lambda(\mu)$, there is an element $x \in W_\mathrm{aff}$ such that $x\Cdot(\lambda+\nu) \in \overline{C}$ and $x\Cdot(\lambda+\nu)-\lambda \in \Lambda(\mu)$.
\end{Corollary}
\begin{proof}
	In the notation of Proposition \ref{prop:shiftedweightsystemsequence}, we set $x = s_m s_{m-1} \cdots s_1$, so that $x \Cdot ( \lambda  + \nu ) = \lambda+\nu_m \in \overline{C}$, where $\nu_m \in \Lambda(\mu)$, as required.
\end{proof}

\begin{Corollary} \label{cor:reflectionsmallsequence}
	Let $\lambda,\mu \in X^+$ such that $\mu$ is reflection small with respect to $\lambda$, and let $C$ be the unique alcove whose upper closure contains $\lambda$.
	For every weight $\nu \in \Lambda(\mu)$, there is a sequence of weights $\nu_1,\ldots,\nu_m \in \Lambda(\mu)$ and reflections $s_1,\ldots,s_m$ in walls of $C$ with $s_i \Cdot \lambda < \lambda$ for $1 \leq i \leq m$, such that $s_i s_{i-1} \cdots s_1\Cdot (\lambda+\nu) = \lambda + \nu_i$ for $1 \leq i \leq m$ and $\lambda + \nu_m \in \overline{C}$.
\end{Corollary}
\begin{proof}
	By Proposition \ref{prop:shiftedweightsystemsequence}, there is a sequence of weights $\nu_1,\ldots,\nu_m \in \Lambda(\mu)$ and reflections $s_1,\ldots,s_m$ in walls of $C$ such that $s_i s_{i-1} \cdots s_1\Cdot (\lambda+\nu) = \lambda + \nu_i$ for $1 \leq i \leq m$ and $\lambda + \nu_m \in \overline{C}$.
	By the proof of Proposition \ref{prop:shiftedweightsystemsequence}, we may further assume that we have either
	\[ (\lambda+\rho,\beta_i^\vee) \leq pr_i < (\lambda+\nu_i+\rho,\beta_i^\vee) \qquad \text{or} \qquad (\lambda+\nu_i+\rho,\beta_i^\vee) < pr_i \leq (\lambda+\rho,\beta_i^\vee) \]
	for $1 \leq i \leq m$, where $s_i = s_{\beta_i,r_i}$ with $\beta_i \in \Phi^+$ and $r_i \in \Z$, according to whether $H_{\beta_i,r_i}$ belongs to the upper closure or to the lower closure of $C$.
	Since $\lambda$ belongs to the upper closure of $C$, we can strengthen the second chain of inequalities to $(\lambda+\nu_i+\rho,\beta_i^\vee) < pr_i < (\lambda+\rho,\beta_i^\vee)$.
	The first chain of inequalities contradicts Lemma \ref{lem:reflectionsmallconditionallweights} because $\mu$ is reflection small with respect to $\lambda$, so we conclude that $(\lambda+\rho,\beta_i^\vee) > p r_i$ and $s_i \Cdot \lambda < \lambda$ for $1 \leq i \leq m$.
\end{proof}

For an alcove $C \subseteq X_\R$, let us write $S_C$ for the set of reflections in the walls of $C$.
Note that we have $S_\mathrm{aff} = S_{C_0}$, and if $x \in W_\mathrm{aff}$ such that $C = x \Cdot C_0$ then $S_C = x S_\mathrm{aff} x^{-1}$.

\begin{Corollary} \label{cor:reflectionsmallorbit}
	Let $\lambda,\mu \in X^+$ such that $\mu$ is reflection small with respect to $\lambda$ and let $C$ be the unique alcove whose upper closure contains $\lambda$.
	For all $\nu \in \Lambda(\mu)$, there is an element $x \in \langle s \in S_C \mid s \Cdot C \uparrow  C \rangle$ such that $x \Cdot (\lambda+\nu) \in \overline{C}$ and $x \Cdot ( \lambda + \nu ) - \lambda  \in \Lambda(\mu)$.
\end{Corollary}
\begin{proof}
	Since $\lambda$ belongs to the upper closure of $C$, we have $s \Cdot \lambda < \lambda$ if and only if $s \Cdot C \uparrow C$, for every reflection $s$ in a wall of $C$.
	The claim follows from Corollary \ref{cor:reflectionsmallsequence} with $x = s_m s_{m-1} \cdots s_1$.
\end{proof}

In view of Corollaries \ref{cor:reflectionsmallsequence} and \ref{cor:reflectionsmallorbit}, the subgroups $W_C$ of $W_\mathrm{aff}$ in the following definition control the set of weights $\lambda + \Lambda(\mu)$ if $\mu \in X^+$ is reflection small with respect to $\lambda \in X^+$.

\begin{Definition}
	For an alcove $C \subseteq X_\R$, let $W_C = \langle s \in S_C \mid s \Cdot C \uparrow C \rangle$.
\end{Definition}

\begin{Remark} \label{rem:stabilizerupperclosure}
	Observe that for an alcove $C \subseteq X_\R$, the subgroup $W_C$ of $W_\mathrm{aff}$ is generated by a conjugate of a proper subset $S' \subsetneq S_\mathrm{aff}$.
	It is straightforward to check that the intersection of the reflection hyperplanes $H_s$ with $s \in S'$ is non-empty, and so $W_C$ is finite by Proposition 4 in \cite[Section V.3.6]{Bourbaki}.
	Furthermore,  if $\lambda \in X^+$ belongs to the upper closure of $C$ then $s \Cdot \lambda < \lambda$ for all $s \in S_C$ with $s \Cdot C \uparrow C$, hence the stabilizer of $\lambda$ in $W_C$ is trivial (cf.\ Section II.6.3 in \cite{Jantzen}).
\end{Remark}

In preparation for proving our complete reducibility criterion (Theorem \ref{thm:reflectionsmalltensorproduct}), we will establish below some identities for characters of simple $G$-modules that are ``shifted" by an element of $W_C$ (see Proposition \ref{prop:charactersimplechishifted}).
We will often assume that $p \geq h$, the Coxeter number of $G$; this is no serious restriction for our application to reflection smallness in view of Remark \ref{rem:reflectionsmallpregular}.

\begin{Proposition} \label{prop:charactersimplecoefficientsnegative}
	Suppose that $p \geq h$ and let $w \in W_\mathrm{aff}^+$ and $s \in S_\mathrm{aff}$ such that $ws \Cdot C_0 \uparrow w \Cdot C_0$.
	Further define integers $a_x \in \Z$, for $x \in W_\mathrm{aff}^+$, via
	\[ \ch L(w\Cdot0) = \sum_{x \in W_\mathrm{aff}^+} a_x \cdot \chi(x\Cdot0) . \]
	Then for all $x \in W_\mathrm{aff}^+$ with $xs \in W_\mathrm{aff}^+$, we have $a_{xs} = - a_x$.
\end{Proposition}
\begin{proof}
	First observe that by Section II.6.21 in \cite{Jantzen}, we have
	\[ a_x = \sum_{i \geq 0} (-1)^i \cdot \dim \Ext_G^i\big( L(w\Cdot0) , \nabla(x\Cdot0) \big) , \qquad a_{xs} = \sum_{i \geq 0} (-1)^i \cdot \dim \Ext_G^i\big( L(w\Cdot0) , \nabla(xs\Cdot0) \big) . \]
	By Section II.6.3 in \cite{Jantzen}, there is a weight $\mu \in \overline{C}_0 \cap X$ such that $\Stab_{W_\mathrm{aff}}(\mu) = \{ e , s \}$, and since we have $ws \Cdot C_0 \uparrow w \Cdot C_0$, the weight $w \Cdot \mu$ does not belong to the upper closure of $w \Cdot C_0$.
	Now \eqref{eq:translationWeylsimple} implies that $T_0^\mu L(w\Cdot0) = 0$.
	We may assume that $x \Cdot 0 < xs \Cdot 0$ (possibly after replacing $x$ by $xs$), and Proposition II.7.19(b) in \cite{Jantzen} yields
	\[ \Ext_G^i\big( L(w\Cdot0) , \nabla(x\Cdot0) \big) \cong \Ext_G^{i-1}\big( L(w\Cdot0) , \nabla(xs\Cdot0) \big) \]
	for all $i>0$.
	The proof of Proposition II.7.19 in \cite{Jantzen} also implies that $\Hom_G\big( L(w\Cdot0) , \nabla(x\Cdot0) \big) = 0$, and we conclude that $a_{xs} = -a_x$, as required.
\end{proof}

In order to prove the next results about characters, we will need the following elementary lemma about the set $W_\mathrm{aff}^+$ of minimal length $W_\mathrm{fin}$-coset representatives in $W_\mathrm{aff}$.

\begin{Lemma} \label{lem:cosetrepresentativespermutation}
	Fix an element $y \in W_\mathrm{aff}$.
	For all $x \in W_\mathrm{aff}^+$, there is a unique element $w = w(x,y) \in W_\mathrm{fin}$ such that $wxy \in W_\mathrm{aff}^+$.
	Furthermore, the map $x \mapsto w(x,y)  x y$ is a permutation of $W_\mathrm{aff}^+$.
\end{Lemma}
\begin{proof}
	By Proposition 2.4.4 in \cite{BjoernerBrentiCoxeter}, the multiplication in $W_\mathrm{aff}$ induces a bijection
	\begin{equation} \label{eq:multiplicationWfinWaff+}
	W_\mathrm{fin} \times W_\mathrm{aff}^+ \xrightarrow{~1:1~} W_\mathrm{aff} , \qquad (w,x) \longmapsto wx .
	\end{equation}
	Thus, for all $x \in W_\mathrm{aff}$, there is a unique pair of elements $u \in W_\mathrm{fin}$ and $z \in W_\mathrm{aff}^+$ with $u z = x y$, and the first claim follows with $w = u^{-1}$.
	The second claim is immediate from the fact that the multiplication map \eqref{eq:multiplicationWfinWaff+} is a bijection.
\end{proof}

Recall that $\{ \chi(\lambda) \mid \lambda \in X^+ \}$ is a basis of $\Z[X]^{W_\mathrm{fin}}$.
For $\lambda \in \overline{C}_0 \cap X$, we write $\Z[X]^{W_\mathrm{fin}}_\lambda$ for the $\Z$-submodule of $\Z[X]^{W_\mathrm{fin}}$ that is spanned by $\{ \chi(w\Cdot\lambda) \mid w \in W_\mathrm{aff} \text{ such that } w \Cdot\lambda \in X^+\}$.
Then the characters of all $G$-modules in the linkage class $\Rep_\lambda(G)$ belong to $\Z[X]^{W_\mathrm{fin}}_\lambda$.

\begin{Lemma} \label{lem:charactertranslationfunctor}
	Suppose that $p \geq h$, and for $\lambda \in \overline{C}_0 \cap X$, consider the $\Z$-linear map
	\[ \mathrm{tr}_\lambda \colon \Z[X]^{W_\mathrm{fin}}_0 \longrightarrow \Z[X]^{W_\mathrm{fin}}_\lambda \qquad \text{with} \quad \chi(x\Cdot0) \mapsto \chi(x\Cdot\lambda) \quad \text{for } x \in W_\mathrm{aff}^+ . \]
	\begin{enumerate}
		\item For every $G$-module $M$ in $\Rep_0(G)$, we have $\mathrm{tr}_\lambda(\ch M) = \ch(T_0^\lambda M)$.
		\item For all $y \in W_\mathrm{aff}$, we have $\mathrm{tr}_\lambda \chi(y\Cdot0) = \chi(y\Cdot\lambda)$.
	\end{enumerate}
\end{Lemma}
\begin{proof}
	The first claim follows from Proposition II.7.8 in \cite{Jantzen}.
	For $y \in W_\mathrm{aff}$, there is a unique element $w \in W_\mathrm{fin}$ such that $wy \in W_\mathrm{aff}^+$, and so $wy \Cdot 0 \in X^+$ and $\chi(y\Cdot0) = (-1)^{\ell(w)} \cdot \chi(wy \Cdot0)$.
	We conclude that
	\[ \mathrm{tr}_\lambda \chi(y\Cdot0) = (-1)^{\ell(w)} \cdot \mathrm{tr}_\lambda \chi(wy\Cdot0) = (-1)^{\ell(w)} \cdot \chi(wy\Cdot\lambda) = \chi(y\Cdot\lambda) , \]
	as required. 
\end{proof}

The key technical result that we will need in order to prove Theorem \ref{thm:reflectionsmalltensorproduct} is the following proposition.

\begin{Proposition} \label{prop:charactersimplechishifted}
	Suppose that $p \geq h$, let $w \in W_\mathrm{aff}^+$ and write $C = w \Cdot C_0$ and
	\[ \ch L(w\Cdot0) = \sum_{x \in W_\mathrm{aff}^+} a_x \cdot \chi(x\Cdot0) , \]
	with $a_x \in \Z$ for all $x \in W_\mathrm{aff}^+$.
	For all $u \in \langle s \in S_\mathrm{aff} \mid ws \Cdot C_0 \uparrow w \Cdot C_0 \rangle$ and $\lambda \in \overline{C}_0 \cap X$, we have
	\[ \sum_{x \in W_\mathrm{aff}^+} a_x \cdot \chi(xu\Cdot\lambda) = \begin{cases} (-1)^{\ell(u)} \cdot \ch L(w\Cdot\lambda) & \text{if } w \Cdot \lambda \in \widehat{C} \\ 0 & \text{otherwise} . \end{cases} \]
\end{Proposition}
\begin{proof}
	We first prove the claim for $\lambda = 0$.
	Using the notation from Lemma \ref{lem:cosetrepresentativespermutation}, we can write
	\[ \ch L(w\Cdot0) = \sum_{x \in W_\mathrm{aff}^+} a_x \cdot \chi(x\Cdot0) = \sum_{x \in W_\mathrm{aff}^+} a_{w(x,u) x u} \cdot \chi\big( w(x,u) x u \Cdot 0  \big) , \]
	and we also have
	\[ \sum_{x \in W_\mathrm{aff}^+} a_x \cdot \chi(xu\Cdot0) = \sum_{x \in W_\mathrm{aff}^+} (-1)^{\ell(w(x,u))} \cdot a_x \cdot \chi\big( w(x,u) x u \Cdot 0  \big) . \]
	Therefore, the claim will follow (for $\lambda = 0$) if we prove the equality
	\begin{equation} \label{eq:coefficientsequaluptosign}
		a_{w(x,u) x u} = (-1)^{\ell(u) + \ell(w(x,u))} \cdot a_x
	\end{equation}
	for all $u \in \langle s \in S_\mathrm{aff} \mid ws \Cdot C_0 \uparrow w \Cdot C_0 \rangle$ and $x \in W_\mathrm{aff}^+$.
	
	We prove \eqref{eq:coefficientsequaluptosign} by induction on $\ell(u)$.
	If $\ell(u) = 0$ then $u = e$ and $w(x,u) = e$, and the equation is trivially satisfied.
	Now suppose that we have $u \in \langle s \in S_\mathrm{aff} \mid ws \Cdot C_0 \uparrow w \Cdot C_0 \rangle$ such that \eqref{eq:coefficientsequaluptosign} holds, and let $s \in S_\mathrm{aff}$ such that $ws \Cdot C_0 \uparrow w \Cdot C_0$ and $\ell(u)<\ell(us)$.
	If $w(x,u) x u s \in W_\mathrm{aff}^+$ then $w(x,us) = w(x,u)$, and as $a_{w(x,u)xus} = - a_{w(x,u)xu}$ by Proposition \ref{prop:charactersimplecoefficientsnegative}, it follows that
	\[ a_{w(x,us)xus} = a_{w(x,u)xus} = - a_{w(x,u)xu} = - (-1)^{\ell(u) + \ell(w(x,u))} \cdot a_x = (-1)^{\ell(us) + \ell(w(x,us))} \cdot a_x . \]
	Now suppose that $w(x,u)xus \notin W_\mathrm{aff}^+$.
	As the dominant alcove $C = w(x,u)xu \Cdot C_0$ and the non-dominant alcove $C^\prime = w(x,u)xus \Cdot C_0$ are adjacent, the reflection $\tilde s \coloneqq w(x,u)xus(w(x,u)xu)^{-1}$
	belongs to $W_\mathrm{fin}$, and as $\tilde s w(x,u) x u s = w(x,u) x u \in W_\mathrm{aff}^+$, we conclude that $w(x,us) = \tilde s w(x,u)$.
	In particular, we have $(-1)^{\ell(w(x,us))} = - (-1)^{\ell(w(x,u))}$ and therefore
	\[ a_{w(x,us)xus} = a_{w(x,u)xu} = (-1)^{\ell(u) + \ell(w(x,u))} \cdot a_x = (-1)^{\ell(us) + \ell(w(x,us))} \cdot a_x , \]
	as required.
	This completes the proof of \eqref{eq:coefficientsequaluptosign}, and so the claim follows for $\lambda = 0$.
	
	For an arbitrary weight $\lambda \in \overline{C}_0 \cap X$, let $\mathrm{tr}_\lambda \colon \Z[X]^{W_\mathrm{fin}}_0 \to \Z[X]^{W_\mathrm{fin}}_\lambda$ be the $\Z$-linear map from Lemma~\ref{lem:charactertranslationfunctor}, with $\chi(x\Cdot0) \mapsto \chi(x\Cdot\lambda)$ for all $x \in W_\mathrm{aff}^+$.
	Using Lemma \ref{lem:charactertranslationfunctor} and the validity of the claim for $\lambda = 0$, we obtain
	\[ \sum\nolimits_x a_x \cdot \chi(xu\Cdot\lambda) = \mathrm{tr}_\lambda\Big( \sum\nolimits_x a_x \cdot \chi(xu\Cdot0) \Big) = (-1)^{\ell(u)} \cdot \mathrm{tr}_\lambda\big( \ch L(w\Cdot0) \big) = (-1)^{\ell(u)} \cdot \ch T_0^\lambda L(w\Cdot0) . \]
	By \eqref{eq:translationWeylsimple}, we have
	\[ T_0^\lambda L(w\Cdot0) \cong \begin{cases} L(w\Cdot\lambda) & \text{if } w \Cdot \lambda \in \widehat{C} , \\ 0 & \text{otherwise} , \end{cases} \]
	 and the claim is immediate from the two last equations.
\end{proof}

Now we are ready to prove the main theorem of this section.
Recall that for an alcove $C \subseteq X_\R$, we write $S_C$ for the set of reflections in the walls of $C$ and $W_C = \langle s \in S_C \mid s \Cdot C \uparrow C \rangle$.

\begin{Theorem} \label{thm:reflectionsmalltensorproduct}
	Let $\lambda,\mu \in X^+$ such that $\mu$ is reflection small with respect to $\lambda$.
	\begin{enumerate}
	\item The tensor product $L(\lambda) \otimes L(\mu)$ is completely reducible.
	\item Let $C$ be the unique alcove whose upper closure contains $\lambda$ and let $\nu \in X^+$ such that $L(\nu)$ is a composition functor of $L(\lambda) \otimes L(\mu)$.
	Then $\nu$ belong to the upper closure of $C$.
	\item Let $\nu \in X^+$ be a weight in the upper closure of $C$.
	Then the composition multiplicity of $L(\nu)$ in $L(\lambda) \otimes L(\mu)$ is given by
	\[ \big[ L(\lambda) \otimes L(\mu) : L( \nu ) \big] = \sum_{u \in W_C} (-1)^{\ell(u)} \cdot \dim L(\mu)_{u\Cdot\nu-\lambda} . \]
	\end{enumerate}
\end{Theorem}
\begin{proof}
	If $\mu = 0$ then all three claims are trivially satisfied, so we assume that $\mu \neq 0$, and Remark \ref{rem:reflectionsmallpregular} yields that $p \geq h$ and $\lambda$ is $p$-regular.
	Let $\lambda^\prime \in C_0$ and $w \in W_\mathrm{aff}^+$ such that $\lambda = w \Cdot \lambda^\prime$, and write
	\[ \ch L(w\Cdot0) = \sum_{x \in W_\mathrm{aff}^+} a_x \cdot \chi(x\Cdot0) . \]
	with $a_x \in \Z$ for $x \in W_\mathrm{aff}^+$.
	By Lemma \ref{lem:charactertranslationfunctor}, we have
	\[ \ch L(\lambda) = \ch T_0^\lambda L(w\Cdot0) = \sum_{x \in W_\mathrm{aff}^+} a_x \cdot \chi(x\Cdot\lambda^\prime) , \]
	and the last displayed equation in Section II.7.5 in \cite{Jantzen} yields
	\[ \ch\big( L(\lambda) \otimes L(\mu) \big)  = \sum_{\delta \in \Lambda(\mu)} \dim L(\mu)_\delta \cdot \sum_{x \in W_\mathrm{aff}^+} a_x \cdot \chi\big( x\Cdot (\lambda^\prime + \delta) \big) . \]
	Fix a weight $\delta \in \Lambda(\mu)$ and write $w = t_\gamma \tilde w$ with $\gamma \in \Z\Phi$ and $\tilde w \in W_\mathrm{fin}$.
	By Corollary \ref{cor:reflectionsmallorbit} (applied to the weights $\lambda$ and $\tilde w \delta \in \Lambda(\delta)$), we can choose an element $u_\delta \in W_C$ such that
	\[ \nu_\delta \coloneqq u_\delta^{-1} \Cdot (\lambda + \tilde w \delta) \in \overline{C} = w \Cdot \overline{C}_0 , \qquad \mu_\delta \coloneqq \nu_\delta - \lambda \in \Lambda(\mu) , \]
	and it follows that $\lambda + \tilde w \delta = u_\delta \Cdot \nu_\delta = u_\delta \Cdot (\lambda + \mu_\delta)$ and
	\[ \lambda^\prime + \delta = w^{-1} \Cdot ( \lambda + \tilde w \delta ) = w^{-1} u_\delta \Cdot \nu_\delta = w^{-1} u_\delta w \Cdot \nu_\delta^\prime , \qquad \nu_\delta^\prime \coloneqq w^{-1} \Cdot \nu_\delta \in \overline{C}_0 . \]
	Observe that $w^{-1} u_\delta w \in w^{-1} W_C w = \langle s \in W_\mathrm{aff} \mid ws \Cdot C_0 \uparrow w \Cdot C_0 \rangle$, so Proposition \ref{prop:charactersimplechishifted} implies that
	\[ \sum_{x \in W_\mathrm{aff}^+} a_x \cdot \chi( x\Cdot (\lambda^\prime + \delta) ) = \sum_{x \in W_\mathrm{aff}^+} a_x \cdot \chi( x w^{-1} u_\delta w \Cdot \nu_\delta^\prime ) = \begin{cases} (-1)^{\ell(u_\delta)} \cdot \ch L(\nu_\delta) & \text{if } \nu_\delta \in \widehat{C} , \\ 0 & \text{otherwise} , \end{cases} \]
	and we obtain
	\begin{align*}
	\ch\big( L(\lambda) \otimes L(\mu) \big)
	& = \sum_{\delta \in \Lambda(\mu)} \dim L(\mu)_\delta \cdot \sum_{x \in W_\mathrm{aff}^+} a_x \cdot \chi( x\Cdot (\lambda^\prime + \delta) ) \\
	& = \sum_{\delta \in \Lambda(\mu) \, : \, \nu_\delta \in \widehat C} \dim L(\mu)_\delta \cdot (-1)^{\ell(u_\delta)} \cdot \ch L(\nu_\delta) .
	\end{align*}	
	Further note that $\tilde w(\delta) = u_\delta \Cdot \nu_\delta - \lambda$, so
	\[ \dim L(\mu)_\delta = \dim L(\mu)_{\tilde w(\delta)} = \dim L(\mu)_{u_\delta \Cdot \nu_\delta - \lambda} , \]
	and recall from Remark \ref{rem:stabilizerupperclosure} that the $W_C$-orbit of $\nu_\delta$ is regular if $\nu_\delta \in \widehat C$.	
	Thus, for any given weight $\nu \in \widehat{C}$, we have
	\[ \sum_{\delta \in \Lambda(\mu) \, : \, \nu_\delta = \nu} (-1)^{\ell(u_\delta)} \cdot \dim L(\mu)_\delta = \sum_{u \in W_C} (-1)^{\ell(u)} \cdot \dim L(\mu)_{u\Cdot\nu-\lambda} , \]
	and we conclude that
	\begin{align*}
	\ch\big( L(\lambda) \otimes L(\mu) \big)
	& = \sum_{\delta \in \Lambda(\mu) \, : \, \nu_\delta \in \widehat C} \dim L(\mu)_\delta \cdot (-1)^{\ell(u_\delta)} \cdot \ch L(\nu_\delta) \\
	& = \sum_{\nu \in \widehat{C}} \Big( \sum_{u \in W_C} \dim L(\mu)_{u\Cdot\nu-\lambda} \cdot (-1)^{\ell(u)} \Big) \cdot \ch L(\nu) .
	\end{align*}	
	Parts (2) and (3) of the theorem are immediate from this character formula.
	Part (1) follows from part (2) together with the observation that there cannot be any non-split extensions between simple $G$-modules with highest weights in $\overline{C}$, by the linkage principle,
	and the fact that simple $G$-modules have no non-trivial self-extensions (see Section II.2.12 in \cite{Jantzen}).
\end{proof}

\section{Further criteria for complete reducibility and multiplicity freeness}
\label{sec:CRMFcriteria}

In this section, we establish further tools and criteria to decide whether a tensor product of two simple $G$-modules is completely reducible or multiplicity free, based on Steinberg's tensor product theorem (Subsection \ref{subsec:prestricted}), the notion of good filtration dimension (Subsection \ref{subsec:GFD}) and the tensor ideal of singular $G$-modules (Subsection \ref{subsec:singular}).
We also discuss how composition multiplicities in tensor products can sometimes be bounded by dimensions of weight spaces (Subsection \ref{subsec:weightspaces}) or by multiplicities arising from tensor product decompositions in characteristic zero (Subsection \ref{subsec:char0}).

\subsection{Reduction to \texorpdfstring{$p$}{p}-restricted weights}
\label{subsec:prestricted}

Recall that we write
\[ X_1 = \{ \lambda \in X^+ \mid (\lambda,\alpha^\vee) < p \text{ for all } \alpha \in \Pi \} \]
for the set of $p$-restricted weights in $X^+$.
Since we assume that $G$ is simply connected, every dominant weight $\lambda \in X^+$ can be written uniquely as $\lambda = \lambda_0 + p \lambda_1 + \cdots + p^m \lambda_m$ with $\lambda_1,\ldots,\lambda_m \in X_1$, and Steinberg's tensor product theorem states that the simple $G$-module $L(\lambda)$ admits a tensor product decomposition
\[ L(\lambda) \cong L(\lambda_0) \otimes L(\lambda_1)^{[1]} \otimes \cdots \otimes L(\lambda_m)^{[m]} , \]
where $M \mapsto M^{[r]}$ denotes the $r$-th Frobenius twist functor on $\Rep(G)$.
See Subsections II.3.16--17 in \cite{Jantzen} for more details.
Using Steinberg's tensor product theorem, the problem of classifying pairs of simple modules whose tensor product is completely reducible can be reduced to pairs of simple modules with $p$-restricted highest weights; see Corollary 9.3 in \cite{GruberCompleteReducibility}.

\begin{Theorem}
	Let $\lambda,\mu \in X^+$ and write $\lambda = \lambda_0 + \cdots + p^m \lambda_m$ and $\mu = \mu_0 + \cdots + p^m \mu_m$, with $\lambda_i,\mu_i \in X_1$ for all $i = 0 , \ldots , m$.
	Then $L(\lambda) \otimes L(\mu)$ is completely reducible if any only if $L(\lambda_i) \otimes L(\mu_i)$ is completely reducible for $i=0,\ldots,m$.
\end{Theorem}

In order to obtain an analogous reduction to $p$-restricted weights for classifying pairs of simple modules whose tensor product is multiplicity free, we will use the following result from Lemma 4.12 in \cite{GruberCompleteReducibility}, which relates multiplicity freeness and complete reducibility.

\begin{Lemma}\label{lem:MFimpliesCR}
	Let $\lambda,\mu\in X^+$.
	If $L(\lambda)\otimes L(\mu)$ is multiplicity free then $L(\lambda)\otimes L(\mu)$ is completely reducible.
\end{Lemma}

\begin{Lemma}
	Let $\lambda,\mu \in X^+$ and write $\lambda = \lambda_0 + \cdots + p^m \lambda_m$ and $\mu = \mu_0 + \cdots + p^m \mu_m$, with $\lambda_i,\mu_i \in X_1$ for all $i = 0 , \ldots , m$.
	Then $L(\lambda) \otimes L(\mu)$ is multiplicity free if any only if $L(\lambda_i) \otimes L(\mu_i)$ is multiplicity free for $i=0,\ldots,m$.
\end{Lemma}
\begin{proof}
	First observe that by Steinberg's tensor product theorem, we have
	\[ L(\lambda) \otimes L(\mu) \cong \big( L(\lambda_0) \otimes L(\mu_0) \big) \otimes \big( L(\lambda_1) \otimes L(\mu_1) \big)^{[1]} \otimes \cdots \otimes \big( L(\lambda_m) \otimes L(\mu_m) \big)^{[m]} . \]
	Thus, if $L(\lambda_i) \otimes L(\mu_i)$ is not multiplicity free for some $i \in \{0,\ldots,m\}$ then neither is $L(\lambda) \otimes L(\mu)$.
	Now suppose that $L(\lambda_i) \otimes L(\mu_i)$ is multiplicity free for $i=0,\ldots,m$.
	Then $L(\lambda_i) \otimes L(\mu_i)$ is completely reducible by Lemma \ref{lem:MFimpliesCR}, and furthermore, all composition factors of $L(\lambda_i) \otimes L(\mu_i)$ have $p$-restricted highest weights by Theorem A in \cite{GruberCompleteReducibility}, for $i=0,\ldots,m$.
	Let us write
	\[ L(\lambda_i) \otimes L(\mu_i) = L(\nu_{i,1}) \oplus L(\nu_{i,2}) \oplus \cdots \oplus L(\nu_{i,m_i}) , \]
	where the weights $\nu_{i,1},\nu_{i,2},\ldots,\nu_{i,m_i} \in X_1$ are $p$-restricted and pairwise distinct.
	Then we have
	\begin{align*}
	L(\lambda) \otimes L(\mu) & \cong \big( L(\lambda_0) \otimes L(\mu_0) \big) \otimes \big( L(\lambda_1) \otimes L(\mu_1) \big)^{[1]} \otimes \cdots \otimes \big( L(\lambda_m) \otimes L(\mu_m) \big)^{[m]} \\
	& \cong \bigotimes_{i=0}^m \big( L(\nu_{i,1}) \oplus L(\nu_{i,2}) \oplus \cdots \oplus L(\nu_{i,m_i}) \big)^{[i]} \\
	& \cong \bigoplus_{\mathbf{j} = (j_0,\ldots,j_m)} L(\nu_{0,j_0}) \otimes L(\nu_{1,j_1})^{[1]} \otimes \cdots \otimes L(\nu_{m,j_m})^{[m]} \\
	& \cong \bigoplus_{\mathbf{j} = (j_0,\ldots,j_m)} L(\nu_{0,j_0} + p\nu_{1,j_1} + \cdots + p^m \nu_{m,j_m}) ,
	\end{align*}
	where $\mathbf{j} = (j_0,\ldots,j_m)$ runs over the tuples with $1 \leq j_i \leq m_i$ for $i=0,\ldots,m$.
	Now it is straightforward to see that the highest weights $\nu_\mathbf{j} = \nu_{0,j_0} + p\nu_{1,j_1} + \cdots + p^m \nu_{m,j_m}$ of the composition factors of $L(\lambda) \otimes L(\mu)$ are pairwise distinct, and so $L(\lambda) \otimes L(\mu)$ is multiplicity free, as required.
\end{proof}

\subsection{Good filtration dimension}
\label{subsec:GFD}

Following \cite[Section 3]{FriedlanderParshallGFD}, we define the \emph{good filtration dimension} of a non-zero $G$-module $M$ as
\[ \gfd\, M = \max\big\{ d>0 \mathrel{\big|} \Ext_G^d( \Delta(\lambda) , M ) \neq 0 \text{ for some  } \lambda \in X^+ \big\} . \]
Note that $\gfd\, M = 0$ if and only if $M$ admits a good filtration, by Proposition II.4.16 in \cite{Jantzen}.
The following lemma shows that the good filtration dimension is well-behaved with respect to tensor products and direct sums of $G$-modules.

\begin{Lemma} \label{lem:GFDtensordirectsum}
	For $G$-modules $M$ and $N$, we have
	\[ \gfd(M \otimes N) \leq \gfd(M) + \gfd(N) \qquad \text{and} \qquad \gfd(M \oplus N) = \max\{ \gfd(M) , \gfd(N) \} . \]
\end{Lemma}
\begin{proof}
	The first claim is proven in part (c) of \cite[Proposition 3.4]{FriedlanderParshallGFD}.
	The second claim is immediate from the definition of the good filtration dimension.
\end{proof}

The good filtration dimension of simple $G$-modules with $p$-regular highest weights was determined by A.\ Parker in \cite[Corollary 4.5]{ParkerGFD}.

\begin{Lemma} \label{lem:GFDsimple}
	For $\lambda \in C_0 \cap X$ and $x \in W_\mathrm{aff}^+$, we have $\gfd \, L(x\Cdot\lambda) = \ell(x)$.
\end{Lemma}

Using the two preceding lemmas, we obtain a new necessary condition for the complete reducibility of tensor products of simple $G$-modules with $p$-regular highest weights.

\begin{Corollary} \label{cor:nonCRcriterionGFD}
	Let $\lambda,\mu,\nu \in C_0 \cap X$ and $x,y,z \in W_\mathrm{aff}^+$, and suppose that $L(z\Cdot\nu)$ is a composition factor of $L(x\Cdot\lambda) \otimes L(y\Cdot\lambda)$.
	If $L(x\Cdot\lambda) \otimes L(y\Cdot\lambda)$ is completely reducible then $\ell(z) \leq \ell(x) + \ell(y)$.
\end{Corollary}
\begin{proof}
	If $L(x\Cdot\lambda) \otimes L(y\Cdot\mu)$ is completely reducible then $L(z\Cdot\nu)$ is a direct summand of $L(x\Cdot\lambda) \otimes L(y\Cdot\mu)$, and using Lemmas \ref{lem:GFDtensordirectsum} and \ref{lem:GFDsimple}, we obtain
	\[ \ell(z) = \gfd \, L(z\Cdot\nu) \leq \gfd\big( L(x\Cdot\lambda) \otimes L(y\Cdot\mu) \big) \leq \gfd \, L(x\Cdot\lambda) + \gfd \, L(y\Cdot\mu) = \ell(x) + \ell(y) , \]
	as claimed.
\end{proof}

\subsection{Singular and regular modules}
\label{subsec:singular}

A tilting module $T = \bigoplus_{\mu \in X^+} T(\mu)^{\oplus m_\mu}$ is called \emph{negligible} if we have $m_\mu = 0$ for all $\mu \in C_0$.
Following Section~3 in \cite{GruberLinkageTranslation}, we call a $G$-module $M$ \emph{singular} if there exists a bounded complex
\[ C = ( \cdots \to T_{-1} \to T_0 \to T_1 \to \cdots ) \]
of negligible tilting modules such that $H^0(C) \cong M$ and $H^i(C) = 0$ for $i \neq 0$, and we say that $M$ is \emph{regular} if $M$ is not singular.
The singular $G$-modules form a thick tensor ideal in $\Rep(G)$, that is, for $G$-modules $M$ and $N$, we have
\begin{enumerate}
	\item if $M$ is singular then $M \otimes N$ is singular;
	\item $M \oplus N$ is singular if and only if both $M$ and $N$ are singular.
\end{enumerate}
Furthermore, for $\lambda \in X^+$, the simple $G$-module $L(\lambda)$ is singular if and only if the weight $\lambda$ is $p$-singular, see Lemma 3.3 in \cite{GruberLinkageTranslation}.
This observation gives rise to the following necessary condition for complete reducibility of tensor products of simple $G$-modules.

\begin{Lemma} \label{lem:nonCRcriterionpregularweight}
	Let $\lambda,\mu,\nu \in X^+$ such that $L(\nu)$ is a composition factor of $L(\lambda) \otimes L(\mu)$.
	If $\lambda$ is $p$-singular and $\nu$ is $p$-regular then $L(\lambda) \otimes L(\mu)$ is not completely reducible.
\end{Lemma}
\begin{proof}
	If $\lambda$ is $p$-singular then $L(\lambda)$ is singular, and as the singular $G$-modules form a tensor ideal, it follows that all indecomposable direct summands of $L(\lambda) \otimes L(\mu)$ are singular.
	In particular, the regular $G$-module $L(\nu)$ cannot be a direct summand of $L(\lambda) \otimes L(\mu)$, and as $L(\nu)$ is a composition factor of $L(\lambda) \otimes L(\mu)$, it follows that $L(\lambda) \otimes L(\mu)$ is not completely reducible.
\end{proof}

For every $G$-module $M$, we can write
\[ M = M_\mathrm{sing} \oplus M_\mathrm{reg} , \]
where $M_\mathrm{sing}$ is the direct sum of all singular indecomposable direct summands of $M$ and $M_\mathrm{reg}$ is the direct sum of all regular indecomposable direct summands of $M$, for a fixed Krull--Schmidt decomposition of $M$.
This direct sum decomposition into a \emph{singular part} and a \emph{regular part} has been used in \cite{GruberLinkageTranslation} to describe how tensor products of $G$-modules interact with the decomposition of $G$ into linkage classes and with translation functors, on the level of regular parts.
In order to state these results, we define the \emph{Verlinde coefficient} $c_{\lambda,\mu}^\nu$, for $\lambda,\mu,\nu \in C_0 \cap X$, as the multiplicity
\[ c_{\lambda,\mu}^\nu = \big[ T(\lambda) \otimes T(\mu) : T(\nu) \big]_\oplus \]
of $T(\nu)$ in a Krull--Schmidt decomposition of $T(\lambda) \otimes T(\mu)$.
We have the following linkage principle and translation principle for tensor products; see Lemma 3.12 and Theorem 3.14 in \cite{GruberLinkageTranslation}.

\begin{Theorem}  \label{thm:translationprincipletensorproducts}
	Let $M$ and $N$ be $G$-modules in the principal block $\Rep_0(G)$ of $G$.
	\begin{enumerate}
	\item The regular part $(M \otimes N)_\mathrm{reg}$ belongs to $\Rep_0(G)$.
	\item For $\lambda,\mu \in C_0 \cap X$, we have
	\[ \big( T_0^\lambda M \otimes T_0^\mu N \big)_\mathrm{reg} \cong \bigoplus_{\nu \in C_0 \cap X} T_0^\nu (M \otimes N)_\mathrm{reg}^{\oplus c_{\lambda,\mu}^\nu} . \]
	\end{enumerate}
\end{Theorem}

The coefficients $c_{\lambda,\mu}^\nu$ are the structure coefficients of the so called \emph{Verlinde category} $\mathrm{Ver}(G)$ of $G$, i.e.\ the quotient of the subcategory $\Tilt(G)$ of tilting $G$-modules by the thick tensor ideal of negligible tilting modules
\cite{GeorgievMathieuFusion,AndersenParadowskiFusionCategories,EtingofOstrikSemisimplification}.
The category $\mathrm{Ver}(G)$ is semisimple, with simple objects $T(\delta)$ for $\delta \in C_0 \cap X$.
We next recall two elementary results that will be useful for computing Verlinde coefficients in concrete examples.
Let us write $W_\mathrm{ext} = X \rtimes W_\mathrm{fin}$ for the extended affine Weyl group and $\Omega = \Stab_{W_\mathrm{ext}}(C_0)$, so that
\[ \Omega \cong W_\mathrm{ext} / W_\mathrm{aff} \cong X / \Z\Phi . \]
The Verlinde coefficients are invariant under the action of $\Omega$ on $C_0$ by Lemma 1.1 in \cite{GruberLinkageTranslation}.

\begin{Lemma} \label{lem:VerlindecoefficientsFundamentalgroup}
	For $\lambda,\mu,\nu \in C_0 \cap X$ and $\omega \in \Omega$, we have
	\[ c_{\omega \Cdot \lambda , \mu}^{\omega\Cdot\nu} = c_{\lambda , \omega \Cdot \mu}^{\omega\Cdot\nu} = c_{\lambda,\mu}^\nu . \]
\end{Lemma}

Recall that we write $w_0$ for the longest element in $W_\mathrm{fin}$.
The Verlinde coefficients also have the following exchange property:

\begin{Lemma} \label{lem:Verlindecoefficientsflipping}
	Let $\lambda,\mu,\nu \in C_0 \cap X$.
	Then we have $c_{\lambda,\mu}^\nu = c_{\nu,-w_0(\mu)}^\lambda$
\end{Lemma}
\begin{proof}
	The dual of $T(\mu)$ is $T(\mu)^* \cong T(-w_0(\mu))$, and as the Verlinde category is semisimple with simple objects $T(\delta)$ for $\delta \in C_0 \cap X$, we have
	\[ c_{\lambda,\mu}^\nu = \dim \Hom_{\mathrm{Ver}(G)}\big( T(\lambda) \otimes T(\mu) , T(\nu) \big) = \dim \Hom_{\mathrm{Ver}(G)}\big( T(\lambda) , T( - w_0(\mu) ) \otimes T(\nu) \big) = c_{\nu,-w_0(\mu)}^\lambda , \]
	as claimed.
\end{proof}

The following lemma relates the Verlinde coefficients to multiplicity freeness of tensor products of simple $G$-modules.

\begin{Lemma} \label{lem:MFboundVerlindecoeff}
	For $\lambda,\mu,\nu \in C_0 \cap X$ and $x \in W_\mathrm{aff}^+$, the multiplicity of $L(x\Cdot\nu)$ in a Krull--Schmidt decomposition of $L(x\Cdot\lambda) \otimes L(\mu)$ is given by
	\[ \big[ L(x\Cdot\lambda) \otimes L(\mu) : L(x\Cdot\nu) \big]_\oplus = c_{\lambda,\mu}^\nu . \]
	In particular, if $L(x\Cdot\lambda) \otimes L(\mu)$ is multiplicity free then $c_{\lambda,\mu}^\nu \leq 1$ for all $\nu \in C_0 \cap X$.
\end{Lemma}
\begin{proof}
	We clearly have $\big( L(x\Cdot0) \otimes L(0) \big)_\mathrm{reg} \cong L(x\Cdot0)$, and so Theorem \ref{thm:translationprincipletensorproducts} yields
	\[ \big( L(x\Cdot\lambda) \otimes L(\mu) \big)_\mathrm{reg} \cong \bigoplus_{\nu \in C_0 \cap X} T_0^\nu \big( L(x\Cdot0) \otimes L(0) \big)_\mathrm{reg}^{\oplus c_{\lambda,\mu}^\nu} \cong \bigoplus_{\nu \in C_0 \cap X} T_0^\nu L(x\Cdot0)^{\oplus c_{\lambda,\mu}^\nu} \cong \bigoplus_{\nu \in C_0 \cap X} L(x\Cdot\nu)^{\oplus c_{\lambda,\mu}^\nu} . \]
	As $L(x\Cdot\nu)$ is regular, the multiplicity of $L(x\Cdot\nu)$ in a Krull--Schmidt decomposition of $L(x\Cdot\lambda) \otimes L(\mu)$ coincides with the multiplicity of $L(x\Cdot\nu)$ in a Krull--Schmidt decomposition of $\big( L(x\Cdot\lambda) \otimes L(\mu) \big)_\mathrm{reg}$, and the claim follows.
\end{proof}

\subsection{Tensor products and weight space dimensions}
\label{subsec:weightspaces}

In this subsection, we prove some results that relate multiplicities in tensor products to dimensions of weight spaces.
This will be useful later on to establish the multiplicity freeness of certain tensor products of simple $G$-modules.

\begin{Lemma}\label{lem:boundformultiplicityintensorproductifCR}
	Let $M$ be a $G$-module and $\lambda\in X^+$ such that $L(\lambda)\otimes M$ is completely reducible.
	For all weights $\mu \in X^+$, we have
	\[ [ L(\lambda) \otimes M : L(\mu) ] \leq \dim M_{\mu-\lambda} . \]
\end{Lemma}
\begin{proof}
	Since $L(\lambda) \otimes M$ is completely reducible, we have
	\begin{align*}
		[ L(\lambda) \otimes M : L(\mu) ] & = \dim \Hom_G\big( L(\mu) , L(\lambda) \otimes M \big) \\
		& \leq \dim \Hom_G\big( \Delta(\mu) , \nabla(\lambda) \otimes M \big) \\
		& \leq \dim M_{\mu-\lambda} ,
	\end{align*}
	where the second inequality follows from Theorem 2.5 in \cite{BrundanKleshchevBranchhing}.
\end{proof}

\begin{Corollary} \label{cor:onedimensionalweightspacesMF}
	Let $\lambda \in X^+$ and let $M$ be a $G$-module such that all weight spaces of $M$ are at most one-dimensional.
	If $L(\lambda) \otimes M$ is completely reducible then $L(\lambda) \otimes M$ is multiplicity free.
\end{Corollary}
\begin{proof}
	This is immediate from Lemma \ref{lem:boundformultiplicityintensorproductifCR}.
\end{proof}

Recall that a weight $\varpi \in X^+$ is called \emph{minuscule} if all weights of $\nabla(\varpi)$ belong to the same $W_\mathrm{fin}$-orbit, and that $L(\varpi) = \nabla(\varpi)$ if $\varpi \in X^+$ is minuscule (see Section II.2.15 in \cite{Jantzen}).
The following two results describe tensor products where one factor is simple with minuscule highest weight.

\begin{Lemma} \label{lem:minusculetensorproducttranslation}
	Let $\lambda \in \overline{C}_0 \cap X$ and let $\varpi \in X^+$ be minuscule.
	Then for every $G$-module $M$ in the linkage class $\Rep_\lambda(G)$, we have
	\[ L(\varpi) \otimes M \cong \bigoplus_{\varpi^\prime} T_\lambda^{\lambda+\varpi^\prime} M , \]
	where $\varpi^\prime$ runs over the weights in $W_\mathrm{fin} \varpi$ such that $\lambda+\varpi^\prime \in \overline{C}_0$.
\end{Lemma}
\begin{proof}
	By the linkage principle, we have
	\[ L(\varpi) \otimes M \cong \bigoplus_{\nu \in \overline{C}_0 \cap X} \pr_\nu\big( L(\varpi) \otimes M \big) . \]
	For a weight $\nu \in \overline{C}_0 \cap X$ such that $\pr_\nu\big( L(\varpi) \otimes M \big) \neq 0$, there exists $x \in W_\mathrm{aff}$ such that $x \Cdot\nu-\lambda$ is a weight of $L(\varpi)$ by Lemma II.7.5 in \cite{Jantzen}, and by Corollary \ref{cor:shiftedweightsystemorbit}, there further exists $y \in W_\mathrm{aff}$ such that $yx\Cdot\nu-\lambda$ is a weight of $L(\varpi)$ and $yx\Cdot\nu \in \overline{C}_0$.
	This forces $yx\Cdot\nu = \nu$ because $\overline{C}_0$ is a fundamental domain for the $p$-dilated dot action of $W_\mathrm{aff}$ on $X_\R$, and we conclude that $\pr_\nu\big( L(\varpi) \otimes M \big) \neq 0$ only if $\nu-\lambda$ is a weight of $L(\varpi)$.
	For a weight $\varpi^\prime$ of $L(\varpi)$ such that $\lambda + \varpi^\prime \in \overline{C}_0$, we have $\pr_\nu\big( L(\varpi) \otimes M \big) \cong T_\lambda^{\lambda+\varpi^\prime} M$ because $\varpi$ is the unique dominant weight in the $W_\mathrm{fin}$-orbit of $\varpi^\prime$, and the claim follows.
\end{proof}

\begin{Corollary} \label{cor:minusculetensorproductMFCR}
	Let $\lambda \in X^+$ be a $p$-regular weight and let $C$ be the unique alcove containing $\lambda$.
	For a minuscule weight $\varpi \in X^+$, we have
	\[ L(\lambda) \otimes L(\varpi) \cong \bigoplus_{\varpi^\prime} L(\lambda+\varpi^\prime) , \]
	where $\varpi^\prime$ runs over the weights in $W_\mathrm{fin} \varpi$ such that $\lambda +\varpi^\prime \in \widehat{C}$.
\end{Corollary}
\begin{proof}
	Let $x \in W_\mathrm{aff}$ such that $x \Cdot C_0 = C$ and set $\lambda_0 = x^{-1} \Cdot \lambda \in C_0$.
	By Lemma \ref{lem:minusculetensorproducttranslation}, we have
	\[ L(\lambda) \otimes L(\varpi) = L(x\Cdot\lambda_0) \otimes L(\varpi) \cong \bigoplus_{\varpi^\prime} T_{\lambda_0}^{\lambda_0+\varpi^\prime} L(x\Cdot\lambda_0) , \]
	where $\varpi^\prime$ runs over the weights in $W_\mathrm{fin} \varpi$ such that $\lambda_0+\varpi^\prime \in \overline{C}_0$.
	Further let $w \in W_\mathrm{fin}$ and $\gamma \in \Z\Phi$ such that $x = w t_\gamma$, and observe that $x \Cdot (\lambda_0+\varpi^\prime) = x\Cdot\lambda_0 + w(\varpi^\prime) = \lambda + w(\varpi^\prime)$.
	We have
	\[ T_{\lambda_0}^{\lambda_0+\varpi^\prime} L(x\Cdot\lambda_0) \cong \begin{cases} L( \lambda + w(\varpi^\prime) ) & \text{if } \lambda + w(\varpi^\prime) \in \widehat{C} \\ 0 & \text{otherwise} \end{cases} \]
	by \eqref{eq:translationWeylsimple}, and the claim follows because $w$ permutes the weight in $W_\mathrm{fin} \varpi$.
\end{proof}

In Section \ref{sec:B2} below, we will need to explicitly compute the characters of certain tensor products of the form $\Delta(\lambda) \otimes \Delta(\mu)$, with $\lambda,\mu \in X^+$.
We will use the following formula in terms of weight space dimensions; see Lemma II.5.8 in \cite{Jantzen}.

\begin{Proposition} \label{prop:productWeylcharacters}
	For every $G$-module $M$ and all $\lambda \in X^+$, we have
	\[ \ch M \cdot \chi(\lambda) = \sum_{\nu\in X} \dim M_\nu \cdot \chi(\lambda+\nu) . \]
\end{Proposition}

\subsection{Comparison with characteristic zero}
\label{subsec:char0}

Let us write $G_\C$ for the unique simply connected simple algebraic group over the complex numbers $\C$ with root system $\Phi$.
In this subsection, we explain how the composition multiplicities in tensor products of simple $G_\C$-modules can be used to study multiplicity freeness of tensor products of simple $G$-modules.
For $\lambda \in X^+$, denote by $L_\C(\lambda)$ the simple $G_\C$-module of highest weight $\lambda$, and note that the character $\ch L_\C(\lambda) = \chi(\lambda) = \ch \nabla(\lambda)$ is given by Weyl's character formula.
Therefore, we can write
\[ \sum_{\nu \in X^+} \big[ L_\C(\lambda) \otimes L_\C(\mu) : L_\C(\nu) \big] \cdot \chi(\nu) = \chi(\lambda) \cdot \chi(\mu) = \sum_{\nu \in X^+} \big[ \nabla(\lambda) \otimes \nabla(\mu) : \nabla(\nu) \big]_\nabla \cdot \chi(\nu) , \]
and it follows that
\[ \big[ L_\C(\lambda) \otimes L_\C(\mu) : L_\C(\nu) \big] = \big[ \nabla(\lambda) \otimes \nabla(\mu) : \nabla(\nu) \big]_\nabla = \dim \Hom_G\big( \Delta(\nu) , \nabla(\lambda) \otimes \nabla(\mu) \big) \]
for all $\lambda,\mu,\nu \in X^+$.
This observation is crucial for the proofs of the two following results.

\begin{Lemma}\label{lem:MFchar0boundVerlindecoeff}
	For all $\lambda,\mu,\nu\in C_0 \cap X$, we have $c_{\lambda,\mu}^\nu \leq [L_{\C}(\lambda)\otimes L_{\C}(\mu):L_{\C}(\nu)]$.
	In particular, if the tensor product $L_{\C}(\lambda)\otimes L_{\C}(\mu)$ is multiplicity free then $c_{\lambda,\mu}^\nu\leq 1$ for all $\nu\in C_0 \cap X$.
\end{Lemma}
\begin{proof}
	For all $\nu\in C_0 \cap X$, we have 
	\begin{multline*}
		\qquad c_{\lambda,\mu}^\nu = [ T(\lambda) \otimes T(\mu) : T(\nu) ]_\oplus \leq [ T(\lambda) \otimes T(\mu) : \nabla(\nu) ]_\nabla \\
		=[ \nabla(\lambda) \otimes \nabla(\mu) : \nabla(\nu) ]_\nabla = [L_{\C}(\lambda)\otimes L_{\C}(\mu):L_{\C}(\nu)]. \qquad
	\end{multline*}
	Thus if $[L_{\C}(\lambda)\otimes L_{\C}(\mu):L_{\C}(\nu)]\leq 1$ for all $\nu \in C_0 \cap X$, we conclude that $c_{\lambda,\mu}^\nu\leq 1$ for all $\nu\in C_0 \cap X$.
\end{proof}

\begin{Proposition} \label{prop:CRandMFchar0impliesMF}
	Let $\lambda,\mu \in X^+$ and suppose that $L(\lambda) \otimes L(\mu)$ is completely reducible.
	If the tensor product $L_\C(\lambda) \otimes L_\C(\mu)$ of simple $G_\C$-modules is multiplicity free then so is $L(\lambda) \otimes L(\mu)$.
\end{Proposition}
\begin{proof}
	If $L(\lambda) \otimes L(\mu)$ is completely reducible then we have
	\begin{multline*}
		\quad [ L(\lambda) \otimes L(\mu) : L(\nu) ] = \dim \Hom_G\big( L(\nu) , L(\lambda) \otimes L(\mu) \big) \\ \leq \dim \Hom_G\big( \Delta(\nu) , \nabla(\lambda) \otimes \nabla(\mu) \big) = [ L_\C(\lambda) \otimes L_\C(\mu) : L_\C(\nu) ] \quad
	\end{multline*}
	for all $\nu \in X^+$, and the claim is immediate.
\end{proof}

The comparison with multiplicities in characteristic zero also allows us to observe the following monotonicity property for good filtration multiplicities in tensor products of induced modules and Weyl filtration multiplicities in tensor products of Weyl modules.

\begin{Lemma} \label{lem:tensormultiplicitiesmonotonous}
	For $\lambda,\mu,\nu,\delta \in X^+$, we have
	\begin{align*}
		\big[ \nabla(\lambda+\delta) \otimes \nabla(\mu) : \nabla(\nu+\delta) \big]_\nabla & \geq \big[ \nabla(\lambda) \otimes \nabla(\mu) : \nabla(\nu) \big]_\nabla , \\[3pt]
		\big[ \Delta(\lambda+\delta) \otimes \Delta(\mu) : \Delta(\nu+\delta) \big]_\Delta & \geq \big[ \Delta(\lambda) \otimes \Delta(\mu) : \Delta(\nu) \big]_\Delta .
	\end{align*}
\end{Lemma}
\begin{proof}
	As observed at the beginning of the subsection, we have
	\[ \big[ L_\C(\lambda) \otimes L_\C(\mu) : L_\C(\nu) \big] = \big[ \nabla(\lambda) \otimes \nabla(\mu) : \nabla(\nu) \big]_\nabla \]
	for all $\lambda,\mu,\nu \in X^+$,
	and the first inequality follows from Proposition 2.9 in \cite{Stembridge}.
	The second inequality follows by taking duals.
\end{proof}

\section{Type \texorpdfstring{$\mathrm{A}_2$}{A2}}
\label{sec:A2}

In this section, we consider the group $G = \SL_3(\kk)$ and give a complete classification of the pairs of simple $G$-modules whose tensor product is completely reducible or multiplicity free.
We start by setting up our notation for the affine Weyl group and alcove geometry for $G$ in Subsection \ref{subsec:A2setup}, then we state the main results in Subsection \ref{subsec:A2results}.
In Subsection \ref{subsec:A2Weylmodulestiltingmodules}, we recall some well-known results about the submodule structure of certain Weyl modules and tilting modules, and the proofs of the main theorems from Subsection \ref{subsec:A2results} are given in Subsections \ref{subsec:A2proofsCR} and \ref{subsec:A2proofsMF}.

\subsection{Setup}
\label{subsec:A2setup}

The positive roots for $G = \SL_3(\kk)$ can be written as $\Phi^+ = \{ \alpha_1 , \alpha_2 , \alpha_\mathrm{h} \}$, where $\alpha_1$ and $\alpha_2$ are the simple roots and $\alpha_\mathrm{h} = \alpha_1 + \alpha_2$ is the highest root.
The weight lattice $X$ is spanned by the fundamental dominant weights $\varpi_1$ and $\varpi_2$.
The weights $\varpi_1$ and $\varpi_2$ are minuscule, and we have
\[ \alpha_1 = 2 \varpi_1 - \varpi_2 , \qquad \alpha_2 = - \varpi_1 + 2 \varpi_2 , \qquad \alpha_\mathrm{h} = \varpi_1 + \varpi_2 = \rho . \]
We write $S_\mathrm{fin} = \{t,u\}$ for the set of simple reflections in $W_\mathrm{fin}$, with $t = s_{\alpha_1}$ and $u = s_{\alpha_2}$.
The affine simple reflection is denoted by $s = s_0 = t_{\alpha_\mathrm{h}} s_{\alpha_\mathrm{h}}$, so that $S_\mathrm{aff} = \{ s,t,u \}$ is the set of simple reflections in $W_\mathrm{aff}$.
We fix labelings for certain alcoves as in the figure below, so that
\[ C_1 = s \Cdot C_0 , \qquad C_{2a} = su \Cdot C_0 , \qquad C_{2b} = st \Cdot C_0 . \]

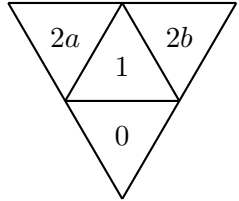
\begin{center}
\begin{tikzpicture}[scale=1.5]
	\pgftransformcm{cos(60)}{sin(60)}{cos(120)}{sin(120)}{\pgfpoint{0cm}{0cm}}
	\draw[thick] (0,0) -- (2,0);
	\draw[thick] (0,0) -- (0,2);
	\draw[thick] (0,1) -- (1,0);
	\draw[thick] (0,1) -- (1,1);
	\draw[thick] (1,0) -- (1,1);
	\draw[thick] (0,2) -- (2,0);
	
	\node at (.325,.325) {0};
	\node at (.675,.675) {1};
	\node at (.325,1.325) {$2a$};
	\node at (1.325,.325) {$2b$};
\end{tikzpicture}
\end{center}
Alternatively, the labeled alcoves can be described as follows:
\begin{align*}
	C_0 & = \{ a \varpi_1 + b \varpi_2 \mid a > -1 , ~ b > -1 , ~ a+b < p-2 \} , \\
	C_1 & = \{ a \varpi_1 + b \varpi_2 \mid a+b > p-2 , ~ a < p-1 , ~ b < p-1 \} , \\
	C_{2a} & = \{ a \varpi_1 + b \varpi_2 \mid a > p-1 , ~ b > -1 , ~ a+b < 2p-2 \} , \\
	C_{2b} & = \{ a \varpi_1 + b \varpi_2 \mid a > -1 , ~ b > p-1 , ~ a+b < 2p-2 \} .
\end{align*}
Observe that all $p$-restricted weights belong to the upper closure of one of the alcoves $C_0$ and $C_1$.
The wall separating two alcoves $C_x$ and $C_y$ will be denoted by $F_{x,y} = F_{y,x}$, for $x,y \in \{0,1,2a,2b\}$.
In the following, we give a concrete example of what the reflection smallness condition from Section \ref{sec:alcovegeometryCR} means for a weight in one of the alcoves $C_0$ and $C_1$.

\begin{Example} \label{ex:A2reflectionsmall}
	Let $\lambda,\mu \in X^+$ and write $\lambda = a \varpi_1 + b \varpi_2$ and $\mu = a^\prime \varpi_1 + b^\prime \varpi_2$.
	Then we have
	\[ W_\mathrm{fin} \mu = \big\{ c \varpi_1 + d \varpi_2 \mathrel{\big|} \pm (c,d) \in \{ (a^\prime,b^\prime) , (a^\prime+b^\prime,-b^\prime) , (-a^\prime,a^\prime+b^\prime) \} \big \} . \]
	The alcove $C_0$ has a unique wall $F_{0,1} = H_{\alpha_\mathrm{h},1}$ that belongs to its upper closure, and so if $\lambda \in \widehat{C}_0$ then $\mu$ is reflection small with respect to $\lambda$ if and only if $\lambda+w\mu \leq s_{\alpha_\mathrm{h},1} \Cdot (\lambda+w\mu)$ for all $w \in W_\mathrm{fin}$, or equivalently, $(\lambda+w\mu+\rho,\alpha_\mathrm{h}^\vee) \leq p$ for all $w \in W_\mathrm{fin}$.
	As $(\lambda+\rho,\alpha_\mathrm{h}^\vee) = a+b+2$ and $(w\mu,\alpha_\mathrm{h}^\vee) \in \{ \pm (a'+b') , \pm a' , \pm b' \}$, we conclude that $\mu$ is reflection small with respect to $\lambda$ if and only if
	\[ a + b + a^\prime + b^\prime \leq p-2 . \]
	Similarly, the walls that belong to the upper closure of $C_1$ are precisely $F_{1,2a} = H_{\alpha_1,1}$ and $F_{1,2b} = H_{\alpha_2,1}$, and so if $\lambda \in \widehat{C}_1$ then $\mu$ is reflection small with respect to $\lambda$ if and only if $(\lambda+w\mu+\rho,\alpha_1^\vee) \leq p$ and $(\lambda+w\mu+\rho,\alpha_2^\vee) \leq p$ for all $w \in W_\mathrm{fin}$.
	This is easily seen to be equivalent to the conditions
	\[ a+a^\prime+b^\prime \leq p-1 \qquad \text{and} \qquad b+a^\prime+b^\prime \leq p-1 . \]
\end{Example}

\subsection{Main results}
\label{subsec:A2results}

We are now ready to state our main results about complete reducibility and multiplicity freeness of tensor products of simple modules for $G = \SL_3(\kk)$.
The proofs will be given at the end of Subsections \ref{subsec:A2proofsCR} and \ref{subsec:A2proofsMF}, respectively.
Recall from Subsection \ref{subsec:prestricted} that it suffices to consider pairs of simple $G$-modules with $p$-restricted highest weights.

\begin{Theorem} \label{thm:A2CR}
	For $\lambda , \mu \in X_1 \setminus \{0\}$, the tensor product $L(\lambda) \otimes L(\mu)$ is completely reducible if and only if $\lambda$ and $\mu$ satisfy one of the conditions in Table \ref{tab:A2CR}, up to interchanging $\lambda$ and $\mu$.
	
	\begin{table}[htbp]
	\centering
	\caption{Weights $\lambda,\mu \in X_1 \setminus \{ 0 \}$ such that $L(\lambda) \otimes L(\mu)$ is completely reducible, for $G$ of type $\mathrm{A}_2$.}
	\label{tab:A2CR}
	\begin{tabular}{|l||l|}
	\hline
	$\mathbf{\#}$ & \textbf{conditions} \\
	\hline\hline
	1 & $\lambda = a\varpi_1$ and $\mu = a^\prime \varpi_1$, where $a+a^\prime = p-1$ \\
	\hline
	1* & $\lambda = b\varpi_2$ and $\mu = b^\prime \varpi_2$, where $b+b^\prime = p-1$ \\
	\hline
	2 & $\lambda = (p-1) \cdot \varpi_1$ and $\mu = \varpi_2$ \\
	\hline
	2* & $\lambda = (p-1) \cdot \varpi_2$ and $\mu = \varpi_1$ \\
	\hline
	3 & $\lambda = a \varpi_1 + b \varpi_2 \in C_1$ and $\mu = (b+1) \cdot \varpi_2$, where $a+b = p-1$ and $b<a$ \\
	\hline
	3* & $\lambda = a \varpi_1 + b \varpi_2 \in C_1$ and $\mu = (a+1) \cdot \varpi_1$, where $a+b = p-1$ and $a<b$ \\
	\hline
	4 & $\lambda \in C_0 \cup C_1$ and $\mu$ is reflection small with respect to $\lambda$ \\
	\hline
	\end{tabular}
	\end{table}
\end{Theorem}

\begin{Theorem} \label{thm:A2MF}
	For $\lambda , \mu \in X_1 \setminus \{0\}$, the tensor product $L(\lambda) \otimes L(\mu)$ is multiplicity free if and only if $\lambda$ and $\mu$ satisfy one of the conditions in Table \ref{tab:A2MF}, up to interchanging $\lambda$ and $\mu$.
	\medskip

	\begin{table}[htbp]
	\centering
	\caption{Weights $\lambda,\mu \in X_1 \setminus \{ 0 \}$ such that $L(\lambda) \otimes L(\mu)$ is multiplicity free, for $G$ of type $\mathrm{A}_2$.}
	\label{tab:A2MF}
	\begin{tabular}{|l||l|}
	\hline
	$\mathbf{\#}$ & \textbf{conditions} \\
	\hline\hline
	1 & $\lambda = a\varpi_1$ and $\mu = a^\prime \varpi_1$, where $a+a^\prime = p-1$ \\
	\hline
	1* & $\lambda = b\varpi_2$ and $\mu = b^\prime \varpi_2$, where $b+b^\prime = p-1$ \\
	\hline
	2 & $\lambda = (p-1) \cdot \varpi_1$ and $\mu = \varpi_2$ \\
	\hline
	2* & $\lambda = (p-1) \cdot \varpi_2$ and $\mu = \varpi_1$ \\
	\hline
	3 & $\lambda = a \varpi_1 + b \varpi_2 \in C_1$ and $\mu = (b+1) \cdot \varpi_2$, where $a+b = p-1$ and $b<a$ \\
	\hline
	3* & $\lambda = a \varpi_1 + b \varpi_2 \in C_1$ and $\mu = (a+1) \cdot \varpi_1$, where $a+b = p-1$ and $a<b$ \\
	\hline
	4a & $\lambda \in C_0 \cup C_1$ and $\mu = a^\prime \varpi_1$ is reflection small with respect to $\lambda$; \\
	\hline
	4a* & $\lambda \in C_0 \cup C_1$ and $\mu = b^\prime \varpi_2$ is reflection small with respect to $\lambda$; \\
	\hline
	4b & $\lambda = a \varpi_1 + b \varpi_2 \in C_1$, where $a+b=p-1$, and $\mu$ is reflection small with respect to $\lambda$; \\
	\hline
	\end{tabular}
	\end{table}
\end{Theorem}

\subsection{Structure of Weyl modules and tilting modules} \label{subsec:A2Weylmodulestiltingmodules}

In order to prove our main results, we will need some information about the submodule structure of certain Weyl modules and indecomposable tilting modules for $G = \SL_3(\kk)$.
The structure of these modules has been computed by Bowman, Doty and Martin in \cite[Theorem B and Section 4]{BowmanDotyMartin}, for $p \geq 5$.
We recall some of their results below, noting that our statements remain valid for $p \leq 3$ by \cite[Sections 3.1, 4.1, 4.2]{BowmanDotyMartinsmallprimes} (with the exception that there are no $p$-regular weights for $p=2$).

The Weyl modules $\Delta(w\Cdot\lambda)$, with $\lambda \in C_0 \cap X$ and $w \in \{ e , s , st , su \}$, are all uniserial\footnote{A $G$-module $M$ is called uniserial if it admits a unique composition series $0 = M_0 \subseteq \cdots \subseteq M_r = M$.
We depict the structure of $M$ by writing $M = [L_r,\ldots,L_1]$, where $L_i = M_i / M_{i-1}$ for $1 \leq i \leq r$.},
with respective composition series given by $\Delta(\lambda) = [ L(\lambda) ]$ and
\[ \Delta(s\Cdot\lambda) = [ L(s\Cdot\lambda) , L(\lambda) ] , \qquad \Delta(st \Cdot \lambda) = [ L(st \Cdot\lambda) , L(s \Cdot\lambda) ] , \qquad \Delta(su\Cdot\lambda) = [ L(su \Cdot\lambda) , L(s \Cdot\lambda) ] . \]
Using the translation principle, it follows that the Weyl modules with $p$-singular $p$-restricted highest weights are all simple, and so the indecomposable tilting modules with $p$-singular $p$-restricted highest weights are also simple.
Furthermore, we have $T(\lambda) = \Delta(\lambda) = L(\lambda)$, and the tilting module $T(s\Cdot\lambda)$ is uniserial with composition series $T(s\Cdot\lambda) = [ L(\lambda) , L(s\Cdot\lambda) , L(\lambda) ]$.
The structure of the tilting modules $T(st\Cdot\lambda)$ and $T(su\Cdot\lambda)$ with highest weights in $C_{2a}$ and $C_{2b}$ is displayed in the following diagrams, where we replace a simple $G$-module $L(w\Cdot\lambda)$ by the label $w \in W_\mathrm{aff}^+$:
\[ T(st\Cdot\lambda) = 
	\begin{tikzpicture}[scale=.6,baseline={([yshift=-.5ex]current bounding box.center)}]
	\node (C1) at (0,0) {\small $s$};
	\node (C2) at (-1.5,1.5) {\small $st$};
	\node (C3) at (1.5,1.5) {\small $e$};
	\node (C4) at (0,3) {\small $s$};
	
	\draw (C1) -- (C2);
	\draw (C1) -- (C3);
	\draw (C2) -- (C4);
	\draw (C3) -- (C4);
	\end{tikzpicture} \hspace{2cm}
	T(su\Cdot\lambda) = 
	\begin{tikzpicture}[scale=.6,baseline={([yshift=-.5ex]current bounding box.center)}]
	\node (C1) at (0,0) {\small $s$};
	\node (C2) at (-1.5,1.5) {\small $su$};
	\node (C3) at (1.5,1.5) {\small $e$};
	\node (C4) at (0,3) {\small $s$};
	
	\draw (C1) -- (C2);
	\draw (C1) -- (C3);
	\draw (C2) -- (C4);
	\draw (C3) -- (C4);
	\end{tikzpicture} \]

\subsection{Proofs: complete reducibility} \label{subsec:A2proofsCR}

In this subsection, we establish all the necessary preliminary results for proving Theorem \ref{thm:A2CR}.
The proof of the theorem is given at the end of the subsection (see page \pageref{proof:A2CR}).

\begin{Remark} \label{rem:LeviprestrctedconditionA2}
	Let $\lambda,\mu \in X_1$ and write $\lambda = a\varpi_1 + b\varpi_2$ and $\mu = a^\prime\varpi_1 + b^\prime\varpi_2$, with $0 \leq a,b,a^\prime,b^\prime < p$.
	Suppose that $L(\lambda) \otimes L(\mu)$ is completely reducible.
	Then all composition factors of $L(\lambda) \otimes L(\mu)$ have $p$-restricted highest weights by Theorem A in \cite{GruberCompleteReducibility}.
	By truncation to the two Levi subgroups of type $\mathrm{A}_1$ corresponding to the simple roots $\alpha_1$ and $\alpha_2$ respectively, we see that $L(\lambda+\mu - c\alpha_1)$ and $L(\lambda+\mu-d\alpha_2)$ are composition factors of $L(\lambda) \otimes L(\mu)$ for all $0 \leq c \leq \min\{b,b^\prime\}$ and $0 \leq d \leq \min\{a,a^\prime\}$, each appearing with composition multiplicity one (cf.\ Remark 4.13 in \cite{GruberCompleteReducibility}).
	In particular, the weights $\lambda+\mu - c \alpha_1$ and $\lambda+\mu - d \alpha_2$ are $p$-restricted, and it follows that
	\begin{align*}
	a + a^\prime + \min\{ b , b^\prime \} & \leq p-1 , \\
	b + b^\prime + \min\{ a , a^\prime \} & \leq p-1 .
	\end{align*}
\end{Remark}

We next compute two explicit direct sum decompositions of tensor products.
These will be used in Lemma \ref{lem:A2api1tensoraprimepi1} below to establish the complete reducibility of tensor products of simple $G$-modules that satisfy the condition 1 in Table \ref{tab:A2CR}

\begin{Lemma} \label{lem:A2wallstensorpi1}
	We have the following direct sum decompositions of tensor products:
	\begin{enumerate}
		\item \quad $L( (p-1) \cdot \varpi_2 ) \otimes L(\varpi_1) \cong L( \varpi_1 + (p-1) \cdot \varpi_2 ) \oplus L( (p-2) \cdot \varpi_2 )$;
		\item \quad If $p \geq 3$ then $L( (p-2) \cdot \varpi_1 ) \otimes L(\varpi_1) \cong L( (p-1) \cdot \varpi_1 ) \oplus L( (p-3) \cdot \varpi_1 + \varpi_2 )$.
	\end{enumerate}
\end{Lemma}
\begin{proof}
	By Subsection \ref{subsec:A2Weylmodulestiltingmodules}, we have $L( (p-1) \cdot \varpi_2 ) \cong \Delta( (p-1) \cdot \varpi_2 )$, and using Lemma \ref{lem:minusculetensorproducttranslation} and Proposition II.7.13 in \cite{Jantzen}, it follows that
	\[ L( (p-1) \cdot \varpi_2 ) \otimes L(\varpi_1) \cong \Delta( (p-1) \cdot \varpi_1 + \varpi_2 ) \oplus \Delta( (p-2) \cdot \varpi_2 ) . \]
	Now the first claim follows because
	\[ \Delta( (p-1) \cdot \varpi_1 + \varpi_2 ) = L( (p-1) \cdot \varpi_1 + \varpi_2 ) , \qquad \Delta( (p-2) \cdot \varpi_2 ) = L( (p-2) \cdot \varpi_2 ) , \]
	again by Subsection \ref{subsec:A2Weylmodulestiltingmodules}.
	The proof of (2) is analogous.
\end{proof}

\begin{Lemma} \label{lem:A2api1tensoraprimepi1}
	Let $0 \leq a \leq p-1$ and $a^\prime = p-1-a$.
	Then $L(a\varpi_1) \otimes L(a^\prime \varpi_1)$ is completely reducible.
\end{Lemma}
\begin{proof}
	If $a=0$ or $a=p-1$ then the claim is trivially satisfied.
	Now suppose that $0<a<p-1$, and observe that $L(a\varpi_1)$ is a direct summand of $L((a-1)\cdot\varpi_1) \otimes L(\varpi_1)$ by Corollary \ref{cor:minusculetensorproductMFCR}.
	Let $M$ be an indecomposable direct summand of $L(a\varpi_1) \otimes L(a^\prime\varpi_1)$; then $M$ is also a direct summand of the tensor product $L((a-1)\cdot\varpi_1) \otimes L(a^\prime\varpi_1) \otimes L(\varpi_1)$.
	The weight $(a-1) \cdot \varpi_1$ is reflection small with respect to $a^\prime \varpi_1$ by Example \ref{ex:A2reflectionsmall}, so $L((a-1)\cdot\varpi_1) \otimes L(a^\prime\varpi_1)$ is a direct sum of simple $G$-modules with highest weights in $\widehat{C}_0$ by Theorem \ref{thm:reflectionsmalltensorproduct}.
	It is straightforward to see by weight considerations that the only simple direct summand of $L((a-1)\cdot\varpi_1) \otimes L(a^\prime\varpi_1)$ with highest weight in $\widehat{C}_0 \setminus C_0$ is $L((p-2)\cdot\varpi_1)$, and therefore $M$ is a direct summand of $L(\nu) \otimes L(\varpi_1)$ for some $\nu \in C_0 \cup \{(p-2) \cdot \varpi_1\}$.
	Now $L(\nu) \otimes L(\varpi_1)$ is completely reducible by Corollary \ref{cor:minusculetensorproductMFCR} (for $\nu \in C_0$) and Lemma \ref{lem:A2wallstensorpi1} (for $\nu = (p-2) \cdot \varpi_1$), so $M$ is simple and $L(\lambda) \otimes L(\mu)$ is completely reducible, as claimed.
\end{proof}

Our next goal is to prove that a tensor product $L(\lambda) \otimes L(\mu)$ with $\lambda \in \widehat{C}_1$ and $\mu \in X_1$ is completely reducible if and only if one of the conditions 2--4 or 2*--3* from Table \ref{tab:A2CR} are satisfied.
This will follow from Lemma \ref{lem:A2lambdainC1reflectionsmall} and Proposition \ref{prop:A2lambdaabovewall} below.

\begin{Lemma} \label{lem:A2lambdainC1reflectionsmall}
	Let $\lambda = a\varpi_1 + b\varpi_2 \in \widehat{C}_1 \cap X$ and $\mu = a^\prime \varpi_1 + b^\prime \varpi_2 \in X_1 \setminus \{0\}$ such that $L(\lambda) \otimes L(\mu)$ is completely reducible.
	Then one of the following holds:
	\begin{enumerate}
		\item $\mu$ is reflection small with respect to $\lambda$;
		\item $a+b = p-1$ and $a^\prime=0$ and $b < b^\prime \leq a$;
		\item $a+b = p-1$ and $b^\prime=0$ and $a < a^\prime \leq b$.
	\end{enumerate}
\end{Lemma}
\begin{proof}
	If $a^\prime \leq a$ and $b^\prime \leq b$ then $a + a^\prime + b^\prime \leq p-1$ and $b + a^\prime + b^\prime \leq p-1$ by Remark \ref{rem:LeviprestrctedconditionA2}, and so $\mu$ is reflection small with respect to $\lambda$ (see Example \ref{ex:A2reflectionsmall}).
	If $b^\prime>b$ then Remark \ref{rem:LeviprestrctedconditionA2} and the fact that $\lambda$ belongs to the upper closure of $C_1$ imply that
	\[ p-1 \leq a+b \leq a+b+a^\prime \leq p-1 , \] and so $a^\prime = 0$ and $a+b = p-1$.
	The inequality $b^\prime \leq p-1-b = a$ also follows from Remark \ref{rem:LeviprestrctedconditionA2}, and the case $a^\prime>a$ is analogous.
\end{proof}

\begin{Proposition} \label{prop:A2directsumandsC1tilting}
	Let $\lambda \in C_1 \cap X$ and let $T$ be a tilting module.
	If $M$ is an indecomposable direct summand of $L(\lambda) \otimes T$ then either $M$ is a negligible tilting module or $M \cong L(\mu)$ for some $\mu \in C_1$.
\end{Proposition}
\begin{proof}
	We may assume without loss of generality that $T = T(\nu)$ for some $\nu = a \varpi_1 + b \varpi_2 \in X^+$, and we prove the claim by induction on $a+b$.
	The case $a+b=0$ is trivial since $L(\lambda) \otimes T(0) \cong L(\lambda)$.
	Now suppose that $a+b \neq 0$.
	By symmetry, we may further assume that $a \neq 0$, whence $T(\nu)$ is a direct summand of $T(\nu-\varpi_1) \otimes T(\varpi_1)$, and every indecomposable direct summand $M$ of $T(\nu) \otimes L(\lambda)$ is a direct summand of $N \otimes T(\varpi_1)$, for some indecomposable direct summand $N$ of $T(\nu-\varpi_1) \otimes L(\lambda)$.
	By induction, we know that either $N$ is a negligible tilting module or $N \cong L(\mu)$ for some $\mu \in C_1$.
	If $N$ is a negligible tilting module then so is $M$ because the negligible tilting modules form a thick tensor ideal in $\Tilt(G)$ (cf.\ Subsection \ref{subsec:singular}).
	If $N \cong L(\mu)$ for some $\mu \in C_1$ then $T(\varpi_1) \otimes N \cong L(\varpi_1) \otimes L(\mu)$ is a direct sum of simple $G$-modules with highest weights in the upper closure of $C_1$ by Corollary \ref{cor:minusculetensorproductMFCR}, and so $M \cong L(\mu^\prime)$ for some weight $\mu^\prime \in \widehat{C}_1$.
	Now it only remains to observe that for $\mu \in \widehat{C}_1 \setminus C_1$, the simple $G$-module $L(\mu) = T(\mu)$ is a negligible tilting module; see Subsection \ref{subsec:A2Weylmodulestiltingmodules}.
\end{proof}

\begin{Proposition} \label{prop:A2lambdaabovewall}
	Let $\lambda = a \varpi_1 + b\varpi_2 \in \widehat{C}_1 \cap X$ and $\mu = b^\prime \varpi_2$ with $a+b = p-1$ and $b < b' \leq a$.
	Then the tensor product $L(\lambda) \otimes L(\mu)$ is completely reducible if and only if $b^\prime = b+1$.
\end{Proposition}
\begin{proof}
	First suppose that $b^\prime > b+1$ and consider the weight $\nu = (p-1-b^\prime+b) \cdot \varpi_1 \in C_0$.
	By the results in Subsection \ref{subsec:A2Weylmodulestiltingmodules}, we have $L(\mu) = T(\mu)$ and $L(\nu) = T(\nu)$, and it is straightforward to see using weight considerations that the indecomposable tilting module of highest weight
	\[ -w_0(\mu)+\nu = (p-1+b) \cdot \varpi_1 = sus \Cdot \lambda \in C_{2a} \]
	is a direct summand of $L(\mu)^* \otimes L(\nu)$.
	Since $L(\lambda)$ is a submodule of $T(sus\Cdot\lambda)$ by Subsection \ref{subsec:A2Weylmodulestiltingmodules}, we conclude that
	\[ 0 \neq \Hom_G\big( L(\lambda) , L(\mu)^* \otimes L(\nu) \big) \cong \Hom_G\big( L(\lambda) \otimes
	 L(\mu) , L(\nu) \big) . \]
	The non-negligible tilting module $T(\nu) \cong L(\nu)$ cannot be a direct summand of $L(\lambda) \otimes L(\mu)$ by Proposition \ref{prop:A2directsumandsC1tilting}, and we conclude that $L(\lambda) \otimes L(\mu)$ is not completely reducible.
	 
	 It remains to show that $L(\lambda) \otimes L(\mu)$ is completely reducible if $b^\prime=b+1$.
	 To that end, observe that $L(\mu) = L( (b+1) \cdot \varpi_2 )$ is a direct summand of $L(b\varpi_2) \otimes L(\varpi_2)$, so every indecomposable direct summand $M$ of $L(\lambda) \otimes L(\mu)$ is also a direct summand of $L(\lambda) \otimes L(b\varpi_2) \otimes L(\varpi_2)$.
	 Now $b \varpi_2$ is reflection small with respect to $\lambda$, so $L(\lambda) \otimes L(b\varpi_2)$ is a direct sum of simple $G$-modules with highest weights in $\widehat{C}_1$ by Theorem \ref{thm:reflectionsmalltensorproduct}, and by weight considerations, it is straightforward to see that the only simple direct summand of $L(\lambda) \otimes L(b\varpi_2)$ with highest weight in $\widehat{C}_1 \setminus C_1$ is $L((p-1)\cdot \varpi_1)$.
	 We conclude that $M$ is a direct summand of $L(\nu) \otimes L(\varpi_1)$ for some weight $\nu \in C_1 \cup\{ (p-1) \cdot \varpi_1 \}$.
	Now $L(\nu) \otimes L(\varpi_1)$ is completely reducible by by Corollary \ref{cor:minusculetensorproductMFCR} (for $\nu \in C_1$) and Lemma \ref{lem:A2wallstensorpi1} (for $\nu = (p-1) \cdot \varpi_2$), so $M$ is simple and $L(\lambda) \otimes L(\mu)$ is completely reducible, as required.
\end{proof}

Before we can give the proof of Theorem \ref{thm:A2CR}, it remains to consider tensor products of simple $G$-modules with highest weights in the upper closure of $C_0$.

\begin{Proposition} \label{prop:A2CRweightsinC0}
	Let $\lambda,\mu \in \widehat{C}_0 \cap X$ and suppose that the tensor product $L(\lambda) \otimes L(\mu)$ is completely reducible.
	Then one of the following holds:
	\begin{enumerate}
	\item $\lambda+\mu \in \widehat{C}_0$ (i.e.\ $\mu$ is reflection small with respect to $\lambda$);
	\item $\lambda = a \varpi_1$ and $\mu = a^\prime \varpi_1$, where $a+a^\prime = p-1$;
	\item $\lambda = b \varpi_2$ and $\mu = b^\prime \varpi_2$, where $b+b^\prime = p-1$.
\end{enumerate}
\end{Proposition}
\begin{proof}
	As $\lambda,\mu \in \widehat{C}_0$, we have $L(\lambda) \cong T(\lambda)$ and $L(\mu) \cong T(\mu)$, and it follows that $T(\lambda+\mu)$ is a direct summand of $L(\lambda) \otimes L(\mu) \cong T(\lambda) \otimes T(\mu)$.
	As the indecomposable tilting modules with highest weights in $C_1$ are non-simple (see Subsection \ref{subsec:A2Weylmodulestiltingmodules}) and as all composition factors of $L(\lambda) \otimes L(\mu)$ have $p$-restricted highest weights (see Remark \ref{rem:LeviprestrctedconditionA2}), this implies that either $\lambda+\mu$ belongs to the upper closure of $C_0$ or $\lambda+\mu$ belongs to one of the walls $F_{1,2a}$ and $F_{1,2b}$ of $C_1$.
	Observe that the first condition is equivalent to $\mu$ being reflection small with respect to $\lambda$ by Example \ref{ex:A2reflectionsmall}.
	
	Now assume without loss of generality that $\lambda + \mu \in F_{1,2a}$ (the case $\lambda+\mu \in F_{1,2b}$ being analogous).
	If we write $\lambda = a \varpi_1 + b \varpi_2$ and $\mu = a^\prime \varpi_1 + b^\prime \varpi_2$ then this means that $a+a^\prime = p-1$.
	Since we assume that $\lambda$ and $\mu$ belong to the upper closure of $C_0$, we further have $a+b \leq p-2$ and $a^\prime+b^\prime \leq p-2$, and it follows that $a,a^\prime \neq 0$ and
	\[ b + b^\prime \leq 2p-4 - (a+a^\prime) = p-3 . \]
	By Remark \ref{rem:LeviprestrctedconditionA2}, the tensor product $L(\lambda) \otimes L(\mu)$ has a composition factor of highest weight
	\[ \nu = \lambda+\mu - \alpha_1 = (p-3) \cdot \varpi_1 + (b+b^\prime+1) \cdot \varpi_2 , \]
	and by weight considerations, this implies that $T(\nu)$ is a direct summand of $L(\lambda) \otimes L(\mu)$.
	If $b+b^\prime \neq 0$ then we have $\nu \in C_1$ and $T(\nu)$ is non-simple, contradicting the complete reducibility of $L(\lambda) \otimes L(\mu)$.
	We conclude that $b+b^\prime = 0$, and $\lambda = a \varpi_1$ and $\mu = a^\prime \varpi_1$ with $a+a^\prime = p-1$, as required.
\end{proof}

Now we are ready to give the proof of our classification of pairs of simple $G$-modules whose tensor product is completely reducible, for $G = \SL_3(\kk)$.

\begin{proof}[Proof of Theorem \ref{thm:A2CR}]
\label{proof:A2CR}
	Let $\lambda,\mu \in X_1 \setminus \{0\}$.
	If $\lambda$ and $\mu$ satisfy one of the conditions in Table \ref{tab:A2CR} then the tensor product $L(\lambda) \otimes L(\mu)$ is completely reducible by the following results.
	\begin{itemize}
		\item Condition 1: $L(\lambda) \otimes L(\mu)$ is completely reducible by Lemma \ref{lem:A2api1tensoraprimepi1}.
		\item Condition 2: $L(\lambda) \otimes L(\mu)$ is completely reducible by Lemma \ref{lem:A2wallstensorpi1}.
		\item Condition 3: $L(\lambda) \otimes L(\mu)$ is completely reducible by Proposition \ref{prop:A2lambdaabovewall}.
		\item Condition 4: $L(\lambda) \otimes L(\mu)$ is completely reducible by Theorem \ref{thm:reflectionsmalltensorproduct}.
	\end{itemize}
	For the remaining cases 1*--3* in Table \ref{tab:A2CR}, the complete reducibility of $L(\lambda) \otimes L(\mu)$ follows from the cases 1--3 by taking duals.
	
	Now suppose that $L(\lambda) \otimes L(\mu)$ is completely reducible, and recall that all $p$-restricted weights belong to the upper closure of one of the alcoves $C_0$ and $C_1$.
	If $\lambda$ belongs to the upper closure of $C_1$ then either $\lambda \in C_1$ and $\mu$ is reflection small with respect to $\lambda$ (so we are in case 4 of Table \ref{tab:A2CR}) or we are in one of the cases 2, 2*, 3 or 3* of Table \ref{tab:A2CR}, by Lemma \ref{lem:A2lambdainC1reflectionsmall} and Proposition \ref{prop:A2lambdaabovewall}.
	If $\lambda$ belongs to the upper closure of $C_0$ then we may further assume that $\mu$ belongs to the upper closure of $C_0$ (by the preceding case), and Proposition \ref{prop:A2CRweightsinC0} implies that we are in one of the cases 1, 1* or 4 of Table \ref{tab:A2CR}.
\end{proof}

\subsection{Proofs: multiplicity freeness} \label{subsec:A2proofsMF}

In this subsection, we give the proof of Theorem \ref{thm:A2MF}; see page~\pageref{proof:A2MF}.
We start with preliminary results about weight space dimensions and Verlinde coefficients.

\begin{Lemma} \label{lem:A2onedimensionalweightspaces}
	Let $\lambda = a\varpi_1 + b \varpi_2 \in X_1$.
	Then all weight spaces of the simple $G$-module $L(\lambda)$ are at most one-dimensional if and only if $a=0$ or $b=0$ or $a+b=p-1$.
\end{Lemma}
\begin{proof}
	If all weight spaces of $L(\lambda)$ are at most one-dimensional then $a=0$ or $b=0$ or $a+b=p-1$ by Section~6.1 in \cite{SeitzMemoir}.
	The converse is established in the proof of Proposition 1.2 in \cite{SuprunenkoZalesskii}.
	Alternatively, the converse follows from the fact that all weight spaces of the Weyl modules $\Delta(c\varpi_i)$ are at most one-dimensional, for $c \in \Z_{\geq 0}$ and $i=1,2$.
	If $a=0$ or $b=0$ then $L(\lambda)$ is a quotient of one of these Weyl modules, and if $a+b = p-1$ then $L(\lambda)$ is a submodule of $\Delta((p-1+b)\cdot\varpi_1)$ (cf.\ Subsection \ref{subsec:A2Weylmodulestiltingmodules}).
\end{proof}

\begin{Lemma} \label{lem:A2VerlindesuminC0}
	Let $\lambda = a\varpi_1 + b\varpi_2 \in C_0 \cap X$ and $\mu = a^\prime\varpi_1 + b^\prime\varpi_2 \in C_0 \cap X$ such that $\lambda+\mu \in \widehat{C}_0$, and further suppose that $a,a^\prime,b,b^\prime \geq 1$.
	Then we have $c_{\lambda,\mu}^{\lambda+\mu-\alpha_\mathrm{h}} \geq 2$.
\end{Lemma}
\begin{proof}
	As $\lambda,\mu \in C_0$ and $\lambda + \mu \in \widehat{C}_0$, we have $T(\lambda) = \nabla(\lambda)$ and $T(\mu) = \nabla(\mu)$, and $T(\lambda) \otimes T(\mu)$ is a direct sum of tilting modules $T(\nu) = \nabla(\nu)$ with highest weights $\nu$ in the upper closure of $C_0$.
	This implies that
	\[ c_{\lambda,\mu}^\nu = [ T(\lambda) \otimes T(\mu) : T(\nu) ]_\oplus  = [ \nabla(\lambda) \otimes \nabla(\mu) : \nabla(\nu) ]_\nabla \]
	for all $\nu \in C_0 \cap X$.
	A straightforward computation at the level of characters shows that
	\[ [ \nabla(\varpi_1+\varpi_2) \otimes \nabla(\varpi_1+\varpi_2) : \nabla(\varpi_1+\varpi_2) ]_\nabla = 2 , \]
	and since $a,a^\prime,b,b^\prime \geq 1$, Lemma \ref{lem:tensormultiplicitiesmonotonous} yields
	\[ c_{\lambda,\mu}^{\lambda+\mu-\alpha_\mathrm{h}} = [ \nabla(\lambda) \otimes \nabla(\mu) : \nabla(\lambda+\mu-\alpha_\mathrm{h}) ]_\nabla \geq [ \nabla(\varpi_1+\varpi_2) \otimes \nabla(\varpi_1+\varpi_2) : \nabla(\varpi_1+\varpi_2) ]_\nabla = 2 , \]
	as claimed.
\end{proof}

Before giving the proof of Theorem \ref{thm:A2MF}, we record two corollaries of Lemma \ref{lem:A2VerlindesuminC0}.

\begin{Corollary} \label{cor:A2reflectionsmallC0MF}
	Let $\lambda = a\varpi_1+b\varpi_2 \in X^+$ and $\mu = a^\prime\varpi_1+b^\prime\varpi_2 \in X^+$ such that $\lambda+\mu \in \widehat{C}_0$.
	If the tensor product $L(\lambda) \otimes L(\mu)$ is multiplicity free then $0 \in \{a,b,a^\prime,b^\prime\}$.
\end{Corollary}
\begin{proof}
	If $\lambda=0$ or $\mu=0$ then the claim is trivially satisfied.
	Now suppose that $\lambda \neq 0$ and $\mu \neq 0$, and observe that $\lambda + \mu \in \widehat{C}_0$ implies that $\lambda \in C_0$ and $\mu \in C_0$, in particular $L(\lambda) = T(\lambda)$ and $L(\mu) = T(\mu)$.
	If $a,a^\prime,b,b^\prime \geq 1$ then
	\[ [ L(\lambda) \otimes L(\mu) : L(\lambda+\mu-\alpha_\mathrm{h}) ] \geq [ T(\lambda) \otimes T(\mu) : T(\lambda+\mu-\alpha_\mathrm{h}) ]_\oplus = c_{\lambda,\mu}^{\lambda+\mu-\alpha_\mathrm{h}} \geq 2 \]
	by Lemma \ref{lem:A2VerlindesuminC0}, and we conclude that $L(\lambda) \otimes L(\mu)$ is multiplicity free only if $0 \in \{a,b,a^\prime,b^\prime\}$.
\end{proof}

\begin{Corollary} \label{cor:A2reflectionsmallC1MF}
	Let $\lambda = a\varpi_1+b\varpi_2 \in C_1 \cap X$ and $\mu = a^\prime\varpi_1+b^\prime\varpi_2 \in X^+$ such that $\mu$ is reflection small with respect to $\lambda$.
	If $L(\lambda) \otimes L(\mu)$ is multiplicity free then $a^\prime=0$ or $b^\prime=0$ or $a+b=p-1$.
\end{Corollary}
\begin{proof}
	We assume that $a^\prime,b^\prime \geq 1$ and show that $L(\lambda) \otimes L(\mu)$ is multiplicity free only if $a+b=p-1$.
	Since $\mu$ is reflection small with respect to $\lambda$, we have $a+a'+b' \leq p-1$ and $b+a'+b' \leq p-1$ by Example \ref{ex:A2reflectionsmall} and $\lambda+\mu \in \widehat{C}_1$.
	The assumption that $a^\prime,b^\prime \geq 1$ further implies that $\lambda+\mu-\alpha_\mathrm{h} \in C_1$ and
	\begin{multline*}
		\qquad a'+b' \leq 2 \cdot (a'+b') - 2 \leq 2 \cdot (a'+b') - 2 + (a+b) - (p-1) \\ = (a+a'+b') + (b+a'+b') - (p+1) \leq p-3 , \qquad
	\end{multline*}
	whence $\mu \in C_0$.
	Consider the weights $\lambda_0 = s\Cdot\lambda \in C_0$ and $\nu = s \Cdot (\lambda + \mu - \alpha_\mathrm{h}) = \lambda_0 + w_0(\mu) + \alpha_\mathrm{h} \in C_0$ (note that $w_0 = s_{\alpha_\mathrm{h}}$), and observe that
	\[ \lambda_0 = (p-2-b) \cdot \varpi_1 + (p-2-a) \cdot \varpi_2 , \qquad \nu = (p-1-b-b^\prime) \cdot \varpi_1 + (p-1-a-a^\prime) \cdot \varpi_2 , \] where $a+a^\prime < p-1$ and $b+b^\prime<p-1$ because $\mu$ is reflection small with respect to $\lambda$.
	If $a+b>p-1$ then we have $\lambda_0+\alpha_\mathrm{h} \in \widehat{C}_0$, and using Lemmas \ref{lem:Verlindecoefficientsflipping} and \ref{lem:A2VerlindesuminC0}, we compute
	\[ c_{\lambda_0,\mu}^\nu = c_{\nu,-w_0(\mu)}^{\lambda_0} \geq 2 \]
	because $\lambda_0 = \nu - w_0(\mu) - \alpha_\mathrm{h}$.
	Now Lemma \ref{lem:MFboundVerlindecoeff} implies that $L(\lambda) \otimes L(\mu) = L(s\Cdot\lambda_0) \otimes L(\mu)$ is not multiplicity free, and we conclude that $L(\lambda) \otimes L(\mu)$ is multiplicity free only if $a+b=p-1$.
\end{proof}

Now we are ready to give the proof of our classification of pairs of simple $G$-modules whose tensor product is multiplicity free, for $G = \SL_3(\kk)$.

\begin{proof}[Proof of Theorem \ref{thm:A2MF}]
\label{proof:A2MF}
	Let $\lambda,\mu \in X_1 \setminus \{0\}$.
	If $\lambda$ and $\mu$ satisfy one of the conditions in Table \ref{tab:A2MF} then the tensor product $L(\lambda) \otimes L(\mu)$ is completely reducible by Theorem \ref{thm:A2CR}.
	Furthermore, the conditions imply all weight spaces are at most one-dimensional for one of the simple $G$-modules $L(\lambda)$ and $L(\mu)$ by Lemma \ref{lem:A2onedimensionalweightspaces}, and so $L(\lambda) \otimes L(\mu)$ is multiplicity free by Corollary \ref{cor:onedimensionalweightspacesMF}.
	
	Now suppose that $L(\lambda) \otimes L(\mu)$ is multiplicity free.
	Then $L(\lambda) \otimes L(\mu)$ is completely reducible by Lemma \ref{lem:MFimpliesCR}, and so $\lambda$ and $\mu$ satisfy one of the conditions in Table \ref{tab:A2CR} by Theorem \ref{thm:A2CR}.
	As the cases 1--3 and 1*--3* in Table \ref{tab:A2CR} match the cases 1--3 and 1*--3* in Table \ref{tab:A2MF}, it remains to consider the weights $\lambda$ and $\mu$ such that $\lambda \in C_0 \cup C_1$ and $\mu$ is reflection small with respect to $\lambda$ (i.e.\ case 4 in Table \ref{tab:A2CR}).
	
	If $\lambda \in C_0$ and $\mu$ is reflection small with respect to $\lambda$ then $\lambda + \mu \in \widehat{C}_0$, and by Corollary \ref{cor:A2reflectionsmallC0MF}, we are in one of the cases 4a or 4a* in Table \ref{tab:A2MF}.
	If $\lambda \in C_1$ and $\mu$ is reflection small with respect to $\lambda$ then we are in one of the cases 4a, 4a* or 4b in Table \ref{tab:A2MF}, by Corollary \ref{cor:A2reflectionsmallC1MF}.
\end{proof}

\section{Type \texorpdfstring{$\mathrm{B}_2$}{B2}}
\label{sec:B2}

In this section, we consider the group $G = \mathrm{Sp}_4(\kk)$ and give a complete classification of the pairs of simple $G$-modules whose tensor product is completely reducible or multiplicity free.
We start by setting up our notation for the affine Weyl group and alcove geometry for $G$ in Subsection \ref{subsec:B2setup}, then we state the main results in Subsection \ref{subsec:B2results}.
In Subsection \ref{subsec:B2Weylmodulestiltingmodules}, we recall some well-known results about the submodule structure of certain Weyl modules and tilting modules.
The proofs of the main theorems are split into sufficient and necessary conditions for complete reducibility (in Subsections \ref{subsec:B2proofsCRsufficient} and \ref{subsec:B2proofsCRnecessary}) and multiplicity freeness (in Subsections \ref{subsec:B2proofsMFsufficient} and \ref{subsec:B2proofsMFnecessary}) of tensor products of simple $G$-modules.

\subsection{Setup}
\label{subsec:B2setup}

The positive roots for $G = \mathrm{Sp}_4(\kk)$ can be written as $\Phi^+ = \{ \alpha_1 , \alpha_2 , \alpha_\mathrm{hs} , \alpha_\mathrm{h} \}$, where $\alpha_1$ is the long simple root, $\alpha_2$ is the short simple root, $\alpha_\mathrm{hs} = \alpha_1 + \alpha_2$ is the highest short root and $\alpha_\mathrm{h} = \alpha_1 + 2 \alpha_2$ is the highest root.
The weight lattice $X$ is spanned by the fundamental dominant weights $\varpi_1$ and $\varpi_2$.
The weight $\varpi_2$ is minuscule, and we have
\[ \alpha_1 = 2 \varpi_1 - 2 \varpi_2 , \qquad \alpha_2 = - \varpi_1 + 2 \varpi_2 , \qquad \alpha_\mathrm{hs} = \varpi_1 , \qquad \alpha_\mathrm{h} = 2 \varpi_2 . \]
We write $S_\mathrm{fin} = \{t,u\}$ for the set of simple reflections in $W_\mathrm{fin}$, with $t = s_{\alpha_1}$ and $u = s_{\alpha_2}$.
The affine simple reflection is denoted by $s = s_0 = t_{\alpha_\mathrm{hs}} s_{\alpha_\mathrm{hs}}$, so that $S_\mathrm{aff} = \{ s,t,u \}$ is the set of simple reflections in $W_\mathrm{aff}$.
We fix labelings for certain alcoves as in the figure below, so that
\begin{gather*}
	C_1 = s \Cdot C_0 , \qquad C_2 = st \Cdot C_0 , \qquad C_3 = stu \Cdot C_0 , \\
	C_{3a} = sts \Cdot C_0 , \qquad C_{4a} = stut \Cdot C_0 , \qquad C_{4b} = stus \Cdot C_0 .
\end{gather*}

\begin{center}
\begin{tikzpicture}[scale=1.5]
	\draw[thick] (0,0) -- (0,3);
	\draw[thick] (0,0) -- (3,3);
	\draw[thick] (0,1) -- (1,1);
	\draw[thick] (1,1) -- (0,2);
	\draw[thick] (1,1) -- (1,3);
	\draw[thick] (0,2) -- (2,2);
	\draw[thick] (0,2) -- (1,3);
	\draw[thick] (2,2) -- (1,3);
	\draw[thick] (2,2) -- (2,3);
	\draw[thick] (0,3) -- (3,3);
	
	\node at (.325,.675) {$0$};
	\node at (.325,1.325) {$1$};
	\node at (.675,1.675) {$2$};
	\node at (.675,2.325) {$3$};
	\node at (.325,2.675) {$4a$};
	\node at (1.325,1.675) {$3a$};
	\node at (1.325,2.325) {$4b$};
\end{tikzpicture}
\end{center}
Note that all $p$-restricted weights belong to the upper closure of one of the alcoves $C_0$, $C_1$, $C_2$ and $C_3$.
Alternatively, these alcoves can be described as follows:
\begin{align*}
	C_0 & = \{ a \varpi_1 + b \varpi_2 \mid a > -1 , ~ b > -1 , ~ 2a+b < p-3 \} , \\
	C_1 & = \{ a \varpi_1 + b \varpi_2 \mid b > -1 , ~ 2a+b > p-3 , ~ a+b < p-2 \} , \\
	C_2 & = \{ a \varpi_1 + b \varpi_2 \mid a+b > p-2 , ~ b < p-1 , ~ 2a+b < 2p-3 \} , \\
	C_3 & = \{ a \varpi_1 + b \varpi_2 \mid 2a+b > 2p-3 , ~ a < p-1 , ~ b < p-1 \} .
\end{align*}
The wall separating two alcoves $C_x$ and $C_y$ is denoted by $F_{x,y} = F_{y,x}$, for $x,y \in \{0,1,2,3,3a,4a,4b\}$.
In the following, we give a concrete example of what the reflection smallness condition from Section \ref{sec:alcovegeometryCR} means for a weight in one of the alcoves $C_0$, $C_1$, $C_2$ and $C_3$.

\begin{Example} \label{ex:B2reflectionsmall}
	Let $\lambda,\mu \in X^+$ and write $\lambda = a \varpi_1 + b \varpi_2$ and $\mu = a^\prime \varpi_1 + b^\prime \varpi_2$.
	Then we have
	\[ W_\mathrm{fin} \mu = \big\{ c \varpi_1 + d \varpi_2 \mathrel{\big|} \pm (c,d) \in \{ (a^\prime,b^\prime) , (a^\prime+b^\prime,-b^\prime) , (-a^\prime,2a^\prime+b^\prime) , (-a^\prime-b^\prime,2a^\prime+b^\prime) \} \big \} . \]
	The alcove $C_0$ has a unique wall $F_{0,1} = H_{\alpha_\mathrm{hs},1}$ that belongs to its upper closure, and so if $\lambda \in \widehat{C}_0$ then $\mu$ is reflection small with respect to $\lambda$ if and only if $(\lambda+w\mu,\alpha_\mathrm{hs}^\vee) \leq p-3$ for all $w \in W_\mathrm{fin}$.
	Using the description of $W_\mathrm{fin}\mu$ above, one easily verifies that this is equivalent to the condition
	\[ 2a + b + 2a^\prime + b^\prime \leq p-3 . \]
	Similarly, the alcove $C_1$ has a unique wall $F_{1,2} = H_{\alpha_\mathrm{h},1}$ that belongs to its upper closure, and so if $\lambda \in \widehat{C}_1$ then $\mu$ is reflection small with respect to $\lambda$ if and only if $(\lambda+w\mu,\alpha_\mathrm{h}^\vee) \leq p-2$ for all $w \in W_\mathrm{fin}$.
	This is easily seen to be equivalent to the condition
	\[ a+b+a^\prime+b^\prime \leq p-2 . \]
	The alcove $C_2$ has two walls $F_{2,3} = H_{\alpha_\mathrm{hs},2}$ and $F_{2,3a} = H_{\alpha_2,1}$ that belong to its upper closure, and so for $\lambda \in \widehat{C}_2$, the weight $\mu$ is reflection small with respect to $\lambda$ if and only if
	\[ 2a + b + 2a^\prime + b^\prime \leq 2p-3 \qquad \text{and} \qquad b + 2a^\prime + b^\prime \leq p-1 . \]
	Finally, for $\lambda \in \widehat{C}_3$, the weight $\mu$ is reflection small with respect to $\lambda$ if and only if
	\[ a + a^\prime + b^\prime \leq p-1 \qquad \text{and} \qquad b + 2a^\prime + b^\prime \leq p-1 . \]
\end{Example}

\begin{Remark} \label{rem:B2stabC0}
	The stabilizer $\Omega \coloneqq \Stab_{W_\mathrm{ext}}(C_0)$ is cyclic of order $2$, and it is generated by the affine reflection $\omega \coloneqq t_{\varpi_2} s_{\alpha_\mathrm{h}}$.
	For $\lambda \in X$ and $a,b \in \Z$ such that $\lambda = a\varpi_1+b\varpi_2$, we have
	\[ \omega\Cdot\lambda = a \varpi_1 + (p - 2a - b - 4) \cdot \varpi_2 . \]
\end{Remark}

\begin{Remark}\label{rem:B2weightsfundamentalmodules}
	Since $\varpi_2$ is minuscule, we have $L(\varpi_2) = \Delta(\varpi_2)$, and so $\dim L(\varpi_2)=4$ and the non-zero weight spaces of $L(\varpi_2)$ correspond to the weights $\{\varpi_2, \varpi_1-\varpi_2,-\varpi_1+\varpi_2,-\varpi_2\}$ by Weyl's character formula.
	If $p\geq 3$ then we further have $L(\varpi_1) = \Delta(\varpi_1)$, whence $\dim L(\varpi_1)=5$ and the non zero weight spaces of $L(\varpi_1)$ correspond to the weights $\{\varpi_1, -\varpi_1+2\varpi_2,0,\varpi_1-2\varpi_2, -\varpi_1\}$.
\end{Remark}

\subsection{Main results}
\label{subsec:B2results}

We are now ready to state our main results about complete reducibility and multiplicity freeness of tensor products for $G = \mathrm{Sp}_4(\kk)$.
The proofs will be split into necessary conditions and sufficient conditions; see Subsections \ref{subsec:B2proofsCRsufficient} and \ref{subsec:B2proofsCRnecessary} for the proofs of the complete reducibility theorem and Subsection \ref{subsec:B2proofsMFsufficient} and \ref{subsec:B2proofsMFnecessary} for the proofs of the multiplicity freeness theorem.
Also recall from Subsection \ref{subsec:prestricted} that it suffices to consider pairs of simple $G$-modules with $p$-restricted highest weights.

\begin{Theorem}\label{thm:B2CR}
	For $\lambda , \mu \in X_1 \setminus \{0\}$, the tensor product $L(\lambda) \otimes L(\mu)$ is completely reducible if and only if $\lambda$ and $\mu$ satisfy one of the conditions in Table \ref{tab:B2CR}, up to interchanging $\lambda$ and $\mu$.
	\begin{table}[htbp]
	\centering
	\caption{Weights $\lambda,\mu \in X_1 \setminus \{ 0 \}$ such that $L(\lambda) \otimes L(\mu)$ is completely reducible, for $G$ of type $\mathrm{B}_2$.}
	\label{tab:B2CR}
	\begin{tabular}{|l||l|}
	\hline
	$\mathbf{\#}$ & \textbf{conditions} \\
	\hline\hline
	\\[-1.05em]
	1 & $p \geq 3$ and $\{\lambda, \mu\}\subseteq \{\frac{p-1}{2}\varpi_1,\frac{p-3}{2}\varpi_1+\varpi_2\}$\\
	\\[-1.05em]
	\hline
	2 & $\lambda=(p-1)\varpi_1$ and $\mu=\varpi_2$\\
	\hline
	3 & $\lambda=(p-1)\varpi_2$ and $\mu=\varpi_1$\\
	\hline
	4 & $p \neq 3$ and $\lambda=(p-2)\varpi_1+\varpi_2$ and $\mu=\varpi_1$\\
	\hline
	5 & $p \geq 3$ and $\lambda=(p-2)\varpi_2$ and $\mu=\varpi_2$\\
	\hline
	6 & $p \geq 3$ and $\lambda=(p-2)\varpi_1$ and $\mu=\varpi_2$\\
	\hline
	7 & $p \geq 5$ and $\lambda=(p-3)\varpi_2$ and $\mu=\varpi_1$\\
	\hline
	8 & $\lambda\in C_0\cup C_1\cup C_2\cup C_3$ and $\mu$ is reflection small with respect to $\lambda$\\
	\hline
	\end{tabular}
	\end{table}
\end{Theorem}

Recall from Proposition \ref{prop:CRandMFchar0impliesMF} that for $\lambda,\mu \in X^+$ such that $L(\lambda) \otimes L(\mu)$ is completely reducible and the $G_\C$-module $L_\C(\lambda) \otimes L_\C(\mu)$ is multiplicity free, the tensor product $L(\lambda) \otimes L(\mu)$ is also multiplicity free.
In order to state and prove our multiplicity freeness theorem, we will make recourse to the classification of pairs of simple $G_\C$ modules whose tensor product is multiplicity free, due to J.\ Stembridge; see Theorem 1.1B in \cite{Stembridge}.
In fact, J.\ Stembridge has classified the pairs of simple modules whose tensor product is multiplicity free for all complex simple algebraic groups.
We will only need his results for the group $G = \mathrm{Sp}_4(\kk)$, and these can be summarized as follows:

\begin{Theorem}\label{thm:StembridgeB2}
	For $\lambda , \mu \in X^+ \setminus \{0\}$, the tensor product $L_{\C}(\lambda) \otimes L_{\C}(\mu)$ is multiplicity free if and only if $\lambda$ and $\mu$ satisfy one of the conditions in Table \ref{tab:B2MFchar0}, up to interchanging $\lambda$ and $\mu$.
	\begin{table}[htbp]
	\centering
	\caption{Weights $\lambda,\mu \in X_1 \setminus \{ 0 \}$ such that $L_\C(\lambda) \otimes L_\C(\mu)$ is multiplicity free, for $G$ of type $\mathrm{B}_2$.}
	\label{tab:B2MFchar0}
	\begin{tabular}{|l||l|}
		\hline
		$\mathbf{\#}$ & \textbf{conditions} \\
		\hline\hline
		1a& $\mu=\varpi_1$\\
		\hline
		1b& $\mu=\varpi_2$\\
		\hline
		2a& $\lambda=a\varpi_1$ and $\mu=a'\varpi_1$\\
		\hline
		2b&$\lambda=b\varpi_2$ and $\mu=b'\varpi_2$\\
		\hline
		3&$\lambda=a\varpi_1$ and $\mu=b'\varpi_2$\\
		\hline
		4& $\lambda=a\varpi_1+\varpi_2$ and $\mu=a'\varpi_1$\\
		\hline
	\end{tabular}
	\end{table}
\end{Theorem}

Now we can state our multiplicity freeness theorem for $G = \mathrm{Sp}_4(\kk)$.

\begin{Theorem}\label{thm:B2MF}
	For $\lambda , \mu \in X_1 \setminus \{0\}$, the tensor product $L(\lambda) \otimes L(\mu)$ is multiplicity free if and only if $\lambda$ and $\mu$ satisfy one of the conditions in Table \ref{tab:B2MF}, up to interchanging $\lambda$ and $\mu$.
	
	\begin{table}[ht]
	\centering
	\caption{Weights $\lambda,\mu \in X_1 \setminus \{ 0 \}$ such that $L(\lambda) \otimes L(\mu)$ is multiplicity free, for $G$ of type $\mathrm{B}_2$.}
	\label{tab:B2MF}
	\begin{tabular}{|l||p{15.5cm}|}
		\hline
		$\mathbf{\#}$ & \textbf{conditions} \\
		\hline\hline
		\\[-1.05em]
		1 & $p \geq 3$ and $\{\lambda, \mu\}\subseteq \{\frac{p-1}{2}\varpi_1,\frac{p-3}{2}\varpi_1+\varpi_2\}$\\
		\\[-1.05em]
		\hline
		2 & $\lambda=(p-1)\varpi_1$ and $\mu=\varpi_2$\\
			\hline
			3 & $\lambda=(p-1)\varpi_2$ and $\mu=\varpi_1$\\
			\hline
			4 & $p \neq 3$ and $\lambda=(p-2)\varpi_1+\varpi_2$ and $\mu=\varpi_1$\\
			\hline
			5 & $p \geq 3$ and $\lambda=(p-2)\varpi_2$ and $\mu=\varpi_2$\\
			\hline
			6 & $p \geq 3$ and $\lambda=(p-2)\varpi_1$ and $\mu=\varpi_2$\\
			\hline
			7 & $p \geq 5$ and $\lambda=(p-3) \varpi_2$ and $\mu=\varpi_1$\\
			\hline
			8a & $\lambda\in C_0$, $\mu$ is reflection small with respect to $\lambda$ and $L_{\C}(\lambda)\otimes L_{\C}(\mu)$ is multiplicity free as a $G_{\C}$-module (cf.\ Table \ref{tab:B2MFchar0})\\
			\hline
			8b & $\lambda\in C_1\cup C_2\cup C_3$ and $\mu\in\{\varpi_1,\varpi_2\}$ is reflection small with respect to $\lambda$\\
			\hline
			8c & $\lambda=a\varpi_1+\varpi_2\in C_1$ and $\mu=a'\varpi_1$ is reflection small with respect to $\lambda$\\
			\hline
			8c\om & $\lambda=a\varpi_1+b\varpi_2\in C_1$ with $2a+b=p-1$ and $\mu=a'\varpi_1$ is reflection small with respect to $\lambda$\\
			\hline
			8d &  $\lambda=a\varpi_1\in C_1$, $\mu=a'\varpi_1+b'\varpi_2$ is reflection small with respect to $\lambda$ and $b'\in\{0,1\}$\\
			\hline
			8d\om & $\lambda=a\varpi_1+b\varpi_2\in C_1$ with $2a+b=p-2$, $\mu=a'\varpi_1+b'\varpi_2$ is reflection small with respect to $\lambda$ and $b'\in\{0,1\}$\\
			\hline
			8e&$\lambda=a\varpi_1\in C_1$ and $\mu=b'\varpi_2$ is reflection small with respect to $\lambda$\\
			\hline
			8e\om & $\lambda=a\varpi_1+b\varpi_2\in C_1$ with $2a+b=p-2$ and $\mu=b'\varpi_2$ is reflection small with respect to $\lambda$ \\
			\hline
			\\[-1.05em]
			8f&$\lambda=\frac{p-1}{2}\varpi_1$ and $\mu$ is reflection small with respect to $\lambda$\\
			\\[-1.05em]
			\hline
			\\[-1.05em]
			8f\om &$\lambda=\frac{p-3}{2}\varpi_1+\varpi_2$ and $\mu$ is reflection small with respect to $\lambda$\\
			\\[-1.05em]
			\hline
			8g & $\lambda=a\varpi_1+b\varpi_2\in C_2$ with $a+b=p-1$ and $\mu=a'\varpi_1$ is reflection small with respect to $\lambda$\\
			\hline
			8h & $\lambda=a\varpi_1+b\varpi_2\in C_2$ with $a+b=p-1$ and $\mu=b'\varpi_2$ is reflection small with respect to $\lambda$\\
			\hline
			8i & $\lambda=a\varpi_1+b\varpi_2\in C_3$ with $2a+b=2p-2$ and $\mu=b'\varpi_2$ is reflection small with respect to $\lambda$\\
			\hline
			8j & $\lambda=a\varpi_1+b\varpi_2\in C_3$ with $2a+b=2p-2$ and $\mu=a'\varpi_1+b'\varpi_2$ is reflection small with respect to $\lambda$ with $b'\in \{0,1\}$\\
			\hline
			8k & $\lambda=a\varpi_1+b\varpi_2\in C_3$ with $2a+b=2p-1$ and $\mu=a'\varpi_1$ is reflection small with respect to $\lambda$\\
			\hline
		\end{tabular}
	\end{table}
\end{Theorem}

\subsection{Structure of Weyl modules and tilting modules}
\label{subsec:B2Weylmodulestiltingmodules}

In order to prove our main results, we will need some information about the submodule structure of certain Weyl modules and indecomposable tilting modules for $G = \mathrm{Sp}_4(\kk)$.
The structure of the Weyl modules with $p$-restricted highest weights has been determined by J.C.\ Jantzen in \cite[Section 7]{JantzenKontravarianteFormen}.
We first assume that $p \geq h = 4$, so that $C_0 \cap X$ is non-empty.
The Weyl modules $\Delta(w\Cdot\lambda)$, with $\lambda \in C_0 \cap X$ and $w \in \{ e , s , st , stu \}$, are all uniserial, with respective composition series given by $\Delta(\lambda) = [ L(\lambda) ]$ and
\[ \Delta(s\Cdot\lambda) = [ L(s\Cdot\lambda) , L(\lambda) ] , \qquad \Delta(st \Cdot \lambda) = [ L(st \Cdot\lambda) , L(s \Cdot\lambda) ] , \qquad \Delta(stu\Cdot\lambda) = [ L(stu \Cdot\lambda) , L(st \Cdot\lambda) ] . \]
Using the translation principle, it follows that the Weyl modules with $p$-singular $p$-restricted highest weights are simple, except for those with highest weights in $\widehat{C}_3 \setminus (C_3 \cup F_{3,4a})$ (note that this set is contained in the wall $F_{3,4b}$).
The weights in $\widehat{C}_3 \setminus (C_3 \cup F_{3,4a})$ are of the form $stu\Cdot \nu$ for $\nu \in \overline{C}_0$ with $\Stab_{W_\mathrm{aff}}(\nu) = \{ e , s \}$, and the Weyl module $\Delta(stu\Cdot\nu)$ is uniserial, with composition series
\[ \Delta(stu\Cdot\nu) = [ L(stu\Cdot\nu) , L(st\Cdot\nu) ] . \]
Furthermore, we have $T(\lambda) = \Delta(\lambda) = L(\lambda)$ for $\lambda \in C_0 \cap X$, and the tilting modules with highest weights in the alcoves $C_1$, $C_2$ and $C_3$ can be obtained by applying translation functors to the simple tilting modules with highest weights on the walls $F_{0,1}$, $F_{1,2}$ and $F_{2,3}$, respectively; see Section II.E.11 in \cite{Jantzen}.
Using elementary properties of translation functors (Sections II.7.19--20 in \cite{Jantzen}) and the structure of the Weyl modules described above, one sees that the tilting module $T(s\Cdot\lambda)$ is uniserial with composition series $T(s\Cdot\lambda) = [ L(\lambda) , L(s\Cdot\lambda) , L(\lambda) ]$;
the structure of the indecomposable tilting modules with $p$-regular highest weights in $C_2$ and $C_3$ is displayed in the following diagrams, where we replace a simple $G$-module $L(w\Cdot\lambda)$ by the label $w \in W_\mathrm{aff}^+$:
\[ T(st\Cdot\lambda) = 
	\begin{tikzpicture}[scale=.6,baseline={([yshift=-.5ex]current bounding box.center)}]
	\node (C1) at (0,0) {\small $s$};
	\node (C2) at (-1.5,1.5) {\small $st$};
	\node (C3) at (1.5,1.5) {\small $e$};
	\node (C4) at (0,3) {\small $s$};
	
	\draw (C1) -- (C2);
	\draw (C1) -- (C3);
	\draw (C2) -- (C4);
	\draw (C3) -- (C4);
	\end{tikzpicture} \hspace{2cm}
	T(stu\Cdot\lambda) = 
	\begin{tikzpicture}[scale=.6,baseline={([yshift=-.5ex]current bounding box.center)}]
	\node (C1) at (0,0) {\small $st$};
	\node (C2) at (-1.5,1.5) {\small $stu$};
	\node (C3) at (1.5,1.5) {\small $s$};
	\node (C4) at (0,3) {\small $st$};
	
	\draw (C1) -- (C2);
	\draw (C1) -- (C3);
	\draw (C2) -- (C4);
	\draw (C3) -- (C4);
	\end{tikzpicture} \]
If $p=3$ then the Weyl module $\Delta(\varpi_1+2\varpi_2)$ is uniserial with composition series $[ L(\varpi_1+2\varpi_2) , L(2\varpi_2) ]$, and for $\lambda \in X_1 \setminus \{ \varpi_1 + 2\varpi_2 \}$, we have $\Delta(\lambda) = L(\lambda)$.
If $p=2$ then the Weyl module $\Delta(\varpi_1)$ is uniserial with composition series $[ L(\varpi_1) , L(0) ]$, and for $\lambda \in X_1 \setminus \{ \varpi_1 \}$, we have $\Delta(\lambda) = L(\lambda)$.

\subsection{Proofs: complete reducibility, sufficient conditions}
\label{subsec:B2proofsCRsufficient}

In this subsection, we show that the tensor product $L(\lambda) \otimes L(\mu)$ is completely reducible for all pairs of weights $\lambda , \mu \in X_1$ that satisfy one of the conditions from Table \ref{tab:B2CR}.
We consider the different conditions in turn and summarize the proof at the end of the subsection (see page \pageref{proof:B2CRsufficient}).

We first establish that the tensor product $L(\lambda) \otimes L(\mu)$ is completely reducible for all $\lambda,\mu \in X_1$ that satisfy one of the conditions 2--7 in Table \ref{tab:B2CR}, in the six propositions below.
In all of these cases, we can explicitly write the character of $L(\lambda) \otimes L(\mu)$ as a sum of characters of pairwise non-isomorphic simple $G$-modules using Proposition \ref{prop:productWeylcharacters} and the results in Subsection \ref{subsec:B2Weylmodulestiltingmodules}, so $L(\lambda) \otimes L(\mu)$ is multiplicity free, and also completely reducible by Lemma \ref{lem:MFimpliesCR}.
We only give a detailed proof for Proposition \ref{prop:B2MFCRp-1_1} below (i.e.\ condition 2 in Table \ref{tab:B2CR}), since the proofs for the conditions 3--7 are completely analogous.

\begin{Proposition}\label{prop:B2MFCRp-1_1}
	Let $\lambda=(p-1) \cdot \varpi_1$ and $\mu=\varpi_2$.
	Then $L(\lambda)\otimes L(\mu)$ is completely reducible and multiplicity free.
	More precisely, if $p = 2$ then we have $L(\lambda) \otimes L(\mu) \cong L(\varpi_1+\varpi_2)$, and if $p \geq 3$ then we have
	\[ L(\lambda) \otimes L(\mu) = L\big( (p-1) \cdot \varpi_1 + \varpi_2 \big) \oplus L\big( (p-2) \cdot \varpi_1 + \varpi_2 \big) . \]
\end{Proposition}
\begin{proof}
	For $p=2$, we have $\lambda = \varpi_1$, and using Proposition \ref{prop:productWeylcharacters} and the results in Subsection \ref{subsec:B2Weylmodulestiltingmodules}, we compute that
	\[ \ch\big( L(\varpi_1) \otimes L(\varpi_2) \big) = \big( \chi(\varpi_1) - \chi(0) \big) \cdot \chi(\varpi_2) = \chi(\varpi_1 + \varpi_2) = \ch L(\varpi_1+\varpi_2) , \]
	so $L(\varpi_1) \otimes L(\varpi_2)$ is simple.
	For $p \geq 3$, we have $L(\lambda) = \Delta(\lambda)$ and $L(\mu) = \Delta(\mu)$ by Subsection \ref{subsec:B2Weylmodulestiltingmodules}, and using Proposition \ref{prop:productWeylcharacters} and Remark \ref{rem:B2weightsfundamentalmodules}, we compute that
	\begin{align*}
		\ch\big( L(\lambda) \otimes L(\mu) \big) & = \chi\big( (p-1) \cdot \varpi_1 \big) \cdot \chi(\varpi_2) \\
		& \hspace{-1.5cm} = \chi\big( (p-1) \cdot \varpi_1 + \varpi_2 \big) + \chi( p \varpi_1 - \varpi_2 ) + \chi\big( (p-2) \cdot \varpi_1 + \varpi_2 \big) + \chi\big( (p-1) \cdot \varpi_1 - \varpi_2 \big) \\
		& \hspace{-1.5cm} = \chi\big( (p-1) \cdot \varpi_1 + \varpi_2 \big) + \chi\big( (p-2) \cdot \varpi_1 + \varpi_2 \big) .
	\end{align*}
	The Weyl modules with highest weights $(p-1) \cdot \varpi_1 + \varpi_2 \in F_{3,4a}$ and $(p-2) \cdot \varpi_1 + \varpi_2 \in F_{2,3}$ are simple, again by Subsection \ref{subsec:B2Weylmodulestiltingmodules}, and it follows that
	\[ \ch\big( L(\lambda) \otimes L(\mu) \big) = \ch L\big( (p-1) \cdot \varpi_1 + \varpi_2 \big) + \ch L\big( (p-2) \cdot \varpi_1 + \varpi_2 \big) . \]
	In particular, $L(\lambda) \otimes L(\mu)$ is multiplicity free, and also completely reducible by Lemma \ref{lem:MFimpliesCR}.
\end{proof}

\begin{Proposition}\label{prop:B2MFCRp-1_2}
	Let $\lambda = (p-1) \cdot \varpi_2$ and $\mu = \varpi_1$.
	Then $L(\lambda)\otimes L(\mu)$ is completely reducible and multiplicity free.
	More precisely, if $p = 2$ then we have $L(\lambda) \otimes L(\mu) \cong L(\varpi_1+\varpi_2)$, and if $p \geq 3$ then we have
	\[ L(\lambda) \otimes L(\mu) \cong L\big( \varpi_1+(p-1) \cdot \varpi_2 \big) \oplus L\big( (p-1) \cdot \varpi_2 \big) \oplus L\big( \varpi_1 + (p-3) \cdot \varpi_2 \big) . \]
\end{Proposition}

\begin{Proposition}\label{prop:B2MFCRp-21}
	Let $p \neq 3$ and $\lambda = (p-2) \cdot \varpi_1 + \varpi_2$ and $\mu=\varpi_1$.
	Then $L(\lambda)\otimes L(\mu)$ is completely reducible and multiplicity free.
	More precisely, if $p=2$ then we have $L(\lambda) \otimes L(\mu) \cong L(\varpi_1+\varpi_2)$, and if $p \geq 5$ then we have
	\[ L(\lambda) \otimes L(\mu) \cong L\big( (p-1) \cdot \varpi_1 + \varpi_2 \big) \oplus L\big( (p-3) \cdot \varpi_1 + 3 \varpi_2 \big) \oplus L\big( (p-2) \cdot \varpi_1 + \varpi_2 \big) \oplus L\big( (p-2) \varpi_1 \big) . \]
\end{Proposition}

\begin{Proposition}\label{prop:B2MFCRp-2_2}
	Let $p \geq 3$ and $\lambda = (p-2) \varpi_2$ and $\mu=\varpi_2$.
	Then $L(\lambda)\otimes L(\mu)$ is completely reducible and multiplicity free.
	More precisely, we have
	\[L(\lambda)\otimes L(\mu)\cong L\big((p-1)\varpi_2\big)\oplus L\big(\varpi_1+(p-3)\varpi_2\big)\oplus L\big((p-3)\varpi_2\big).\]
\end{Proposition}

\begin{Proposition}\label{prop:B2MFCRp-2_1}
	Let $p \geq 3$ and $\lambda = (p-2) \cdot \varpi_1$ and $\mu=\varpi_2$.
	Then $L(\lambda)\otimes L(\mu)$ is completely reducible and multiplicity free.
	More precisely, we have
	\[ L(\lambda) \otimes L(\mu) \cong L\big( (p-2) \cdot \varpi_1 + \varpi_2 \big) \oplus L\big( (p-3) \cdot \varpi_1 + \varpi_2 \big) . \]
\end{Proposition}

\begin{Proposition}\label{prop:B2MFCRp-3_2}
	Let $p \geq 5$ and $\lambda=(p-3) \cdot \varpi_2$ and $\mu=\varpi_1$.
	Then $L(\lambda)\otimes L(\mu)$ is completely reducible and multiplicity free.
	More precisely, we have
	\[ L(\lambda) \otimes L(\mu) \cong L\big( \varpi_1 + (p-3) \cdot \varpi_2 \big) \oplus L\big( (p-3) \cdot \varpi_2 \big) \oplus L\big( \varpi_1 + (p-5) \cdot \varpi_2 \big) . \]
\end{Proposition}

\vspace{.6cm}

In order to show that $L(\lambda) \otimes L(\mu)$ is completely reducible for $p \geq 3$ and $\lambda,\mu \in \{ \frac{p-1}{2} \cdot \varpi_1 , \frac{p-3}{2} \cdot \varpi_1 + \varpi_2 \}$ (as in case 1 in Table \ref{tab:B2CR}), we need to establish some preliminary results about the dimensions of weight spaces in $L(\lambda)$ and $L(\mu)$.

\begin{Lemma} \label{lem:B2_weightspacedimensions_api1}
	Let $a \in \Z_{\geq 0}$ and $\lambda = a \varpi_1$.
	For all $\mu \in X^+$ with $\mu \leq \lambda$, we have $\mu = \lambda - d \varpi_1 - e \alpha_1$ for some $d,e \geq 0$ with $d \leq a$ and $2e \leq a-d$.
	Furthermore, we have $\dim \Delta(\lambda)_\mu = \lfloor  \frac{d}{2} \rfloor + 1$.
\end{Lemma}
\begin{proof}
	Let $c,d \geq 0$ such that $\mu = \lambda - c \alpha_1 - d \alpha_2$, and note that $\mu = \lambda - c \alpha_1 - d \alpha_2 = \lambda - d \varpi_1 - (c-d) \cdot \alpha_1$.
	As $\mu$ is dominant, we have $a - d - 2 \cdot (c-d) \geq 0$ and $2 \cdot (c-d) \geq 0$, and with $e \coloneqq c-d$, it follows that $e \geq 0$, $d \leq a$ and $2e \leq a-d$, as claimed.
	For the claim about dimensions of weight spaces, see Exercise 16.11 in \cite{FultonHarris} or Theorem 4.1 in \cite{CaglieroTirao}.
\end{proof}

\begin{Lemma} \label{lem:B2_weightspacedimensions_api1+pi2}
	Let $a \in \Z_{\geq 0}$ and $\lambda = a \varpi_1 + \varpi_2$.
	For all $\mu \in X^+$ with $\mu \leq \lambda$, we have $\mu = \lambda - d \varpi_1 - e \alpha_1$ for some $d,e \geq 0$ with $d \leq a$ and $2e \leq a-d$.
	Furthermore, we have $\dim \Delta(\lambda)_\mu = d$.
\end{Lemma}
\begin{proof}
	Let $c,d \geq 0$ such that $\mu = \lambda - c \alpha_1 - d \alpha_2$, and note that $\mu = \lambda - c \alpha_1 - d \alpha_2 = \lambda - d \varpi_1 - (c-d) \cdot \alpha_1$.
	As $\mu$ is dominant, we have $a - d - 2 \cdot (c-d) \geq 0$ and $1 + 2 \cdot (c-d) \geq 0$, and with $e \coloneqq c-d$, it follows that $e \geq 0$, $d \leq a$ and $2e \leq a-d$, as claimed.
	Using Proposition \ref{prop:productWeylcharacters}, one further checks that
	\[ \chi(a\varpi_1) \cdot \chi(\varpi_2) = \chi(\lambda) + \chi\big( (a-1) \varpi_1 + \varpi_2 \big) , \]
	and the claim about dimensions of weight spaces easily follows from Lemma \ref{lem:B2_weightspacedimensions_api1} by induction on $a$.
	Alternatively, the claim also follows from Theorem 4.1 in \cite{CaglieroTirao}.
\end{proof}

The following result is due to I.\ Suprunenko and A.\ Zalesskiĭ \cite{SuprunenkoZalesskiiSymplectic} (for symplectic groups of arbitrary rank).
Since their article is only available in Russian, we provide a proof here for $G = \mathrm{Sp}_4(\kk)$.

\begin{Corollary} \label{cor:B2onedimensionalweightspaces}
	Suppose that $p \geq 3$.
	Then all weight spaces of $L\big(\frac{p-1}{2}\varpi_1\big)$ and of $L\big(\frac{p-3}{2}\varpi_1+\varpi_2\big)$ are at most one-dimensional.
\end{Corollary}
\begin{proof}
	For $p=3$, we have $L(\varpi_1) = \Delta(\varpi_1)$ and $L(\varpi_2) = \Delta(\varpi_2)$ by Subsection \ref{subsec:B2Weylmodulestiltingmodules}, and the claim is easily verified using Weyl's character formula.
	For $p \geq 5$, we have
	\[ \ch L\big( \tfrac{p-1}{2}\varpi_1 \big) = \chi\big( \tfrac{p-1}{2}\varpi_1\big) - \chi\big( \tfrac{p-5}{2}\varpi_1 \big) , \qquad \ch L\big( \tfrac{p-3}{2}\varpi_1 + \varpi_2 \big) = \chi\big( \tfrac{p-3}{2}\varpi_1 + \varpi_2 \big) - \chi\big( \tfrac{p-5}{2} \varpi_1 + \varpi_2 \big) \]
	by Subsection \ref{subsec:B2Weylmodulestiltingmodules}, and the claim follows from Lemmas \ref{lem:B2_weightspacedimensions_api1} and \ref{lem:B2_weightspacedimensions_api1+pi2}.
\end{proof}

\begin{Remark} \label{rem:charactersimpleapi1}
	From the proofs of Lemma \ref{lem:B2_weightspacedimensions_api1} and Corollary \ref{cor:B2onedimensionalweightspaces}, it is straightforward to see that the character of the simple $G$-module of highest weight $\lambda = \frac{p-1}{2} \varpi_1$ is given by
	\[ \ch L(\lambda) = \sum_{ \substack{c,d \in \Z \\ |c| + |d| \leq \frac{p-1}{2} } } e^{c \varpi_1 + d \alpha_2} . \]
\end{Remark}

Now we are ready to show that $L(\lambda) \otimes L(\mu)$ is completely reducible for all weights $\lambda,\mu \in X_1$ that satisfy condition 1 in Table \ref{tab:B2CR}.

\begin{Proposition} \label{prop:B2MFpairsinC1}
	Suppose that $p \geq 3$ and let $\lambda,\mu \in \{ \frac{p-1}{2} \cdot \varpi_1 , \frac{p-3}{2} \cdot \varpi_1 + \varpi_2 \}$.
	Then the tensor product $L(\lambda) \otimes L(\mu)$ is multiplicity free and completely reducible.
\end{Proposition}
\begin{proof}
	First suppose that $p = 3$, so that $\lambda,\mu \in \{ \varpi_1 , \varpi_2 \}$.
	If at least one of $\lambda$ and $\mu$ is equal to $\varpi_2$ then the claim follows from Propositions \ref{prop:B2MFCRp-2_2} and \ref{prop:B2MFCRp-2_1}.
	For $\lambda = \mu = \varpi_1$, one easily computes (using Proposition \ref{prop:productWeylcharacters} and the results from Subsection \ref{subsec:B2Weylmodulestiltingmodules}) that
	\[ \ch\big( L(\varpi_1) \otimes L(\varpi_1) \big) = \chi(\varpi_1) \cdot \chi(\varpi_1) = \chi(2 \varpi_1) + \chi(2 \varpi_2) + \chi(0) = \ch L(2 \varpi_1) + \ch L(2 \varpi_2) + \ch L(0) . \]
	In particular, the tensor product $L(\varpi_1) \otimes L(\varpi_1)$ is multiplicity free, and also completely reducible by Lemma \ref{lem:MFimpliesCR}.
	Now suppose that $p \geq 5$ and let
	\[ \lambda \coloneqq \tfrac{p-1}{2} \cdot \varpi_1 \in C_1 , \qquad \mu \coloneqq \tfrac{p-3}{2} \cdot \varpi_1 + \varpi_2 \in C_1 . \]
	We first show that the tensor product $L(\lambda) \otimes L(\mu)$ is multiplicity free, by explicitly writing the character of $L(\lambda) \otimes L(\mu)$ as a sum of characters of pairwise non-isomorphic simple $G$-modules.
	(Then the complete reducibility of $L(\lambda) \otimes L(\mu)$ will follow from Lemma \ref{lem:MFimpliesCR}.)
	To that end, recall from Subsection \ref{subsec:B2Weylmodulestiltingmodules} that we have
	\[ \ch L(\mu) = \chi(\mu) - \chi(s\Cdot\mu) = \chi\big( \tfrac{p-3}{2} \varpi_1 + \varpi_2 \big) - \chi\big( \tfrac{p-5}{2} \varpi_1 + \varpi_2 \big) . \]
	By Proposition \ref{prop:productWeylcharacters} and Remark \ref{rem:charactersimpleapi1}, we further have
	\[ \ch L(\lambda) \cdot \chi(\mu) = \sum_{|c| + |d| \leq \frac{p-1}{2}} \chi(\mu + c \varpi_1 + d \alpha_2) . \]
	Recall from Subsection \ref{subsec:representationscharacters} that $\chi(w\Cdot\nu) = (-1)^{\ell(w)} \cdot \chi(\nu)$ for all $w \in W_\mathrm{fin}$ and $\nu \in X$, and $\chi(\nu) = 0$ if $(\nu,\alpha^\vee) = -1$ for some $\alpha \in \Pi$.
	For $c,d \in \Z$ with $|c| + |d| \leq \frac{p-1}{2}$ and $d<-1$, we have
	\begin{align*}
		\mu + c \varpi_1 + d \alpha_2 & = ( \tfrac{p-3}{2} +c ) \cdot \varpi_1 + \varpi_2 + d \alpha_2 \\
		& = u \Cdot \big( ( \tfrac{p-3}{2} +c ) \cdot \varpi_1 + \varpi_2 + (-d-2) \cdot \alpha_2 \big) \\
		& = u \Cdot \big( \mu + c \varpi_1 + (-d-2) \cdot \alpha_2 \big)
	\end{align*}
	and therefore
	\[ \chi( \mu + c \varpi_1 + d \alpha_2 ) = - \chi\big( \mu + c \varpi_1 + (-d-2) \cdot \alpha_2 \big) , \]
	where $|c| + |-d-2| \leq \frac{p-5}{2}$ and $-d-2 \geq 0$.
	Furthermore, for $c,d \in \Z$ with $|c| + |d| \leq \frac{p-1}{2}$ and $d = -1$, we have $( \mu + c \varpi_1 + d \alpha_2 , \alpha_2^\vee ) = -1$ and therefore $\chi(\mu + c \varpi_1 + d \alpha_2) = 0$, and we conclude that
	\begin{align*}
		\ch L(\lambda) \cdot \chi(\mu) & = \sum_{ \substack{ |c| + |d| = \frac{p-1}{2} \\ d \geq 0 } } \chi(\mu + c \varpi_1 + d \alpha_2) + \sum_{ \substack{ |c| + |d| = \frac{p-3}{2} \\ d \geq 0 } } \chi(\mu + c \varpi_1 + d \alpha_2) .
	\end{align*}
	This character formula can be rewritten as
	\begin{align*}
		\ch L(\lambda) \cdot \chi(\mu) & = \sum_{ 0 \leq d \leq \frac{p-3}{2} } \chi\big( (p-2) \cdot \varpi_1 + \varpi_2 - d \alpha_1 \big) + \sum_{ 0 \leq d \leq \frac{p-1}{2} } \chi\big( -\varpi_1 + (2d+1) \cdot \varpi_2 \big) \\
		& \hspace{2.4cm} + \sum_{ 0 \leq d \leq \frac{p-5}{2} } \chi\big( (p-3) \cdot \varpi_1 + \varpi_2 - d \alpha_1 \big) + \sum_{ 0 \leq d \leq \frac{p-3}{2} } \chi\big( (2d+1) \cdot \varpi_2 \big) ,
	\end{align*}
	and an analogous computation shows that
	\begin{align*}
		\ch L(\lambda) \cdot \chi(s\Cdot\mu) & = \sum_{ 0 \leq d \leq \frac{p-3}{2} } \chi\big( (p-3) \cdot \varpi_1 + \varpi_2 - d \alpha_1 \big) + \sum_{ 0 \leq d \leq \frac{p-1}{2} } \chi\big( -2\varpi_1 + (2d+1) \cdot \varpi_2 \big) \\
		& \hspace{1.75cm} + \sum_{ 0 \leq d \leq \frac{p-5}{2} } \chi\big( (p-4) \cdot \varpi_1 + \varpi_2 - d \alpha_1 \big) + \sum_{ 0 \leq d \leq \frac{p-3}{2} } \chi\big( - \varpi_1 + (2d+1) \cdot \varpi_2 \big) .
	\end{align*}
	Next observe that we have $\chi(-\varpi_1+a\varpi_2) = 0$ and $\chi(-2\varpi_1+a\varpi_2) = - \chi\big( (a-2) \cdot \varpi_2 \big)$ for all $a \in \Z$, and in particular $\chi(-2\varpi_1+\varpi_2)=\chi(-\varpi_2)=0$.
	Thus, we can further rewrite the character formulas above as follows:
	\begin{align*}
		\ch L(\lambda) \cdot \chi(\mu) & = \sum_{ 0 \leq d \leq \frac{p-3}{2} } \chi\big( (p-2) \cdot \varpi_1 + \varpi_2 - d \alpha_1 \big) \\
		& \hspace{2.4cm} + \sum_{ 0 \leq d \leq \frac{p-5}{2} } \chi\big( (p-3) \cdot \varpi_1 + \varpi_2 - d \alpha_1 \big) + \sum_{ 0 \leq d \leq \frac{p-3}{2} } \chi\big( (2d+1) \cdot \varpi_2 \big)  \\[.2cm]
	\end{align*}
	\begin{align*}
		\ch L(\lambda) \cdot \chi(s\Cdot\mu) & = \sum_{ 0 \leq d \leq \frac{p-3}{2} } \chi\big( (p-3) \cdot \varpi_1 + \varpi_2 - d \alpha_1 \big) - \sum_{ 0 \leq d \leq \frac{p-3}{2} } \chi\big( (2d+1) \cdot \varpi_2 \big) \\
		& \hspace{2.4cm} + \sum_{ 0 \leq d \leq \frac{p-5}{2} } \chi\big( (p-4) \cdot \varpi_1 + \varpi_2 - d \alpha_1 \big)
	\end{align*}
	By combining the two character formulas, we obtain
	\begin{align*}
	\ch\big( L(\lambda) \otimes L(\mu) \big) & = \ch L(\lambda) \cdot \chi(\mu) - \ch L(\lambda) \cdot \chi(s\Cdot\mu) \\
	& = \sum_{ 0 \leq d \leq \frac{p-3}{2} } \chi\big( (p-2) \cdot \varpi_1 + \varpi_2 - d \alpha_1 \big) + \chi\big( (p-2) \cdot \varpi_2 \big) - \chi\big( (p-2) \cdot \varpi_2 \big) \\
	& \hspace{2cm} + 2 \cdot \sum_{ 0 \leq d \leq \frac{p-3}{2} } \chi\big( (2d+1) \cdot \varpi_2 \big) - \sum_{ 0 \leq d \leq \frac{p-5}{2} } \chi\big( (p-4) \cdot \varpi_1 + \varpi_2 - d \alpha_1 \big) \\
	& = \chi\big( (p-2) \cdot \varpi_1 + \varpi_2 \big) + \sum_{ 0 \leq d \leq \frac{p-3}{2} } \chi\big( (2d+1) \cdot \varpi_2 \big) + R ,
	\end{align*}
	where
	\[ R \coloneqq \sum_{ 0 \leq d \leq \frac{p-5}{2} } \Big( \chi\big( (p-2) \cdot \varpi_1 + \varpi_2 - (d+1) \cdot \alpha_1 \big) - \chi\big( (p-4) \cdot \varpi_1 + \varpi_2 - d \alpha_1 \big) + \chi\big( (2d+1) \cdot \varpi_2 \big) \Big) . \]
	By Subsection \ref{subsec:B2Weylmodulestiltingmodules}, we have $\ch L(st\Cdot\nu) = \chi(st\Cdot\nu) - \chi(s\Cdot\nu) + \chi(\nu)$ for all weights $\nu \in C_0 \cap X$, and it follows that
	\[ R = \sum_{0 \leq d \leq \frac{p-5}{2}} \ch L\big( (p-1) \cdot \varpi_1 + \varpi_2 - (d+1) \cdot \alpha_1 \big) . \]
	Furthermore, we have $\chi\big( (p-2) \cdot \varpi_1 + \varpi_2 \big) = \ch L\big( (p-2) \cdot \varpi_1 + \varpi_2 \big)$ and $\chi(b \varpi_2) = \ch L(b \varpi_2)$ for $0 \leq b \leq p-1$, again by Subsection \ref{subsec:B2Weylmodulestiltingmodules}, and we conclude that
	\[ \ch\big( L(\lambda) \otimes L(\mu) \big) = \sum_{0 \leq d \leq \frac{p-3}{2}} \ch L\big( (p-2) \cdot \varpi_1 + \varpi_2 - d \alpha_1 \big) + \sum_{0 \leq d \leq \frac{p-3}{2}} \ch L\big( (2d+1) \cdot \varpi_2 \big) . \]
	In particular, the tensor product $L(\lambda) \otimes L(\mu)$ is multiplicity free, and also completely reducible by Lemma \ref{lem:MFimpliesCR}, as claimed.
	For later use, we note that all simple direct summands of $L(\lambda) \otimes L(\mu)$ have highest weights in $C_0 \cup C_2 \cup \{ (p-2) \cdot \varpi_1 + \varpi_2 , (p-2) \cdot \varpi_2 \}$.
	
	Now it remains to show that the tensor squares $L(\lambda) \otimes L(\lambda)$ and $L(\mu) \otimes L(\mu)$ are multiplicity free and completely reducible.
	To that end, observe that we have
	\[ L(\mu) \cong T_{s \Cdot \lambda}^{s \Cdot \mu} L(\lambda) = \pr_{s \Cdot \mu}\big( L(\varpi_2) \otimes L(\lambda) \big) , \qquad L(\lambda) \cong T_{s \Cdot \mu}^{s \Cdot \lambda} L(\mu) = \pr_{s \Cdot \lambda}\big( L(\varpi_2) \otimes L(\mu) \big) , \]
	whence $L(\mu)$ is a direct summand of $L(\varpi_2) \otimes L(\lambda)$ and $L(\lambda)$ is a direct summand of $L(\varpi_2) \otimes L(\mu)$.
	Thus, if $M$ is an indecomposable direct summand of one of the tensor squares $L(\mu) \otimes L(\mu)$ or $L(\lambda) \otimes L(\lambda)$ then $M$ is also a direct summand of $L(\varpi_2) \otimes L(\lambda) \otimes L(\mu)$, and it follows that $M$ is a direct summand of $L(\varpi_2) \otimes L(\nu)$ for some weight $\nu \in X^+$ such that $L(\nu)$ is a direct summand of $L(\lambda) \otimes L(\mu)$.
	We consider the different possibilities for $\nu$ in turn.
	\begin{enumerate}
		\item If $\nu \in C_0 \cup C_2$ then $L(\varpi_2) \otimes L(\nu)$ is completely reducible by Corollary \ref{cor:minusculetensorproductMFCR}, and it follows that $M$ is simple.
		\item If $\nu = (p-2) \cdot \varpi_2$ then $L(\varpi_2) \otimes L(\nu)$ is completely reducible by Proposition \ref{prop:B2MFCRp-2_2}, and it follows that $M$ is simple.
		 \item If $\nu = (p-2) \cdot \varpi_1 + \varpi_2$ then $L(\nu) = \Delta(\nu) = T(\nu)$ by Subsection \ref{subsec:B2Weylmodulestiltingmodules}, and using Lemma \ref{lem:minusculetensorproducttranslation} and standard properties of translation functors, we obtain 
		\[ L(\varpi_2) \otimes L(\nu) \cong T\big( (p-2) \cdot \varpi_1 + 2 \varpi_2 \big) \oplus L\big( (p-1) \cdot \varpi_1 \big) \oplus L\big( (p-2) \cdot \varpi_1 \big) . \]
		As $(p-2) \cdot \varpi_1 + 2 \varpi_2 = 2 \mu + \alpha_\mathrm{hs} = 2 \lambda + \alpha_2$, the tilting module $T\big( (p-2) \cdot \varpi_1 + 2 \varpi_2 \big)$ cannot be a direct summand of $L(\mu) \otimes L(\mu)$ or $L(\lambda) \otimes L(\lambda)$, and we conclude that $M$ is simple. 
	\end{enumerate}
	In all cases, the indecomposable direct summand $M$ of $L(\mu) \otimes L(\mu)$ or of $L(\lambda) \otimes L(\lambda)$ is simple.
	In particular, the tensor squares $L(\mu) \otimes L(\mu)$ and $L(\lambda) \otimes L(\lambda)$ are completely reducible.
	As all weight spaces of $L(\mu)$ and $L(\lambda)$ are at most one-dimensional by Corollary \ref{cor:B2onedimensionalweightspaces}, the tensor squares $L(\mu) \otimes L(\mu)$ and $L(\lambda) \otimes L(\lambda)$ are also multiplicity free by Corollary \ref{cor:onedimensionalweightspacesMF}.
\end{proof}

Using the preceding proposition, we can classify the indecomposable direct summands of $L(\lambda) \otimes L(\mu)$, for weights $\lambda,\mu \in C_1 \cap X$.
We will use this result in the next subsection to establish necessary conditions for the complete reducibility of $L(\lambda) \otimes L(\mu)$.

\begin{Proposition} \label{prop:B2directsummandsC1}
	Let $\lambda,\mu \in C_1 \cap X$ and let $M$ be an indecomposable direct summand of $L(\lambda) \otimes L(\mu)$.
	Then either $M$ is a tilting module or $M \cong L(\delta)$ for some $\delta \in C_2 \cap X$.
\end{Proposition}
\begin{proof}
	Recall that the assumption that $C_1 \cap X$ is non-empty implies that $p \geq 5$, and consider the weights $\nu = \frac{p-1}{2} \cdot \varpi_1 \in C_1$ and $\nu' = \frac{p-3}{2} \cdot \varpi_1 + \varpi_2 \in C_1$.
	By the translation principle, we have
	\[ L(\lambda) \cong T_{s\Cdot\nu}^{s\Cdot\lambda} L(\nu) \qquad \text{and} \qquad L(\mu) \cong T_{s\Cdot\nu'}^{s\Cdot\mu} L(\nu') , \]
	and it follows that $L(\lambda) \otimes L(\mu)$ is a direct summand of $L(\nu) \otimes L(\nu') \otimes T$ for some tilting module $T$.
	In particular, every indecomposable direct summand of $L(\lambda) \otimes L(\mu)$ also appears as a direct summand in $L(\nu) \otimes L(\nu') \otimes T(\gamma)$ for some $\gamma \in X^+$.
	Next observe that for $a,b \geq 0$ such that $\gamma = a \varpi_1 + b \varpi_2$, the tilting module $T(\gamma)$ is a direct summand of $T(\varpi_1)^{\otimes a} \otimes T(\varpi_2)^{\otimes b}$, and since
	\[ T(\varpi_2) \otimes T(\varpi_2) \cong T(2\varpi_2) \oplus T(\varpi_1) \oplus T(0) , \]
	we conclude that $T(\gamma)$ is also a direct summand of $T(\varpi_2)^{\otimes (2a+b)}$.
	Thus, every indecomposable direct summand of $L(\lambda) \otimes L(\mu)$ appears as a direct summand in $L(\nu) \otimes L(\nu') \otimes T(\varpi_2)^{\otimes c}$ for some $c \geq 0$, and it suffices to prove that every indecomposable direct summand $N$ of $L(\nu) \otimes L(\nu') \otimes T(\varpi_2)^{\otimes c}$ is either a tilting module or a simple $G$-module with highest weight in $C_2 \cap X$.
	
	If $c=0$ then $N$ is a direct summand of $L(\nu) \otimes L(\nu')$, and $L(\nu) \otimes L(\nu')$ is a direct sum of simple $G$-module with highest weights in $C_0 \cup C_2 \cup \{ (p-2) \cdot \varpi_1 + \varpi_2 , (p-2) \cdot \varpi_2 \}$ by the proof of Proposition \ref{prop:B2MFpairsinC1}.
	As the simple $G$-modules with highest weights in $C_0 \cup \{ (p-2) \cdot \varpi_1 + \varpi_2 , (p-2) \cdot \varpi_2 \}$ are tilting modules by Subsection \ref{subsec:B2Weylmodulestiltingmodules}, we conclude that either $N$ is a tilting module or $N \cong L(\delta)$ for some $\delta \in C_2 \cap X$.
	Now suppose that $c>0$. 
	Then $N$ is a direct summand of $N' \otimes T(\varpi_2)$, for some indecomposable direct summand $N'$ of $L(\nu) \otimes L(\nu') \otimes T(\varpi_2)^{\otimes (c-1)}$, and by induction, we may assume that either $N'$ is a tilting module or $N' \cong L(\delta')$ for some $\delta' \in C_2 \cap X$.
	If $N'$ is a tilting module then so is the direct summand $N$ of $N' \otimes T(\varpi_2)$.
	If $N' \cong L(\delta')$ for some $\delta' \in C_2 \cap X$ then $N' \otimes T(\varpi_2)$ is a direct sum of simple $G$-modules with highest weights in the upper closure of $C_2$ by Corollary \ref{cor:minusculetensorproductMFCR}, and since the simple $G$-modules with highest weights in $\widehat{C}_2 \setminus C_2$ are tilting modules by Subsection \ref{subsec:B2Weylmodulestiltingmodules}, we conclude that either $N$ is a tilting module or $N \cong L(\delta)$ for some $\delta \in C_2$, as required.
\end{proof}

Now we are ready to establish the sufficient conditions in Theorem \ref{thm:B2CR}, that is, we prove that for weights $\lambda,\mu \in X_1 \setminus \{ 0 \}$ that satisfy one of the conditions from Table \ref{tab:B2CR}, the tensor product $L(\lambda) \otimes L(\mu)$ is completely reducible.

\begin{proof}[Proof of Theorem \ref{thm:B2CR}, sufficient conditions]
\label{proof:B2CRsufficient}
	Let $\lambda,\mu \in X_1 \setminus \{0\}$ and suppose that $\lambda$ and $\mu$ satisfy one of the conditions from Table \ref{tab:B2CR}.
	We consider the different conditions in turn.
	\begin{itemize}
	\item Condition 1: $L(\lambda) \otimes L(\mu)$ is completely reducible by Proposition \ref{prop:B2MFpairsinC1}.
	\item Condition 2: $L(\lambda) \otimes L(\mu)$ is completely reducible by Proposition \ref{prop:B2MFCRp-1_1}.
	\item Condition 3: $L(\lambda) \otimes L(\mu)$ is completely reducible by Proposition \ref{prop:B2MFCRp-1_2}.
	\item Condition 4: $L(\lambda) \otimes L(\mu)$ is completely reducible by Proposition \ref{prop:B2MFCRp-21}.
	\item Condition 5: $L(\lambda) \otimes L(\mu)$ is completely reducible by Proposition \ref{prop:B2MFCRp-2_2}.
	\item Condition 6: $L(\lambda) \otimes L(\mu)$ is completely reducible by Proposition \ref{prop:B2MFCRp-2_1}.
	\item Condition 7: $L(\lambda) \otimes L(\mu)$ is completely reducible by Proposition \ref{prop:B2MFCRp-3_2}.
	\item Condition 8: $L(\lambda) \otimes L(\mu)$ is completely reducible by Theorem \ref{thm:reflectionsmalltensorproduct}.
	\end{itemize}
	Thus, if $\lambda$ and $\mu$ satisfy one of the conditions from Table \ref{tab:B2CR} then $L(\lambda) \otimes L(\mu)$ is completely reducible.
\end{proof}

\subsection{Proofs: complete reduciblity, necessary conditions}
\label{subsec:B2proofsCRnecessary}

In this subsection, we show that for all $\lambda,\mu \in X_1 \setminus \{0\}$ such that the tensor product $L(\lambda) \otimes L(\mu)$ is completely reducible, the weights $\lambda$ and $\mu$ must satisfy one of the conditions from Table \ref{tab:B2CR}, up to interchanging $\lambda$ and $\mu$.
The proof, which is a case-by-case analysis based on which alcoves or walls the weights $\lambda$ and $\mu$ belong to, will be given at the end of the subsection (see page \pageref{proof:B2CRnecessary}).
We will repeatedly use the following elementary observation:

\begin{Remark} \label{rem:LeviprestrctedconditionB2}
	Let $\lambda,\mu \in X_1$ and write $\lambda = a\varpi_1 + b\varpi_2$ and $\mu = a^\prime\varpi_1 + b^\prime\varpi_2$, with $0 \leq a,b,a^\prime,b^\prime < p$.
	Suppose that $L(\lambda) \otimes L(\mu)$ is completely reducible.
	Then all composition factors of $L(\lambda) \otimes L(\mu)$ have $p$-restricted highest weights by Theorem A in \cite{GruberCompleteReducibility}.
	By truncation to the two Levi subgroups of type $\mathrm{A}_1$ corresponding to the simple roots $\alpha_1$ and $\alpha_2$ respectively, we see that $L(\lambda+\mu - c\alpha_1)$ and $L(\lambda+\mu-d\alpha_2)$ are composition factors of $L(\lambda) \otimes L(\mu)$ for all $0 \leq c \leq \min\{a,a^\prime\}$ and $0 \leq d \leq \min\{b,b^\prime\}$, each appearing with composition multiplicity one (cf.\ Remark 4.13 in \cite{GruberCompleteReducibility}).
	In particular, the weights $\lambda+\mu - c \alpha_1$ and $\lambda+\mu - d \alpha_2$ are $p$-restricted, and it follows that
	\begin{align*}
	a + a^\prime + \min\{ b , b^\prime \} & \leq p-1 , \\
	b + b^\prime + 2 \cdot \min\{ a , a^\prime \} & \leq p-1 .
	\end{align*}
\end{Remark}

In the five propositions below, we consider tensor products $L(\lambda) \otimes L(\mu)$ where $\lambda,\mu \in X_1 \setminus \{0\}$ and $\lambda$ is $p$-singular.
We consider the different walls of the $p$-restricted alcoves in turn and show that $L(\lambda) \otimes L(\mu)$ is completely reducible only if $\lambda$ and $\mu$ satisfy one of the conditions 2--7 in Table \ref{tab:B2CR}.

\begin{Proposition}\label{prop:B2_CR_F34a}
	Let $\lambda\in F_{3,4a}\cap X_1$ and $\mu \in X_1\setminus\{0\}$ such that $L(\lambda)\otimes L(\mu)$ is completely reducible.
	Then $\lambda=(p-1)\varpi_1$ and $\mu=\varpi_2$.
\end{Proposition}

\begin{proof}
	Write $\lambda=(p-1)\varpi_1+b\varpi_2$ and $mu=a'\varpi_1+b'\varpi_2$.
	By Remark \ref{rem:LeviprestrctedconditionB2} we have 
	\[ (p-1) + a' + \min\{b,b'\} \leq p-1 , \]
	and it follows that $a'=0$ and $\min\{b,b'\}=0$.
	Since $\mu \neq 0$, we must have $b=0$, and it remains to prove that $b'=1$.
	
	Suppose for contradiction that $b'\geq 2$, and let $\nu \coloneqq \lambda+\mu - \varpi_1$. 
	By Lemma \ref{lem:tensormultiplicitiesmonotonous}, we have
	\[ [ \Delta(\lambda) \otimes \Delta(\mu) : \Delta(\nu) ]_\Delta \geq [ \Delta(\varpi_1) \otimes \Delta(\varpi_2) : \Delta(\varpi_2) ]_\Delta = 1 . \]
	Now using the results from Subsection \ref{subsec:B2Weylmodulestiltingmodules}, we see that $L(\lambda) = \Delta(\lambda)$ and $L(\mu) = \Delta(\mu)$, whereas the Weyl module $\Delta(\nu)$ is non-simple.
	Hence the non-simple Weyl module $\Delta(\nu)$ appears in a Weyl filtration of $L(\lambda) \otimes L(\mu)$, contradicting the assumption that $L(\lambda) \otimes L(\mu)$ is completely reducible.
\end{proof}

\begin{Proposition}\label{prop:B2_CR_F23a}
	Let $\lambda\in F_{2,3a} \cap X_1 = F_{3,4b} \cap X_1$ and $\mu \in X_1 \setminus \{ 0 \}$ such that $L(\lambda) \otimes L(\mu)$ is completely reducible.
	Then $\lambda=(p-1)\varpi_2$ and $\mu=\varpi_1$.
\end{Proposition}

\begin{proof}
	Write $\lambda=a\varpi_1+(p-1)\varpi_2$ and $\mu=a'\varpi_1+b'\varpi_2$.
	By Remark \ref{rem:LeviprestrctedconditionB2}, we have 
	\[ (p-1) + b' + 2 \min\{a,a'\} \leq p-1 , \]
	and it follows that $b'=0$ and $\min\{a,a'\}=0$.
	As $\mu \neq 0$, we conclude that $a=0$, and it remains to show that $a'=1$. 
	
	Suppose for contradiction that $a'\geq 2$. 
	If $a'=p-1$ then $\mu \in F_{3,4a}$, and Proposition \ref{prop:B2_CR_F34a} implies that $p=2$ with $\lambda = \varpi_2$ and $\mu = \varpi_1$, as claimed.
	For $2\leq a'\leq p-2$, we claim that $L(\lambda)\otimes L(\mu)$ has a composition factor of highest weight $\nu \coloneqq \lambda+\mu-2\varpi_2$.
	Indeed, the weight $\mu$ belongs to the closure of one of the alcoves $C_0$ and $C_1$, and by Subsection \ref{subsec:B2Weylmodulestiltingmodules}, we have $L(\lambda) = \Delta(\lambda)$ and either $L(\mu) = \Delta(\mu)$ (for $\mu \in \overline{C}_0 \cup F_{1,2}$) or $\Delta(\mu)$ is uniserial with composition series $[ L(\mu) , L(s\Cdot\mu) ]$ (for $\mu \in C_1$).
	If $\mu \in C_1$ then we further have $\lambda + s\Cdot\mu \leq \lambda \mu - 2 \varpi_1 \nleq \nu$, whence $[ L(\lambda) \otimes L(s\Cdot\mu) : L(\nu) ] = 0$, and in either case, we conclude using Lemma \ref{lem:tensormultiplicitiesmonotonous} that
	\begin{multline*}
		\qquad [ L(\lambda) \otimes L(\mu) : L(\nu) ] = [ \Delta(\lambda) \otimes \Delta(\mu) : L(\nu) ]
		\\ \geq [ \Delta(\lambda) \otimes \Delta(\mu) : \Delta(\nu) ]_\Delta \geq [ \Delta(2\varpi_2) \otimes \Delta(\varpi_1) : \Delta(\varpi_1) ]_\Delta = 1 , \qquad
	\end{multline*}
	as claimed.
	Now the weight $\nu$ is $p$-regular for $2 \leq a' \leq p-2$, contradicting the complete reducibility of $L(\lambda) \otimes L(\mu)$ by Lemma \ref{lem:nonCRcriterionpregularweight}.
\end{proof}

\begin{Proposition}\label{prop:B2_CR_F23}
	Let $\lambda\in F_{2,3}\cap X_1$ and $\mu \in X_1\setminus\{0\}$ such that $L(\lambda)\otimes L(\mu)$ is completely reducible.
	Then $p \neq 3$ and $\lambda=(p-2)\varpi_1+\varpi_2$ and $\mu=\varpi_1$.
\end{Proposition}
\begin{proof}
	If $p=2$ then $F_{2,3} \cap X_1 = \{ \varpi_2 \}$, so $\lambda = \varpi_2$, and Remark \ref{rem:LeviprestrctedconditionB2} forces $\mu = \varpi_1$.
	Now suppose that $p>2$ and write $\lambda=a\varpi_1+b\varpi_2$ and $\mu=a'\varpi_1+b'\varpi_2$.
	As $\lambda \in F_{2,3} \cap X_1$, we have $2a+b=2p-3$ and $b \leq p-1$, and it follows that $b$ is odd and $a \neq 0$.
	The weight $\lambda+\mu$ is $p$-restricted by Remark \ref{rem:LeviprestrctedconditionB2} and $p$-singular by Lemma \ref{lem:nonCRcriterionpregularweight}, and it follows that
	\[ \lambda+\mu \in \widehat{C}_3 \setminus C_3 \subseteq F_{3,4a} \cup F_{3,4b} . \]
	First suppose for a contradiction that $\lambda+\mu \in F_{3,4b} \setminus F_{3,4a}$, so $b+b'=p-1$, and as $b$ is odd, it follows that $b'$ is also odd.
	Using Remark \ref{rem:LeviprestrctedconditionB2} again, we obtain $b+b'+ 2 \min\{a,a'\}\leq p-1=b+b'$, and as $a>0$, it follows that $0=\min\{a,a'\}=a'$ and $\mu=b'\varpi_2$.
	By Subsection \ref{subsec:B2Weylmodulestiltingmodules}, we have $L(\lambda) = \Delta(\lambda)$ and $L(\mu) = \Delta(\mu)$, whereas the Weyl module $\Delta(\lambda+\mu)$ is non-simple.
	As $\Delta(\lambda+\mu)$ appears with multiplicity one in a Weyl filtration of $L(\lambda) \otimes L(\mu)$, this contradicts the complete reducibility of $L(\lambda) \otimes L(\mu)$.

	Now suppose that $\lambda+\mu\in F_{3,4a}$, so $a+a'=(p-1)$.
	By Remark \ref{rem:LeviprestrctedconditionB2}, we have
	\[ a+a'+\min\{b,b'\} \leq p-1 = a+a' , \]
	and as $b > 0$, it follows that $0 = \min\{b,b'\} = b'$ and $\mu = a' \varpi_1$, with $a' > 0$.
	Again by Remark \ref{rem:LeviprestrctedconditionB2}, the simple $G$-module $L(\lambda+\mu-\alpha_1)$ is a composition factor of $L(\lambda) \otimes L(\mu)$.
	If $b = p-2$ then the weight $\lambda+\mu-\alpha_1$ is non-$p$-restricted, contradicting Remark \ref{rem:LeviprestrctedconditionB2}, and if $1 < b < p-3$ then $\lambda+\mu-\alpha_1 \in C_3$ is $p$-regular, contradicting Lemma \ref{lem:nonCRcriterionpregularweight}.
	As $b$ is odd, we conclude that $1 = b \neq p-2$, and it follows that $p \neq 3$ and $\lambda = (p-2) \varpi_1 + \varpi_2$ and $\mu = \varpi_1$, as claimed.
\end{proof}

\begin{Proposition}\label{prop:B2_CR_F12}
	Let $\lambda\in F_{1,2}\cap ( X_1 \setminus \{0\} )$ and $\mu \in X_1\setminus\{0\}$ such that $L(\lambda)\otimes L(\mu)$ is completely reducible.
	Then one of the following holds: 
	\begin{enumerate}
		\item $p \geq 3$ and $\lambda\in\{(p-2)\varpi_1,(p-2)\varpi_2\}$ and $\mu=\varpi_2$;
		\item $p=3$ and $\lambda = \varpi_2$ and $\mu = \varpi_1$;
		\item $p=3$ and $\lambda = \varpi_2$ and $\mu = 2 \varpi_1$;
		\item $p=3$ and $\lambda = \varpi_1$ and $\mu = 2 \varpi_2$;
		\item $p=3$ and $\lambda = \mu = \varpi_1$.
	\end{enumerate}
\end{Proposition}
\begin{proof}
	First observe that the set $F_{1,2} \cap ( X_1 \setminus \{0\} )$ is empty for $p=2$, so we may assume that $p \geq 3$.
	As $\lambda$ is $p$-singular and $L(\lambda) \otimes L(\mu)$ is completely reducible, Lemma \ref{lem:nonCRcriterionpregularweight} implies that
	\begin{flalign} \label{eq:diamond}
		\text{the tensor product } L(\lambda) \otimes L(\mu) \text{ has no composition factors with $p$-regular highest weight} .&& \tag{$\diamond$}
	\end{flalign}
	Write $\lambda=a\varpi_1+b\varpi_2$ and $\mu=a'\varpi_1+b'\varpi_2$, and note that $a+b = p-2$ because $\lambda \in F_{1,2}$.
	The weight $\lambda+\mu$ is $p$-restricted by Remark \ref{rem:LeviprestrctedconditionB2} and $p$-singular by \eqref{eq:diamond}, and it follows that $\lambda+\mu$ belongs to one of the walls $F_{2,3}$, $F_{3,4a}$ or $F_{2,3a} = F_{3,4b}$.
	We consider the three walls in turn.
	\begin{itemize}
	\item First suppose that $\lambda+\mu\in F_{2,3a} = F_{3,4b}$, so $b+b'=p-1$ and $b'>0$ because $a+b=p-2$.
	If $b=0$ then $b'=p-1$ and $\mu \in F_{2,3a}$, and Proposition \ref{prop:B2_CR_F23a} forces that $\lambda = \varpi_1$, $p=3$ and $\mu = 2 \varpi_2$, as in case (4) above.
	Now additionally suppose that $b>0$.
	Then Remark \ref{rem:LeviprestrctedconditionB2} implies that $L(\lambda+\mu-\alpha_2)$ is a composition factor of $L(\lambda)\otimes L(\mu)$.
	The weight $\lambda+\mu-\alpha_2$ is $p$-regular for $1 \leq a + a' \leq p-3$ and non-$p$-restricted for $a+a'=p-1$, and so by \eqref{eq:diamond} and Remark \ref{rem:LeviprestrctedconditionB2}, we must have $a+a' \in \{0,p-2\}$.
	Again by Remark \ref{rem:LeviprestrctedconditionB2}, we have 
	\[ b+b'+2\min\{a,a'\}\leq p-1 = b + b' \]
	and $\min\{a,a'\}=0$.
	As $a+b=p-2$ and $b>0$, we conclude that $a=0$ and $a' \in \{ 0 , p-2 \}$, whence $\lambda = (p-2) \varpi_2$ and $b'=1$.
	Now Proposition \ref{prop:B2_CR_F23} implies that $\mu \notin F_{2,3}$ (because $\lambda \neq \varpi_1$) and so $a' \neq p-2$.
	We conclude that $a'=0$ and $\mu=\varpi_2$, as in case (1) above.
	
	\item Next suppose that $\lambda+\mu\in F_{3,4a}$, so $a+a'=p-1$.
	Using Remark \ref{rem:LeviprestrctedconditionB2}, it follows that 
	\[ a + a' + \min\{b,b'\} \leq p-1 = a + a' \]
	and so $\min\{b,b'\}=0$.
	First assume that $b=0$, so $\lambda=(p-2) \varpi_1$ and $\mu=\varpi_1 + b'\varpi_2$.
	Again by Remark \ref{rem:LeviprestrctedconditionB2}, the simple $G$-module $L(\lambda+\mu-\alpha_1)$ is a composition factor of $L(\lambda) \otimes L(\mu)$, and we have
	\[ b' + 2 \min\{a,a'\} = b'+ 2 \leq  p-1 , \]
	whence $b' \leq p-3$.
	If
	$1<b'<p-3$ then $\lambda+\mu-\alpha_1 \in C_3$ is $p$-regular, and so \eqref{eq:diamond} implies that we have $b' \in \{0,1,p-3\}$.
	For $b'=0$ and $p \geq 5$, the weight $\lambda+\mu-\alpha_1 \in C_2$ is again $p$-regular, so \eqref{eq:diamond} implies that we have $p=3$ and $\mu = \varpi_1$, as in case (5) above.
	If $b'=1$ then we must have $p \geq 5$ and $\mu=\varpi_1+\varpi_2$, and we claim that $L(\lambda) \otimes L(\mu)$ has a composition factor of highest weight $\nu \coloneqq \lambda + \mu - 2 \alpha_1 - \alpha_2$.
	Indeed, if $p \geq 7$ then $L(\lambda) = \Delta(\lambda)$ and $L(\mu) = \Delta(\mu)$ by Subsection \ref{subsec:B2Weylmodulestiltingmodules}, and using Lemma \ref{lem:tensormultiplicitiesmonotonous}, we obtain
	\[ [ L(\lambda) \otimes L(\mu) : L(\nu) ] \geq [ \Delta(\lambda) \otimes \Delta(\mu) : \Delta(\nu) ]_\Delta \geq [ \Delta(2\varpi_1) \otimes \Delta(\varpi_1+\varpi_2) : \Delta(3\varpi_2) ]_\Delta = 1 , \]
	as claimed.
	For $p=5$, the claim can be verified by a direct character computation using the results from Subsection \ref{subsec:B2Weylmodulestiltingmodules}.
	For all $p \geq 5$, the weight $\nu = (p-4) \varpi_1 + 3 \varpi_2$ is $p$-regular, contradicting \eqref{eq:diamond}.
	If $b'=p-3$ then we have $\mu = \varpi_1+(p-3)\varpi_2 \in F_{1,2}$, and we claim that $L(\lambda) \otimes L(\mu)$ has a composition factor of highest weight $\nu \coloneqq \lambda + \mu - \varpi_1$.
	Indeed, we have $L(\mu) = \Delta(\mu)$ by Subsection \ref{subsec:B2Weylmodulestiltingmodules}, and using Lemma \ref{lem:tensormultiplicitiesmonotonous} twice, we obtain
	\begin{multline*}
		\qquad [ L(\lambda) \otimes L(\mu) : L(\nu) ] = [ \Delta(\lambda) \otimes \Delta(\mu) : L(\nu) ] \geq [ \Delta(\lambda) \otimes \Delta(\mu) : \Delta(\nu) ]_\Delta \\
		 \geq [ \Delta(\varpi_1) \otimes \Delta(\mu) : \Delta(\mu) ]_\Delta \geq [ \Delta(\varpi_1) \otimes \Delta(\varpi_2) : \Delta(\varpi_2) ]_\Delta = 1 , \qquad
	\end{multline*}
	as claimed.
	If $p \geq 5$ then the weight $\nu$ is $p$-regular, contradicting \eqref{eq:diamond}, and if $p=3$ then $\mu = \varpi_1$, as in case (1) above.
	
	Next assume that $b>0$ and $b'=0$, so $\mu = a' \varpi_1$.
	If $a=0$ then $\mu=(p-1)\varpi_1 \in F_{3,4a}$ and $\lambda=(p-2)\varpi_2$, so Proposition \ref{prop:B2_CR_F34a} forces that $p=3$ and we are in case (3) above.
	Now additionally assume that $a>0$.
	Then $L(\lambda+\mu-\alpha_1)$ is a composition factor of $L(\lambda)\otimes L(\mu)$ by Remark \ref{rem:LeviprestrctedconditionB2}, and as before, the condition \eqref{eq:diamond} implies that we have $b \in \{1,p-3\}$.
	If $b=p-3$ then $\lambda=\varpi_1+(p-3)\varpi_2$ and $\mu=(p-2)\varpi_1 \in F_{1,2}$, and this case has already been considered before (with the roles of $\lambda$ and $\mu$ interchanged).
	If $b=1$ then we have $a=p-3$ and $a'=2$, and so $L(\lambda+\mu-2\alpha_1)$ is a composition factor of $L(\lambda)\otimes L(\mu)$ for $p \geq 5$ by Remark \ref{rem:LeviprestrctedconditionB2}.
	As the weight $\lambda+\mu-2\alpha_1$ is $p$-regular for $p \geq 5$, condition \eqref{eq:diamond} implies that $p=3$ and we are in case (3) above.
	
	\item Finally suppose that $\lambda+\mu\in F_{2,3}$, so $2 (a+a') + (b+b') = 2p - 3$.
	If $a > 0$ and $a' > 0$ then we must have $p>3$, and the simple $G$-module $L(\lambda+\mu-\alpha_1)$ is a composition factor of $L(\lambda)\otimes L(\mu)$ by Remark \ref{rem:LeviprestrctedconditionB2}.
	The weight $\lambda+\mu-\alpha_1$ is $p$-regular for $p \geq 5$, contradicting \eqref{eq:diamond}, and we conclude that either $a=0$ or $a'=0$.
	
	If $a=0$ then we have $\lambda=(p-2)\varpi_2$ and $\mu=\frac{p-1}{2}\varpi_1$. 
	For $p = 3$, we are in case (2) above, and for $p \geq 5$, we claim that $L(\lambda)\otimes L(\mu)$ has a composition factor of highest weight $\nu \coloneqq \lambda+\mu-\varpi_1$.
	Indeed, by Subsection \ref{subsec:B2Weylmodulestiltingmodules}, we have $L(\lambda) = \Delta(\lambda)$, and the Weyl module $\Delta(\mu)$ is uniserial with composition series $[ L(\mu), L(s\Cdot\mu)]$ because $\mu \in C_1$.
	Further observe that $s\Cdot\mu=\mu-2\varpi_1$, whence $[ L(\lambda) \otimes  L(s\Cdot \mu) : L(\nu) ]=0$, and using Lemma \ref{lem:tensormultiplicitiesmonotonous}, it follows that
	\[
		[ L(\lambda) \otimes L(\mu) : L(\nu) ] = [ \Delta(\lambda) \otimes \Delta(\mu) : L(\nu) ] \geq [ \Delta(\lambda) \otimes \Delta(\mu) : \Delta(\nu) ]_\Delta \geq [ \Delta(\varpi_2) \otimes \Delta(\varpi_1) : \Delta(\varpi_2) ]_\Delta = 1 ,
	\]
	as claimed.
	The weight $\nu$ is $p$-regular, contradicting \eqref{eq:diamond}.
	
	Now assume that $a>0$ and $a'=0$, so $\mu=b'\varpi_2$.
	If $p=3$ then $\lambda = \varpi_1$ and $\mu = \varpi_2$, as in case (1) above.
	For $p \geq 5$, we claim that $L(\lambda) \otimes L(\mu)$ has a composition factor of highest weight $\nu = \lambda+\mu-\varpi_1$.
	Indeed, by Subsection \ref{subsec:B2Weylmodulestiltingmodules}, we have $L(\lambda) = \Delta(\lambda)$ and $L(\mu) = \Delta(\mu)$, and using Lemma \ref{lem:tensormultiplicitiesmonotonous} twice, we obtain
	\begin{multline*}
		[ L(\lambda) \otimes L(\mu) : L(\nu) ] = [ \Delta(\lambda) \otimes \Delta(\mu) : L(\nu) ] \geq [ \Delta(\lambda) \otimes \Delta(\mu) : \Delta(\nu) ]_\Delta \\ \geq [\Delta(\varpi_1)\otimes \Delta(\mu) : \Delta(\mu) ]_\Delta \geq [\Delta(\varpi_1) \otimes \Delta(\varpi_2) : \Delta(\varpi_2) ]_\Delta = 1 ,
	\end{multline*}
	as claimed.
	The weight $\nu$ is $p$-regular unless $b+b'=1$, so \eqref{eq:diamond} implies that $b+b'=1$.
	Since we assume that $\mu \neq 0$, it follows that $\mu = \varpi_2$ and $\lambda = (p-2) \varpi_1$, as in case (1) above.
	\qedhere
	\end{itemize}
\end{proof}

\begin{Proposition}\label{prop:B2_CR_F01}
	Let $\lambda\in F_{0,1} \cap (X_1 \setminus \{0\} )$ and $\mu \in X_1\setminus\{0\}$ such that $L(\lambda)\otimes L(\mu)$ is completely reducible.
	Then one of the following holds:
	\begin{enumerate}
		\item $p \geq 5$ and $\lambda=(p-3)\varpi_2$ and $\mu=\varpi_1$;
		\item $\mu\in C_1$ and $\lambda$ is reflection small with respect to $\mu$;
		\item $p=5$ and $\{ \lambda , \mu \} = \{ \varpi_1 , 2 \varpi_2 \}$;
		\item $p=5$ and $\lambda = \varpi_1$ and $\mu = 4 \varpi_2$.
	\end{enumerate}
\end{Proposition}
\begin{proof}
	First observe that the set $F_{0,1} \cap (X_1 \setminus \{0\} )$ is empty for $p \leq 3$, so we may assume that $p \geq 5$.
	As $\lambda$ is $p$-singular and $L(\lambda) \otimes L(\mu)$ is completely reducible, Lemma \ref{lem:nonCRcriterionpregularweight} implies that
	\begin{flalign} \label{eq:ast}
		\text{the tensor product } L(\lambda) \otimes L(\mu) \text{ has no composition factors with $p$-regular highest weight} . && \tag{$\ast$}
	\end{flalign}
	Write $\lambda = a \varpi_1 + b \varpi_2$ and $\mu=a'\varpi_1+b'\varpi_2$, and note that $2a+b=p-3$ because $\lambda \in F_{1.2}$.
	The weight $\lambda+\mu$ is $p$-restricted by Remark \ref{rem:LeviprestrctedconditionB2} and $p$-singular by \eqref{eq:ast}, and it follows that $\lambda+\mu$ belongs to one of the walls $F_{1,2}$, $F_{2,3}$, $F_{3,4a}$ or $F_{2,3a} = F_{3,4b}$.
	We consider each of the four walls in turn.
	
	\begin{itemize}
	\item First suppose that $\lambda+\mu\in F_{2,3a} = F_{3,4b}$, so $b+b'=p-1$, and Remark \ref{rem:LeviprestrctedconditionB2} implies that 
	\[ b + b' + 2 \min\{a,a'\} \leq p-1 = b + b' \]
	and $\min\{a,a'\}=0$.
	If $a=0$ then $\lambda=(p-3)\varpi_2$ and $b'=2$, and again by Remark \ref{rem:LeviprestrctedconditionB2}, the simple $G$-modules $L(\lambda+\mu-\alpha_2)$ and $L(\lambda+\mu-2\alpha_2)$ are composition factors of $L(\lambda)\otimes L(\mu)$.
	If $1\leq a'\leq p-3$ then $\lambda+\mu-\alpha_2$ is $p$-regular, contradicting \eqref{eq:ast}, and if $a' \geq p-2$ then $a'+\min\{b,b'\} \geq p$, contradicting the complete reducibility of $L(\lambda) \otimes L(\mu)$ by Remark \ref{rem:LeviprestrctedconditionB2}.
	If $a = a' = 0$ then $\lambda+\mu - 2 \alpha_2 = 2 \varpi_1 + (p-5) \varpi_2$ is $p$-regular, again contradicting \eqref{eq:ast}.
	Thus we must have $a>0$ and $a' = 0$, whence $\mu = b' \varpi_2$.
	If $b=0$ then $b' = p-1$ and $\mu = (p-1) \varpi_2 \in F_{3,4b}$, and Proposition \ref{prop:B2_CR_F23a} implies that we must have $\lambda = \varpi_1$ and $p=5$, as in case (4) above.
	If $b>0$ then $L(\lambda+\mu-\alpha_2)$ is a composition factor of $L(\lambda)\otimes L(\mu)$ by Remark \ref{rem:LeviprestrctedconditionB2}.
	As $2a+b=p-3$, we further have $a \leq \frac{p-5}{2}$, and so $\lambda+\mu - \alpha_2 \in C_3$ is $p$-regular, contradicting \eqref{eq:ast}.
	
	\item Next suppose that $\lambda+\mu\in F_{3,4a}$, so $a+a'=p-1$, and by Remark \ref{rem:LeviprestrctedconditionB2}, we have 
	\[ a	+a'+\min\{b,b'\}\leq p-1=a+a' \]
	and $\min\{b,b'\}=0$.
	First assume that $b=0$, so $\lambda=\frac{p-3}{2}\varpi_1$ and $\mu=\frac{p+1}{2}\varpi_1+b'\varpi_2$.
	By Remark \ref{rem:LeviprestrctedconditionB2}, the simple $G$-module $L(\lambda+\mu-i\alpha_1)$ is a composition factor of $L(\lambda)\otimes L(\mu)$ for $1\leq i\leq \frac{p-3}{2}$, and we further have
	\[ b'+2\min\{a,a'\}=b'+p-3\leq p-1 , \]
	whence $b'\leq 2$.
	If $b'\neq 1$ then $\lambda+\mu-\alpha_1$ is $p$-regular, contradicting \eqref{eq:ast}.
	If $b'=1$ then either $p=5$ and $\mu = 3 \varpi_1 + \varpi_2 \in F_{2,3}$, contradicting the complete reducibility of $L(\lambda) \otimes L(\mu)$ by Proposition \ref{prop:B2_CR_F23}, or $p \geq 7$ and the simple $G$-module $L(\lambda+\mu - 2\alpha_2)$ is a composition factor with $p$-regular highest weight in $L(\lambda) \otimes L(\mu)$, contradicting \eqref{eq:ast}.
	
	Hence we must have $b>0$ and $b'=0$.
	If $a>0$ then Remark \ref{rem:LeviprestrctedconditionB2} implies that $L(\lambda+\mu - \alpha_1)$ is a composition factor of $L(\lambda) \otimes L(\mu)$, and as $b = p - 3 - 2a$ is even, we have $2 \leq b \leq p-5$ and $\lambda+\mu-\alpha_1$ is $p$-regular, contradicting \eqref{eq:ast}.
	Thus we further have $a=0$ and $\mu = (p-1) \varpi_1 \in F_{3,4a}$, contradicting the complete reducibility of $L(\lambda) \otimes L(\mu)$ by Proposition \ref{prop:B2_CR_F34a}.
		
	\item Next suppose that $\lambda+\mu\in F_{2,3}$, so that $2(a+a')+(b+b')=2p-3$.
	As $2a+b = p-3$, we further have $2a'+b'= p$, and it follows that $a'>0$ and $b'$ is odd.
	If $a > 0$ then Remark \ref{rem:LeviprestrctedconditionB2} implies that $L(\lambda+\mu-\alpha_1)$ is a composition factor of $L(\lambda)\otimes L(\mu)$, but $\lambda+\mu-\alpha_1 \in C_2 \cup C_{3a}$ is $p$-regular (note that $\lambda+\mu - \alpha_1$ cannot belong to the wall $F_{2,3a}$ because $b+b'$ is odd), contradicting \eqref{eq:ast}.
	Thus we have $a=0$ and $\lambda=(p-3)\varpi_2$, and as $b'$ is odd, it follows that $b'=1$ (otherwise $\lambda+\mu$ is not $p$-restricted) and $\mu=\frac{p-1}{2}\varpi_1+\varpi_2$.
	We claim that the tensor product $L(\lambda)\otimes L(\mu)$ has a composition factor of highest weight $\nu \coloneqq \lambda+\mu - \varpi_1$.
	Indeed, as $\mu\in C_1$, the Weyl module $\Delta(\mu)$ is uniserial with composition series $[L(\mu),L(s\Cdot\mu)]$ by Subsection \ref{subsec:B2Weylmodulestiltingmodules}, and so we have
	\[ [ L(\lambda) \otimes L(\mu) : L(\nu) ] = [ L(\lambda) \otimes \Delta(\mu) : L(\nu) ] - [ L(\lambda) \otimes L(s\Cdot\mu) : L(\nu) ] . \]
	Now $s\Cdot\mu=\mu-3\varpi_1$ and therefore $\nu = \lambda+\mu-\varpi_1 \nleq \lambda+s\Cdot\mu$ and $[ L(\lambda) \otimes L(s\Cdot\mu) : L(\nu) ] = 0$.
	Again by Subsection \ref{subsec:B2Weylmodulestiltingmodules}, we have $\Delta(\lambda) = L(\lambda)$, and using Lemma \ref{lem:tensormultiplicitiesmonotonous} in the second inequality, we obtain
	\[ [ L(\lambda) \otimes L(\mu) : L(\nu) ] = [ \Delta(\lambda) \otimes \Delta(\mu) : L(\nu) ] \geq [ \Delta(\lambda) \otimes \Delta(\mu) : L(\nu) ] \geq [ \Delta(\varpi_2) \otimes \Delta(\varpi_1) :\Delta(\varpi_2) ]_\Delta = 1 , \]
	as claimed.
	The weight $\nu$ is $p$-regular, contradicting \eqref{eq:ast}.
	
	\item Finally, suppose that $\lambda+\mu\in F_{1,2}$, so that $a+b+a'+b'=p-2$.
	If $b > 0$ and $b' > 0$ then $L(\lambda+\mu-\alpha_2)$ is a composition factor of $L(\lambda)\otimes L(\mu)$ by Remark \ref{rem:LeviprestrctedconditionB2}, but $\lambda+\mu-\alpha_2$ is $p$-regular, contradicting \eqref{eq:ast}.
	Thus, we must have $b=0$ or $b'=0$.
	
	First assume that $b=0$, whence $\lambda = \frac{p-3}{2} \varpi_1$ and $a' = \frac{p-1}{2}-b'$.
	For $b' \leq 1$, we have
	\[ \mu \in \big\{ \tfrac{p-1}{2} \varpi_1 , \tfrac{p-3}{2} \varpi_1 + \varpi_2 \big\} \]
	and $\lambda$ is reflection small with respect to $\mu$, as in case (2) above.
	For $b' \geq 2$, we claim that the simple $G$-module of highest weight $\nu = \lambda+\mu - \varpi_1$ is a composition factor of $L(\lambda) \otimes L(\mu)$.
	Indeed, we have $\Delta(\mu) = L(\mu)$ because $2a'+b' \leq p-1-b' \leq p-3$ (so $\mu$ belongs to the upper closure of $C_0$), and using Lemma \ref{lem:tensormultiplicitiesmonotonous}, we obtain
	\[ [ L(\lambda) \otimes L(\mu) : L(\nu) ] \geq [ \Delta(\lambda) \otimes \Delta(\mu) : \Delta(\nu) ]_\Delta \geq [ \Delta(\varpi_1) \otimes \Delta(\varpi_2) : \Delta(\varpi_2) ]_\Delta = 1 , \]
	as claimed.
	The weight $\nu$ is $p$-regular if $p \geq 7$, contradicting \eqref{eq:ast}, and if $p=5$ then we have $\lambda = \varpi_1$ and $\mu = 2\varpi_2$ as in case (3) above.
	
	Next assume that $b>0$ and $b' = 0$.
	Then we have $a \leq \frac{p-5}{2}$ and
	\[ a' = (p-2)- (a+b) = (p-3) - (2a+b) + (a+1) = a+1 \leq \tfrac{p-3}{2} , \]
	whence $\mu = a' \varpi_1$ belongs to the upper closure of $C_0$.
	As before, the simple $G$-module with highest weigh $\nu \coloneqq \lambda+\mu - \varpi_1$ is a composition factor of $L(\lambda) \otimes L(\mu)$ because
	\[ [ L(\lambda) \otimes L(\mu) : L(\nu) ] \geq [ \Delta(\lambda) \otimes \Delta(\mu) : \Delta(\nu) ]_\Delta \geq [ \Delta(\varpi_1) \otimes \Delta(\varpi_2) : \Delta(\varpi_2) ]_\Delta = 1 . \]
	If $a > 0$ then $\nu$ is $p$-regular, contradicting \eqref{eq:ast}.
	We conclude that $a=0$, so $\lambda = (p-3) \varpi_2$ and $\mu = \varpi_1$, as in case (1) above.
	\qedhere
	\end{itemize}
\end{proof}

Now we consider tensor products $L(\lambda) \otimes L(\mu)$ where $\lambda,\mu \in X_1 \setminus \{0\}$ are $p$-regular.
Our first goal is to prove that for $\lambda,\mu \in C_1 \cap X$ such that $L(\lambda) \otimes L(\mu)$ is completely reducible, the weights $\lambda$ and $\mu$ must satisfy condition 1 in Table \ref{tab:B2CR}.
This will follow from the lemma and proposition below.

\begin{Lemma} \label{lem:B2weightsinC1conditions}
	Let $\lambda,\mu \in C_1 \cap X$  and write $\lambda = a\varpi_1 + b\varpi_2$ and $\mu = a^\prime\varpi_1 + b^\prime\varpi_2$, with $a,b,a^\prime,b^\prime \geq 0$.
	Suppose that tensor product $L(\lambda) \otimes L(\mu)$ is completely reducible.
	Then up to interchanging $\lambda$ and $\mu$, one of the following holds:
	\begin{enumerate}
	\item $2a+b = 2a^\prime + b^\prime = p-2$,
	\item $2a+b = p-1$ and $2a^\prime+b^\prime = p-2$,
	\item $a = a^\prime = \frac{p-1}{2}$ and $b = b^\prime = 0$.
	\end{enumerate}
\end{Lemma}
\begin{proof}
	First note that $L(\lambda) \otimes L(\mu)$ cannot have any composition factors of the form $L(\nu)$ with $\nu \in C_3$ by Corollary \ref{cor:nonCRcriterionGFD}; in particular $\lambda+\mu \notin C_3$.
	Since $\lambda+\mu$ must also be $p$-restricted, we have
	\[ a+a^\prime = p-1 \qquad \text{or} \qquad b+b^\prime = p-1 \qquad \text{or} \qquad 2 \cdot (a+a^\prime) + (b+b^\prime) \leq 2p-3 . \]
	As $\lambda,\mu \in C_1$, we also have $2a+b \geq p-2$ and $2a^\prime + b^\prime \geq p-2$, and so the third condition would imply that either (1) or (2) holds.
	Next suppose that $b+b^\prime = p-1$.
	By Remark \ref{rem:LeviprestrctedconditionB2}, we may assume without loss of generality that $a = 0$, but then $\lambda = b \varpi_2$, contradicting the assumption that $\lambda \in C_1$.
	
	Finally, suppose that $a+a^\prime = p-1$.
	As $\lambda,\mu \in C_1$, we have $a+b \leq p-3$ and $a'+b' \leq p-3$, and it follows that $a \geq 2$ and $a' \geq 2$.
	By Remark \ref{rem:LeviprestrctedconditionB2}, we have $b+b' \leq p-5$ and $L(\lambda+\mu-\alpha_1)$ is a composition factor of $L(\lambda) \otimes L(\mu)$, and as $\lambda+\mu - \alpha_1 \in C_3$ for $2 \leq b+b' \leq p-4$, we conclude that $b+b' \in \{ 0,1 \}$ (using again Corollary \ref{cor:nonCRcriterionGFD}).
	Suppose for a contradiction that $b+b' = 1$, and without loss of generality, let $b=0$ and $b'=1$.
	Then we have $a \geq \frac{p-1}{2}$ and $a' \geq \frac{p-3}{2}$, and Remark \ref{rem:LeviprestrctedconditionB2} forces that $a' = \frac{p-3}{2}$ and $p-3 \geq a = \frac{p+1}{2}$, so $\lambda = \frac{p+1}{2} \cdot \varpi_1$ and $\mu = \frac{p-3}{2} \cdot \varpi_1 + \varpi_2$ and $p \geq 7$.
	Consider the weight $\nu \coloneqq \varpi_1 + (p-4) \varpi_2 \in C_1$, and observe that the tensor product $L(\mu) \otimes L(\nu)$ has a composition factor of highest weight
	\[ \mu + \nu - \alpha_2 = \tfrac{p+1}{2} \cdot \varpi_1 + (p-5) \cdot \varpi_2 =sts \Cdot \lambda \in C_2 \]
	by Remark \ref{rem:LeviprestrctedconditionB2}.
	Further consider the weights
	\[ \lambda_0 = s \Cdot \lambda = \tfrac{p-7}{2} \cdot \varpi_1 \in C_0 , \qquad \mu_0 = s \Cdot \mu = \tfrac{p-5}{2} \cdot \varpi_1 + \varpi_2 , \qquad \nu_0 = s \Cdot \nu = (p-4) \cdot \varpi_2 \in C_0 , \]
	and note that $\nu_0 = \omega \Cdot 0$; cf.\ Remark \ref{rem:B2stabC0}.
	As $\omega \Cdot \mu_0 = \frac{p-5}{2} \cdot \varpi_1 \neq \lambda_0$, we have
	\[ c_{\mu_0,\nu_0}^{\lambda_0} = c_{\mu_0,\omega\Cdot0}^{\lambda_0} = c_{\mu_0,0}^{\omega\Cdot\lambda_0} = c_{\omega\Cdot\mu_0,0}^{\lambda_0} = \big[ T(\omega\Cdot\mu_0) \otimes T(0) : T(\lambda_0) \big]_\oplus = 0 \]
	by Lemma \ref{lem:VerlindecoefficientsFundamentalgroup}, and so the regular simple $G$-module $L(\mu+\nu-\alpha_2) = L(st\Cdot\lambda_0)$ cannot be a direct summand of $L(\mu) \otimes L(\nu) = L(s\Cdot\mu_0) \otimes L(s\Cdot\nu_0)$ by Theorem \ref{thm:translationprincipletensorproducts}.
	Using Proposition \ref{prop:B2directsummandsC1} and weight considerations, it follows that the tilting module $T(st\Cdot\lambda_0)$ is a direct summand of $L(\mu) \otimes L(\nu)$, and as $T(st\Cdot\lambda_0)$ has simple socle isomorphic to $L(s\Cdot\lambda_0) = L(\lambda)$ by Subsection \ref{subsec:B2Weylmodulestiltingmodules}, we obtain
	\[ 0 \neq \Hom_G\big( L(\lambda) , L(\mu) \otimes L(\nu) \big) \cong \Hom_G\big( L(\lambda) \otimes L(\mu) , L(\nu) \big) . \]
	As $L(\lambda) \otimes L(\mu)$ is completely reducible, it follows that $L(\lambda) \otimes L(\mu)$ has a simple direct summand of highest weight $\nu \in C_1$, contradicting Proposition \ref{prop:B2directsummandsC1} (because $L(\nu)$ is not a tilting module by Subsection \ref{subsec:B2Weylmodulestiltingmodules}).
	We conclude that $b+b'=0$ and it follows that $a=a'=\frac{p-1}{2}$, as in case (3) above.
\end{proof}

\begin{Proposition}\label{prop:B2_CR_C1}
	Let $\lambda , \mu \in C_1 \cap X$ such that $L(\lambda) \otimes L(\mu)$ is completely reducible.
	Then we have
	\[ \lambda,\mu \in \{ \tfrac{p-1}{2} \cdot \varpi_1 , \tfrac{p-3}{2} \cdot \varpi_1 + \varpi_2 \} . \]
\end{Proposition}
\begin{proof}
	Let us write $\lambda = a\varpi_1 + b\varpi_2$ and $\mu = a^\prime\varpi_1 + b^\prime\varpi_2$, with $a,b,a^\prime,b^\prime \geq 0$.
	By Lemma \ref{lem:B2weightsinC1conditions}, we may assume that up to reordering of $\lambda$ and $\mu$, 
	one of the following holds:
	\begin{enumerate}
	\item $2a+b = 2a^\prime + b^\prime = p-2$,
	\item $2a+b = p-1$ and $2a^\prime+b^\prime = p-2$,
	\item $a = a^\prime = \frac{p-1}{2}$ and $b = b^\prime = 0$.
	\end{enumerate}
	In case (3), we have $\lambda = \mu = \tfrac{p-1}{2} \cdot \varpi_1$, as required.
	
	Now suppose that $2a+b = p-1$ and $2a^\prime+b^\prime = p-2$ as in (2).
	Then Remark \ref{rem:LeviprestrctedconditionB2} yields
	\[ p-1 \geq (p-1-2a) + (p-2-2a^\prime) + 2 \cdot \min\{ a , a^\prime \} = 2p - 3 - 2 \cdot \max\{ a , a^\prime \} , \]
	or equivalently $p-2 \leq 2 \cdot \max\{ a , a^\prime \}$.
	As $p$ is odd, it follows that $\frac{p-1}{2} \leq \max\{ a , a^\prime \}$, and this forces that $a = \frac{p-1}{2}$ and $b=0$.
	Now suppose for a contradiction that $b^\prime = p-2-2a^\prime > 1$ and $2a^\prime < p-3$, note that we have $\mu = a^\prime \cdot \varpi_1 + (p-2-2a^\prime) \cdot \varpi_2$, and consider the weights
	\begin{align*}
	\nu & = (\tfrac{p-3}{2}-a^\prime) \cdot \varpi_1 + (2a^\prime+1) \cdot \varpi_2 \in C_1 , \\
	\nu_0 & = s \Cdot \nu = \nu - \varpi_1 = (\tfrac{p-5}{2}-a^\prime) \cdot \varpi_1 + (2a^\prime+1) \cdot \varpi_2 \in C_0 , \\
	\mu_0 & = s \Cdot \mu = \mu - \varpi_1 =  ( a^\prime - 1 ) \cdot \varpi_1 + ( p-2-2a^\prime ) \cdot \varpi_2 \in C_0 	\\
	\lambda_0 & = s\Cdot\lambda = \lambda - 2 \cdot \varpi_1 = \tfrac{p-5}{2} \cdot \varpi_1 \in C_0 , \\
	\lambda_2 & = sts \Cdot \lambda = st \Cdot \lambda_0 = \tfrac{p-1}{2} \cdot \varpi_1 + (p-3) \cdot \varpi_2 \in C_2
	\end{align*}
	Also observe that $\omega\Cdot\mu_0 = (a^\prime-1) \cdot \varpi_1$ and $\omega\Cdot\nu_0 = ( \tfrac{p-5}{2} - a^\prime ) \cdot \varpi_1$ by Remark \ref{rem:B2stabC0}.
	Using Lemma \ref{lem:VerlindecoefficientsFundamentalgroup}, we compute
	\[ c_{\mu_0,\nu_0}^{\lambda_0} = c_{\omega\Cdot\mu_0,\nu_0}^{\omega\Cdot\lambda_0} = c_{\omega\Cdot\mu_0,\omega\Cdot\nu_0}^{\lambda_0} = \big[ T\big( (a^\prime-1) \cdot \varpi_1 \big) \otimes T\big( ( \tfrac{p-5}{2} - a^\prime ) \cdot \varpi_1 \big) : T\big( \tfrac{p-5}{2} \cdot \varpi_1 \big) \big]_\oplus = 0 , \]
	so Theorem \ref{thm:translationprincipletensorproducts} implies that the tensor product $L(\mu) \otimes L(\nu) = L(s\Cdot\mu_0) \otimes L(s\Cdot\nu_0)$ has no regular indecomposable direct summands which belong to the linkage class of $\lambda_0$.
	On the other hand, it is straightforward to see that $L(\lambda_2) = L(st\Cdot\lambda_0) = L(\mu+\nu-\alpha_2)$ is a composition factor of $L(\mu) \otimes L(\nu)$, and since $L(\lambda_2)$ is regular, Proposition \ref{prop:B2directsummandsC1} implies that there is a tilting module $T(\delta)$, with $\delta \in X^+$, such that $T(\delta)$ is a direct summand of $L(\mu) \otimes L(\nu)$ and $L(\lambda_2)$ is a composition factor of $T(\delta)$.
	Now it is straightforward to see using weight considerations that we must have $\delta = \lambda_2$.
	Since $T(\lambda_2)$ has socle isomorphic to $L(\lambda)$ (see Subsection \ref{subsec:B2Weylmodulestiltingmodules}) and $T(\lambda_2)$ is a direct summand of $L(\mu) \otimes L(\nu)$, it follows that
	\[ 0 \neq \Hom_G\big( L(\lambda) , L(\mu) \otimes L(\nu) \big) \cong \Hom_G\big( L(\nu) , L(\lambda) \otimes L(\mu) \big) , \]
	and $L(\nu)$ cannot be a direct summand of $L(\lambda) \otimes L(\mu)$ by Proposition \ref{prop:B2directsummandsC1}.
	This contradicts the complete reducibility of $L(\lambda) \otimes L(\mu)$, and we conclude that $b^\prime = 1$ and $\mu = \frac{p-3}{2} \cdot \varpi_1 + \varpi_2$.
	
	Now suppose that $2a+b = 2a^\prime+b^\prime = p-2$ as in (1), and note that $b > 0$ and $b^\prime > 0$ as $p$ is odd.
	We may assume without loss of generality that $a^\prime \leq a$, so Remark \ref{rem:LeviprestrctedconditionB2} yields $p-2 = 2a^\prime +b^\prime \leq 2a^\prime + b^\prime + b \leq p-1$, and we conclude that $b = 1$ and $a = \frac{p-3}{2}$.
	Suppose for a contradiction that $b^\prime > 1$, let $c = \frac{p-3}{2}-a^\prime$ so that $\mu = ( \frac{p-3}{2} - c ) \cdot \varpi + (2c+1) \cdot \varpi_2$, and consider the weights
	\begin{align*}
	\nu & = (p-2-c) \cdot \varpi_1 \in C_1 , \\
	\nu_0 & = s \Cdot \nu = \nu - (p-1-2c) \cdot \varpi_1 = (c-1) \cdot \varpi_1 \in C_0 , \\
	\nu_2 & = sts \Cdot \nu = st \Cdot \nu_0 = (p-2-c) \cdot \varpi_1 + 2c \cdot \varpi_2 \in C_2 , \\
	\lambda_0 & = s \Cdot \lambda = \lambda - \varpi_1 = \tfrac{p-5}{2} \cdot \varpi_1 + \varpi_2 \in C_0 , \\
	\mu_0 & = s \Cdot \mu = \mu - \varpi_1 = ( \tfrac{p-5}{2} - c ) \cdot \varpi + (2c+1) \cdot \varpi_2 \in C_0 .
	\end{align*}
	We further have $\omega\Cdot\lambda_0 = \frac{p-5}{2} \cdot \varpi_1$ and $\omega\Cdot\mu_0 = ( \tfrac{p-5}{2} - c ) \cdot \varpi_1$ by Remark \ref{rem:B2stabC0}, and using Lemmas \ref{lem:VerlindecoefficientsFundamentalgroup} and \ref{lem:Verlindecoefficientsflipping}, we compute that
	\[ c_{\lambda_0,\mu_0}^{\nu_0} = c_{\omega\Cdot\lambda_0,\omega\Cdot\mu_0}^{\nu_0} = c_{\omega\Cdot\mu_0,\nu_0}^{\omega\Cdot\lambda_0} = \big[ T\big( (\tfrac{p-5}{2}-c) \cdot \varpi_1 \big) \otimes T\big( (c-1) \cdot \varpi_1 \big) : T\big( \tfrac{p-5}{2} \cdot \varpi_1 \big) \big]_\oplus = 0 . \]
	As before, this implies that the regular $G$-module $L(\nu_2) = L(st\Cdot\nu_0)$ cannot be a direct summand of the tensor product $L(\lambda) \otimes L(\mu) = L(s\Cdot\lambda_0) \otimes L(s\Cdot\mu_0)$ (using Theorem \ref{thm:translationprincipletensorproducts}).
	On the other hand, it is straightforward to see that $L(\lambda) \otimes L(\mu)$ has a composition factor of highest weight $\nu_2 = \lambda+\mu - \alpha_2$.
	This contradicts the complete reducibility of $L(\lambda) \otimes L(\mu)$, and we conclude that $b^\prime = 1$ and $\mu = \frac{p-3}{2} \cdot \varpi_1 + \varpi_2$, as required.
\end{proof}

Next we want to show that for $\lambda , \mu \in X_1 \setminus \{ 0 \}$ such that $\lambda \in C_2 \cup C_3$ and $L(\lambda) \otimes L(\mu)$ is completely reducible, the weight $\mu$ must be reflection small with respect to $\lambda$ (i.e.\ we are in case 8 of Table \ref{tab:B2CR}).
This is a consequence of the three following lemmas.

\begin{Lemma}\label{lem:CR_C2_C3_implies_C0}
	Let $\lambda \in ( \widehat{C}_2 \cup \widehat{C}_3 ) \cap X$ and $\mu \in X_1 \setminus \{0\}$ such that $L(\lambda)\otimes L(\mu)$ is completely reducible.
	Then one of the following holds:
	\begin{enumerate}
		\item $\mu \in \widehat{C}_0$;
		\item $p \leq 3$ and $\lambda = (p-1) \varpi_2$ and $\mu = \varpi_1$;
		\item $p \leq 3$ and $\lambda = (p-1) \varpi_1$ and $\mu = \varpi_2$;
	\end{enumerate}
\end{Lemma}

\begin{proof}
	Write $\lambda=a\varpi_1+b\varpi_2$ and $\mu=a'\varpi_1+b'\varpi_2$, and observe that $a+b \geq p-1$ by our assumption on $\lambda$.
	By Remark \ref{rem:LeviprestrctedconditionB2}, we further have
	\[ a + a' + \min\{ b , b' \} \leq p-1 , \qquad b + b' + 2\min\{ a , a' \} \leq p-1 , \]
	and it follows that $b' + 2 \min\{ a , a' \} \leq a$, so either $a = b' = 0$ or $a' < a$.
	
	If $a=b'=0$ then $\lambda = (p-1) \varpi_2 \in F_{2,3a}$, and Proposition \ref{prop:B2_CR_F23a} implies that $\mu = \varpi_1$.
	Note that we have $\varpi_1 \in \overline{C}_0$ for $p \geq 5$, so we are in one of the cases (1) or (2) above, as claimed.
	
	Now suppose that $a' < a$, and therefore $2a' + b' \leq a \leq p-1$.
	We consider three cases.
	\begin{itemize}
	\item If $2a'+b' \leq p-3$ then $\mu \in \overline{C}_0$, and we are in case (1) above.
	\item If $a = p-1$ then $\lambda \in F_{3,4a}$, and by Proposition \ref{prop:B2_CR_F34a}, we have $\lambda = (p-1) \varpi_1$ and $\mu = \varpi_2$.
	Note that $\varpi_2 \in C_0$ for $p \geq 5$, so we are in one of the cases (1) or (3) above.
	\item If $2a' + b' = a = p-2$ then we have
	\[ a' + \min\{ b , b' \} \leq p-1-a = 1 , \qquad b \geq p-1-a = 1 \]
	For $a' = 1$, it follows that $0 = \min\{ b , b' \} = b'$ and therefore $p-2 = 2a' = 2$, a contradiction.
	For $a' = 0$, we further have $b' = p-2$ and $b \leq p-1-b' = 1$, and we conclude that $b=1$.
	Then $\lambda = (p-2) \cdot \varpi_1 + \varpi_2 \in F_{2,3}$, and Proposition \ref{prop:B2_CR_F23} implies that $\mu = \varpi_1$, contradicting the assumption that $a'=0$.
	\end{itemize}
	We conclude that the weights $\lambda$ and $\mu$ satisfy one of the conditions (1)--(3) above, as claimed.
\end{proof}

\begin{Lemma}\label{lem:B2_C3_CR_reflectionsmall}
	Let $\lambda \in C_3 \cap X$ and $\mu \in \widehat{C}_0 \cap X$ such that $L(\lambda)\otimes L(\mu)$ is completely reducible.
	Then $\mu$ is reflection small with respect to $\lambda$.
\end{Lemma}

\begin{proof}
	Write $\lambda=a\varpi_1+b\varpi_2$ and $\mu=a'\varpi_1+b'\varpi_2$, and recall from Example \ref{ex:B2reflectionsmall} that it suffices to prove that $a+a'+b'\leq p-1$ and $b+b'+2a'\leq p-1$.
	By Remark \ref{rem:LeviprestrctedconditionB2}, we have
	\[ a + a' + \min\{ b , b' \} \leq p-1 , \qquad b + b' + 2 \min\{ a , a' \} \leq p-1 . \]
	As $\lambda \in C_3$ and $\mu \in \overline{C}_0$, we further have $a \geq \frac{2p-2-b}{2} \geq \frac{p-1}{2} > a'$, and it follows that $b + b' + 2a' \leq p-1$, as required.
	Again using the fact that $\lambda \in C_3$, we have $a+b \geq p-1$ and therefore
	\[ b \geq p-1-a \geq p-1 - (a+a') \geq \min\{ b , b' \} , \]
	and we conclude that $b' \leq b$ and $a + a' + b' \leq p-1$, as required.
\end{proof}

\begin{Lemma}\label{lem:B2_C2_CR_reflectionsmall}
	Let $\lambda\in C_2 \cap X$ and $\mu\in C_0 \cap X$ such that $L(\lambda) \otimes L(\mu)$ is completely reducible.
	Then $\mu$ is reflection small with respect to $\lambda$.
\end{Lemma}
\begin{proof}
	Write $\lambda=a\varpi_1+b\varpi_2$ and $\mu=a'\varpi_1+b'\varpi_2$, and recall from Example \ref{ex:B2reflectionsmall} that it suffices to prove that $2(a+a')+(b+b')\leq 2p-3$ and $b+b'+2a'\leq p-1$.
	By Remark \ref{rem:LeviprestrctedconditionB2}, we have
	\[ a + a' + \min\{ b , b' \} \leq p-1 , \qquad b + b' + 2 \min\{ a , a' \} \leq p-1 , \]
	and as $\lambda \in C_2$, we further have $a+b \geq p-1$ and $a>0$.
	This implies that
	\[ b + b' + 2 \min\{ a , a' \} \leq p-1 \leq a+b , \]
	and as $a>0$, we conclude that $a' \leq a$ and $b + b' + 2a' \leq p-1$, as required.
	Next observe that
	\begin{flalign} \label{eq:triangle}
		L(\lambda) \otimes L(\mu) \text{ has no composition factors with highest weights in } C_3 && \tag{$\vartriangle$}
	\end{flalign}
	by Corollary \ref{cor:nonCRcriterionGFD}.
	As $\lambda+\mu$ is $p$-restricted by Remark \ref{rem:LeviprestrctedconditionB2}, it follows that $\lambda+\mu \in \overline{C}_2 \cup F_{3,4a} \cup F_{3,4b}$.
	
	First suppose for a contradiction that $\lambda+\mu \in F_{3,4a}$, so $a + a' = p-1$ and $\min\{b,b'\} = 0$.
	As $\lambda \in C_2$, we have $ 1 \leq a \leq p-3$ and $2 \leq b < p-1$, and it follows that $a' \geq 2$ and $b' = 0$.
	Further note that we have $b \leq p-1 - 2 \min\{ a , a' \} \leq p-3$.
	By Remark \ref{rem:LeviprestrctedconditionB2}, the tensor product $L(\lambda) \otimes L(\mu)$ has a composition factor of highest weight $\lambda + \mu - \alpha_1$, and for $2 \leq b < p-3$, we have $\lambda + \mu - \alpha_1 \in C_3$, contradicting \eqref{eq:triangle}.
	For $b = p-3$, we have $a \leq \frac{p-1}{2}$ because $\lambda \in C_2$, and it follows that $a' \geq \frac{p-1}{2}$, contradicting the assumption that $\mu \in C_0$.
	
	Now suppose that $\lambda+\mu \in F_{3,4a} = F_{2,3a}$, so $b+b' = p-1$ and $\min\{ a , a' \}= 0$.
	As before, we must have $1 \leq a \leq p-3$ and $2 \leq b$ because $\lambda \in C_2$, and it follows that $a'=0$ and $b'>0$ (because $\mu \neq 0$).
	By Remark \ref{rem:LeviprestrctedconditionB2}, the simple $G$-module $L(\lambda+\mu-\alpha_2)$ is a composition factor of $L(\lambda)\otimes L(\mu)$, and as we have $\lambda+\mu-\alpha_2 \in C_3$ for $\frac{p-1}{2} \leq a \leq p-3$, the condition \eqref{eq:triangle} implies that $a \leq \frac{p-3}{2}$ and
	\[ 2a + b + b' \leq (p-3) + (p-1) < 2p-3 , \]
	as required.
	Finally, if $\lambda+\mu \in \overline{C}_3$ then we also have $2 ( a + a' ) + (b + b') \leq 2p-3$, as required.
\end{proof}

In order to complete the proof of the necessary conditions for complete reducibility in the case where $\lambda$ and $\mu$ are $p$-regular, it now remains to consider pairs of weights $\lambda , \mu \in X_1$ such that $\lambda \in C_0 \cup C_1$ and $\mu\in C_0$.

\begin{Proposition}\label{prop:B2_CR_C1_C0}
	Let $\lambda\in C_1 \cap X$ and $\mu \in C_0 \cap X$ such that $L(\lambda) \otimes L(\mu)$ is completely reducible.
	Then $\mu$ is reflection small with respect to $\lambda$.
\end{Proposition}
\begin{proof}
	Write $\lambda=a\varpi_1+b\varpi_2$ and $\mu=a'\varpi_1+b'\varpi_2$.
	All composition factors of $L(\lambda) \otimes L(\mu)$ have $p$-restricted highest weights by Remark \ref{rem:LeviprestrctedconditionB2}, and Corollary \ref{cor:nonCRcriterionGFD} implies that
	\begin{flalign} \label{eq:star}
		L(\lambda) \otimes L(\mu) \text{ has no composition factors with highest weights in } C_2 \cup C_3 . && \tag{$\star$}
	\end{flalign}
	In particular, we have $\lambda+\mu \in \overline{C}_1 \cup F_{2,3} \cup F_{3,4a} \cup F_{2,3a}$.
	(Recall that $F_{2,3a} = F_{3,4b}$.)
	
	First suppose for a contradiction that $\lambda+\mu\in F_{2,3a} = F_{3,4b}$.
	Then we must have $b\geq 1$ and $b'\geq 1$, and the simple $G$-module $L(\lambda+\mu-\alpha_2)$ is a composition factor of $L(\lambda)\otimes L(\mu)$ by Remark \ref{rem:LeviprestrctedconditionB2}.
	However, we also have $\lambda+\mu-\alpha_2 \in C_2 \cup C_3$, contradicting \eqref{eq:star}.
	
	Next suppose for a contradiction that $\lambda+\mu\in F_{3,4a}$.
	Then we must have $a\geq 2$ and $a'\geq 2$, whence the simple $G$-modules $L(\lambda+\mu-\alpha_1)$ and $L(\lambda+\mu-2\alpha_1)$ are composition factors of $L(\lambda)\otimes L(\mu)$  by Remark \ref{rem:LeviprestrctedconditionB2}.
	If $b+b' > 1$ then $\lambda + \mu - \alpha_1 \in C_3$, contradicting \eqref{eq:star}.
	If $b+b'=1$ then $\lambda+\mu-2\alpha_1\in C_2$ and if $b+b' = 0$ then $\lambda + \mu - \alpha_1 \in C_2$, again contradicting \eqref{eq:star}.
	
	Finally suppose for a contradiction that $\lambda + \mu \in F_{2,3}$, and note that $a \geq 1$ because $\lambda \in C_1$.
	If we further have $a'\geq 1$ then $p-1 \geq b + b' + 2 \min\{a,a'\} \geq b + b'+ 2$ and $L(\lambda)\otimes L(\mu)$ has a composition factor of highest weight $\lambda + \mu - \alpha_1 \in C_2$ by Remark \ref{rem:LeviprestrctedconditionB2}, contradicting \eqref{eq:star}. 
	For $a'=0$, we claim that $L(\lambda)\otimes L(\mu)$ has a composition factor of highest weight $\nu \coloneqq \lambda+\mu-\varpi_1$.
	Indeed, we have $L(\lambda) = \Delta(\lambda)$ and the Weyl module $\Delta(\mu)$ is uniserial with composition series $[ L(\mu) , L(s\Cdot\mu) ]$ by Subsection \ref{subsec:B2Weylmodulestiltingmodules}.
	As $\lambda+\mu \in F_{2,3}$ and $\mu \in C_0$, we have
	\[ 2(a+a')+(b+b')=2p-3 , \qquad 2a'+b'\leq p-4 , \]
	whence $2a+b \geq p+1$ and $s\Cdot\mu \leq \mu - 4 \varpi_1$.
	This implies that $[ L(\lambda) \otimes L(s\Cdot\mu) : L(\nu) ] = 0$, and using Lemma \ref{lem:tensormultiplicitiesmonotonous}, we obtain
	\[ [ L(\lambda) \otimes L(\mu) : L(\nu) ] = [ \Delta(\lambda) \otimes \Delta(\mu) : L(\nu) ] \geq [ \Delta(\lambda) \otimes \Delta(\mu) : \Delta(\nu) ]_\Delta \geq [ \Delta(\varpi_1) \otimes \Delta(\varpi_2) : \Delta(\varpi_2) ]_\Delta = 1 , \]
	as claimed.
	Now we have $\nu \in C_2$, contradicting \eqref{eq:star}, and we conclude that $\lambda + \mu \in \overline{C}_1$.
	In particular, the weight $\mu$ is reflection small with respect to $\lambda$ by Example \ref{ex:B2reflectionsmall}, as claimed.
\end{proof}

\begin{Proposition}\label{prop:B2_CR_C0_C0}
	Let $\lambda,\mu\in ( C_0 \cap X ) \setminus \{0\}$ such that $L(\lambda)\otimes L(\mu)$ is completely reducible.
	Then $\mu$ is reflection small with respect to $\lambda$.
\end{Proposition}
\begin{proof}
	Write $\lambda=a\varpi_1+b\varpi_2$ and $\mu=a'\varpi_1+b'\varpi_2$, and let $\nu \in X^+$ such that $L(\nu)$ is a composition factor of $L(\lambda) \otimes L(\mu)$.
	As $\lambda,\mu \in C_0$, we have $2 \cdot (a+a')+(b+b')\leq 2(p-4)=2p-8$, and it follows that $\nu \in \overline{C}_0 \cup \overline{C}_1 \cup C_2 \cup F_{2,3a}$.
	By Corollary \ref{cor:nonCRcriterionGFD}, we further have $\nu \notin C_1\cup C_2$, and we conclude that
	\begin{flalign} \label{eq:dagger}
		\text{all composition factors of } L(\lambda) \otimes L(\mu) \text{ have highest weights in } \overline{C}_0 \cup F_{1,2} \cup F_{2,3a} .&& \tag{$\dagger$}
	\end{flalign}
	In particular, we have $\lambda+\mu \in \overline{C}_0 \cup F_{1,2} \cup F_{2,3a}$.
	
	First suppose for a contradiction that $\lambda+\mu\in F_{2,3a}$.
	Then we have $b + b' = p-1$, and as $2a+b \leq p-4$, it follows that $b' \geq 3$ and analogously $b \geq 3$.
	By Remark \ref{rem:LeviprestrctedconditionB2}, the simple $G$-modules $L(\lambda+\mu-\alpha_2)$ and $L(\lambda+\mu-2\alpha_2)$ are composition factors of $L(\lambda)\otimes L(\mu)$, but if $a + a' > 0$ then $\lambda+\mu-\alpha_2\in C_2$, contradicting \eqref{eq:dagger}, and if $a+a'=0$ then $\lambda+\mu-2\alpha_2\in C_1$, contradicting \eqref{eq:dagger}.
	Next suppose for a contradiction that $\lambda+\mu \in F_{1,2}$.
	If $b > 0$ and $b' > 0$ then $L(\lambda+\mu-\alpha_2)$ is a composition factor of $L(\lambda)\otimes L(\mu)$ by Remark \ref{rem:LeviprestrctedconditionB2}, and as before, we have $\lambda+\mu-\alpha_2 \in C_1$, contradicting \eqref{eq:dagger}.
	Therefore, we may assume without loss of generality that $b'=0$, and as $\mu \neq 0$ and $\lambda+\mu \in F_{1,2}$, we further have $a'>0$ and $b>0$.
	Let $\nu \coloneqq \lambda+\mu - \varpi_1$, and observe that by Lemma \ref{lem:tensormultiplicitiesmonotonous}, we have
	\[ [ L(\lambda) \otimes L(\mu) : L(\nu) ] = [ \Delta(\lambda) \otimes \Delta(\mu) : L(\nu) ] \geq [ \Delta(\lambda) \otimes \Delta(\mu) : \Delta(\nu) ]_\Delta \geq [ \Delta(\varpi_2) \otimes \Delta(\varpi_1) : \Delta(\varpi_2) ]_\Delta = 1 . \]
	As before, we have $\nu \in C_1$, contradicting \eqref{eq:dagger}.
	We conclude that $\lambda+\mu \in \overline{C}_0$, whence $\mu$ is reflection small with respect to $\lambda$, as claimed.
\end{proof}

Now we are ready to establish the necessary conditions in Theorem \ref{thm:B2CR}, that is, we prove that for weights $\lambda,\mu \in X_1 \setminus \{ 0 \}$ such that the tensor product $L(\lambda) \otimes L(\mu)$ is completely reducible, the weights $\lambda$ and $\mu$ must satisfy one of the conditions from Table \ref{tab:B2CR}, up to interchanging $\lambda$ and $\mu$.

\begin{proof}[Proof of Theorem \ref{thm:B2CR}, necessary conditions] \label{proof:B2CRnecessary}
	Let $\lambda,\mu \in X_1 \setminus \{0\}$ be such that $L(\lambda) \otimes L(\mu)$ is completely reducible. 
	Suppose first that the weights $\lambda$ and $\mu$ are $p$-regular, and choose indices $i,j \in \{ 0 , 1 , 2 , 3 \}$ such that $\lambda \in C_i$ and $\mu \in C_j$.
	We may assume without loss of generality that $i \geq j$, and we consider the different possibilities for $i$ and $j$ in turn.
	\begin{itemize}
		\item If $\lambda\in C_3$ then $\mu \in C_0$ by Lemma \ref{lem:CR_C2_C3_implies_C0}, and we are in case 8 of Table \ref{tab:B2CR} by Lemma \ref{lem:B2_C3_CR_reflectionsmall}.
		\item If $\lambda\in C_2$ then $\mu \in C_0$ by Lemma \ref{lem:CR_C2_C3_implies_C0}, and we are in case 8 of Table \ref{tab:B2CR} by Lemma \ref{lem:B2_C2_CR_reflectionsmall}.
		\item If $\lambda,\mu\in C_1$ then we are in case 1 of Table \ref{tab:B2CR} by Proposition \ref{prop:B2_CR_C1}.
		\item If $\lambda\in C_1$ and $\mu\in C_0$ then we are in case 8 of Table \ref{tab:B2CR} by Proposition \ref{prop:B2_CR_C1_C0}
		\item If $\lambda,\mu\in C_0$ then we are in case 8 of Table \ref{tab:B2CR} by Proposition \ref{prop:B2_CR_C0_C0}
	\end{itemize}
	Now suppose that $\lambda$ is $p$-singular.
	Then $\lambda$ belongs to a wall of one of the alcoves $C_0$, $C_1$, $C_2$ and $C_3$, and we consider the different walls in turn.
	\begin{itemize}
		\item If $\lambda\in F_{3,4a}$ then we are in case 2 of Table \ref{tab:B2CR} by Proposition \ref{prop:B2_CR_F34a}.
		\item If $\lambda\in F_{2,3a} = F_{3,4b}$ then we are in case 3 of Table \ref{tab:B2CR} by Proposition \ref{prop:B2_CR_F23a}.
		\item If $\lambda\in F_{2,3}$ then we are in case 4 of Table \ref{tab:B2CR} by Proposition \ref{prop:B2_CR_F23}.
		\item If $\lambda\in F_{1,2}$ then we are in cases 5--6 of Table \ref{tab:B2CR} for $p \neq 3$ and in cases 1--3 or 5--6 of Table \ref{tab:B2CR} for $p = 3$ by Proposition \ref{prop:B2_CR_F12}.
		\item If $\lambda\in F_{0,1}$ then we are in cases 7--8 of Table \ref{tab:B2CR} for $p \neq 5$ and in cases 3 or 7--8 of Table \ref{tab:B2CR} for $p = 5$ by Proposition \ref{prop:B2_CR_F01}.
	\end{itemize}
	In all cases, $\lambda$ and $\mu$ satisfy one of the conditions from Table \ref{tab:B2CR}, up to interchanging $\lambda$ and $\mu$.
\end{proof}

\subsection{Proofs: multiplicity freeness, sufficient conditions}
\label{subsec:B2proofsMFsufficient}

In this subsection, we prove that for all pairs of weights $\lambda,\mu \in X_1$ that satisfy one of the conditions from Table \ref{tab:B2MF}, the tensor product $L(\lambda) \otimes L(\mu)$ is multiplicity free (see page \pageref{proof:B2MFsufficient}).
Many cases are in fact already covered by results that were proven in Subsection \ref{subsec:B2proofsCRsufficient}; it remains to consider the conditions 8a--8k and 8c\om--8f\om.
We first consider the conditions 8a, 8b and 8c.

\begin{Proposition}\label{prop:B2MFC0C0}
	Let $\lambda \in \overline{C}_0 \cap X$ and $\mu \in X^+$ such that $\mu$ is reflection small with respect to $\lambda$.
	Then the tensor product $L(\lambda)\otimes L(\mu)$ is multiplicity free if and only if the tensor product $L_{\C}(\lambda)\otimes L_{\C}(\mu)$ of simple $G_\C$-modules is multiplicity free.
\end{Proposition}
\begin{proof}
	Since $\mu$ is reflection small with respect to $\lambda$, we have $\lambda+\mu \in \overline{C}_0$, and it follows that $\mu \in \overline{C}_0$ and $\nu \in \overline{C}_0$ for all $\nu \in X^+$ with $\nu \leq \lambda+\mu$.
	This implies that $L(\nu) = \nabla(\nu)$ for all $\nu \in X^+$ such that the simple $G$-module $L(\nu)$ is a composition factor of $L(\lambda) \otimes L(\mu)$.
	As we further have $L(\lambda) = \nabla(\lambda)$ and $L(\mu) = \nabla(\mu)$, it follows that $L(\lambda) \otimes L(\mu)$ has a good filtration with
	\[ \big[ L(\lambda) \otimes L(\mu) : L(\nu) \big] = \big[ \nabla(\lambda) \otimes \nabla(\mu) : \nabla(\nu) \big]_\nabla = \big[ L_\C(\lambda) \otimes L_\C(\mu) : L_\C(\nu) \big] \]
	 for all $\nu \in X^+$; cf.\ Subsection \ref{subsec:char0}.
	 The claim is immediate from this equality.
\end{proof}

\begin{Remark} \label{rem:C0reflectionsmalltiltingmultiplicity}
	Let $\lambda \in \overline{C}_0 \cap X$ and $\mu \in X^+$ such that $\mu$ is reflection small with respect to $\lambda$.
	As in the proof of Proposition \ref{prop:B2MFC0C0}, we see that $\mu \in \overline{C}_0$ and that all composition factors of $L(\lambda) \otimes L(\mu)$ have highest weights in $\overline{C}_0 \cap X$.
	In particular, we have $L(\lambda) \otimes L(\mu) \cong T(\lambda) \otimes T(\mu)$ and
	\[ \big[ L(\lambda) \otimes L(\mu) : L(\nu) \big] = \big[ L_\C(\lambda) \otimes L_\C(\mu) : L_\C(\nu) \big] = \big[ T(\lambda) \otimes T(\mu) : T(\nu) \big]_\oplus \]
	for all $\nu \in X^+$.
\end{Remark}

\begin{Proposition}\label{prop:B2MFfundamentalweight}
	Let $\lambda\in ( C_1 \cup C_2 \cup C_3 ) \cap X$ and $\mu\in \{\varpi_1,\varpi_2\}$ such that $\mu$ is reflection small with respect to $\lambda$.
	Then the tensor product $L(\lambda)\otimes L(\mu)$ is multiplicity free.
\end{Proposition}
\begin{proof}
	The tensor product $L(\lambda)\otimes L(\mu)$ is completely reducible and the tensor product $L_{\C}(\lambda)\otimes L_{\C}(\mu)$ is multiplicity free, by Theorems \ref{thm:reflectionsmalltensorproduct} and \ref{thm:StembridgeB2}.
	Now the claim follows from Proposition \ref{prop:CRandMFchar0impliesMF}. 
\end{proof}

\begin{Proposition}\label{prop:B2MFC1_b=1}
	Let $\lambda=a\varpi_1+\varpi_2\in C_1 \cap X$ and $\mu=a^\prime \varpi_1 \in X^+$ such that $\mu$ is reflection small with respect to $\lambda$.
	Then the tensor product $L(\lambda)\otimes L(\mu)$ is multiplicity free.
\end{Proposition}
\begin{proof}
	The tensor product $L(\lambda)\otimes L(\mu)$ is completely reducible and the tensor product $L_{\C}(\lambda)\otimes L_{\C}(\mu)$ is multiplicity free, by Theorems \ref{thm:reflectionsmalltensorproduct} and \ref{thm:StembridgeB2}.
	Now the claim follows from Proposition \ref{prop:CRandMFchar0impliesMF}.
\end{proof}

In order to check that the tensor product $L(\lambda) \otimes L(\mu)$ is multiplicity free for all weights $\lambda,\mu \in X_1$ that satisfy the condition 8c\om{} from Table \ref{tab:B2MF}, we need the following preliminary lemma.
Recall from Remark \ref{rem:B2stabC0} that we write $\Omega = \Stab_{W_\mathrm{ext}}(C_0) = \{ e , \omega \}$.

\begin{Lemma}\label{lem:omega_reflection_small}
	Let $\lambda\in C_0 \cap X$ and let $\mu \in X^+$.
	Then $\mu$ is reflection small with respect to $s \Cdot \lambda$ if and only if $\mu$ is reflection small with respect to $s\omega \Cdot \lambda$.
\end{Lemma}
\begin{proof}
	Observe that $F_{1,2} = H_{\alpha_\mathrm{h},1}$ is the unique wall that belongs to the upper closure of $C_1 = s \Cdot C_0$.
	A straightforward computation shows that $( s \Cdot \lambda , \alpha_\mathrm{h}^\vee ) = ( s \omega \Cdot \lambda , \alpha_\mathrm{h}^\vee )$ and therefore
	\[ ( s \Cdot \lambda + w(\mu) , \alpha_\mathrm{h}^\vee ) = ( s \omega \Cdot \lambda + w(\mu) , \alpha_\mathrm{h}^\vee ) \]
	for all $w \in W_\mathrm{fin}$.
	In particular, $\mu$ is reflection small with respect to $s\Cdot\lambda$ if and only if $\mu$ is reflection small with respect to $s\omega\Cdot\lambda$.
\end{proof}

\begin{Lemma} \label{lem:B2ondimensionalweightspacesboundary}
	Let $\lambda = a \varpi_1 + b \varpi_2 \in X^+$ and $\nu = a' \varpi_1 + b' \varpi_2 \in X$.
	If either $b'=2a+b$ or $a'=a+b$ then $\dim L(\lambda)_\nu \leq 1$.
\end{Lemma}
\begin{proof}
	If $b'=2a+b$ then we have
	\[ t(\nu) = \nu - a' \cdot \alpha_1 = - a' \cdot \varpi_1 + (2a'+b') \cdot \varpi_2 = a \varpi_1 + b \varpi_2 + (a+a') \cdot (-\varpi_1+2\varpi_2) = \lambda + (a+a') \cdot \alpha_2 , \]
	and it follows that $\dim L(\lambda)_\nu = \dim L(\lambda)_{t(\nu)} \leq 1$.
	Similarly, if $a' = a+b$ then
	\[ u(\nu) = \nu - b' \cdot \alpha_2 = (a'+b') \cdot \varpi_1 - b' \cdot \varpi_2 = a \varpi_1 + b \varpi_2 + (b+b') \cdot (\varpi_1-\varpi_2) = \lambda + \tfrac{b+b'}{2} \cdot \alpha_1 , \]
	and therefore $\dim L(\lambda)_\nu = \dim L(\lambda)_{u(\nu)} \leq 1$, as claimed.
\end{proof}

\begin{Proposition}\label{prop:B2MFC1_2a+b=p-1}
	Let $\lambda=a\varpi_1+b\varpi_2\in C_1 \cap X$ with $2a+b=p-1$ and $\mu=a^\prime \varpi_1 \in X^+$ such that $\mu$ is reflection small with respect to $\lambda$.
	Then the tensor product $L(\lambda)\otimes L(\mu)$ is multiplicity free.
\end{Proposition}
\begin{proof}
	First observe that the tensor product $L(\lambda) \otimes L(\mu)$ is completely reducible by Theorem \ref{thm:reflectionsmalltensorproduct}.
	Since $\mu$ is reflection small with respect to $\lambda$, we have $a+b+a^\prime \leq p-2$ (cf.\ Example \ref{ex:B2reflectionsmall}), and using the assumption that $2a+b = p-1$, it follows that $a^\prime \leq a-1 \leq \frac{p-3}{2}$.
	If $a^\prime = \frac{p-3}{2}$ then $a = \frac{p-1}{2}$ and $b=0$, so all weight spaces of $L(\lambda)$ are at most one-dimensional by Corollary \ref{cor:B2onedimensionalweightspaces}, and $L(\lambda) \otimes L(\mu)$ is multiplicity free by Corollary \ref{cor:onedimensionalweightspacesMF}.
	Now assume that $a^\prime < \frac{p-3}{2}$, so that $\mu \in C_0$, and consider the weights
	\[ \lambda_0 = s\Cdot\lambda \in C_0  \qquad \lambda' = s\omega \Cdot \lambda_0 = (a+b-1) \cdot \varpi_1 + \varpi_2 \in C_1 ; \]
	cf.\ Remark \ref{rem:B2stabC0}.
	As in the proof of Lemma \ref{lem:MFboundVerlindecoeff}, we have
	\[ \big( L(\lambda)\otimes L(\mu) \big)_\mathrm{reg} \cong \bigoplus_{\nu\in C_0 \cap X} L(s\Cdot \nu)^{\oplus c_{\lambda_0,\mu}^{\nu}} , \]
	and we claim that $c_{\lambda_0,\mu}^\nu \leq 1$ for all $\nu \in C_0 \cap X$.
	Indeed, the weight $\mu$ is reflection small with respect to $\lambda'$ by Lemma \ref{lem:omega_reflection_small}, so $L(\lambda') \otimes L(\mu)$ is multiplicity free by Proposition \ref{prop:B2MFC1_b=1}, and $c_{\lambda_0,\mu}^\nu = c_{\omega\Cdot\lambda_0,\mu}^{\omega\Cdot\nu} \leq 1$ by Lemmas \ref{lem:VerlindecoefficientsFundamentalgroup} and \ref{lem:MFboundVerlindecoeff}.
	As $L(\lambda) \otimes L(\mu)$ is completely reducible, and as all simple $G$-modules with $p$-regular highest weights are regular, it follows that $[ L(\lambda) \otimes L(\mu) : L(\nu) ] \leq 1$ for all weights $\nu \in X^+$ that are $p$-regular.
	Now let $\delta \in X^+$ be a $p$-singular weight and suppose that $L(\delta)$ is a composition factor of $L(\lambda) \otimes L(\mu)$.
	Again by Theorem \ref{thm:reflectionsmalltensorproduct}, we have $\delta \in F_{1,2}$, and this forces that $\lambda+\mu \in F_{1,2}$ and $\delta = \lambda + \mu - c\alpha_1$ for some $c \geq 0$.
	Then Remark \ref{rem:LeviprestrctedconditionB2} shows that $[ L(\lambda) \otimes L(\mu) : L(\nu) ] = 1$, and we conclude that $L(\lambda) \otimes L(\mu)$ is multiplicity free.
\end{proof}

Next we consider the conditions 8d, 8d\om, 8e and 8e\om{} from Table \ref{tab:B2MF}.

\begin{Proposition}\label{prop:B2MFC1_b=0_2a+b=p-2}
	Let $\lambda=a\varpi_1+b\varpi_2\in C_1 \cap X$ with $b=0$ or $2a+b=p-2$, and let $\mu=a'\varpi_1+b'\varpi_2 \in X^+$ such that $\mu$ is reflection small with respect to $\lambda$.
	If $b'\in\{0,1\}$ or $a'=0$ then the tensor product $L(\lambda)\otimes L(\mu)$ is multiplicity free.
\end{Proposition}

\begin{proof}
	If $b=0$ then we can argue as in the proof of Proposition \ref{prop:B2MFC1_b=1}.
	If $2a+b=p-2$ then we can argue as in the proof of Proposition \ref{prop:B2MFC1_2a+b=p-1}.
\end{proof}

The next proposition shows that $L(\lambda) \otimes L(\mu)$ is multiplicity free for all $\lambda,\mu \in X_1$ that satisfy one of the conditions 8f or 8f\om from Table \ref{tab:B2MF}.

\begin{Proposition}\label{prop:B2MFC1_1dimensionalweightspaces}
	Let $\lambda\in\{\frac{p-1}{2}\varpi_1,\frac{p-3}{2}\varpi_1+\varpi_2\}$ and $\mu \in X^+$ such that $\mu$ is reflection small with respect to $\lambda$.
	Then the tensor product $L(\lambda)\otimes L(\mu)$ is multiplicity free.
\end{Proposition}
\begin{proof}
	The tensor product $L(\lambda)\otimes L(\mu)$ is completely reducible by Theorem \ref{thm:reflectionsmalltensorproduct}, and all weight spaces of $L(\lambda)$ are at most one-dimensional by Corollary \ref{cor:B2onedimensionalweightspaces}.
	Now the claim follows from Corollary \ref{cor:onedimensionalweightspacesMF}.
\end{proof}

Finally, it remains to consider the conditions 8g--8k from Table \ref{tab:B2MF}.

\begin{Proposition}\label{prop:B2MFC2_a+b=p-1}
	Let $\lambda=a\varpi_1+b\varpi_2\in C_2 \cap X$ with $a+b=p-1$ and $\mu=a'\varpi_1+b'\varpi_2 \in X^+$ such that $\mu$ is reflection small with respect to $\lambda$.
	If $a'=0$ or $b'=0$ then the tensor product $L(\lambda)\otimes L(\mu)$ is multiplicity free.
\end{Proposition}

\begin{proof}
	The tensor product $L(\lambda) \otimes L(\mu)$ is completely reducible by Theorem \ref{thm:reflectionsmalltensorproduct}, and with
	\[ \lambda_0 = ts \Cdot \lambda = (b-2) \cdot \varpi_1 \in C_0 , \]
	the tensor product $L_\C(\lambda_0) \otimes L_\C(\mu)$ is multiplicity free by Theorem \ref{thm:StembridgeB2}.
	This implies that $c_{\lambda_0,\mu}^\nu \leq 1$ for all $\nu \in C_0 \cap X$ by Lemma \ref{lem:MFboundVerlindecoeff}, and as in the proof of Proposition \ref{prop:B2MFC1_2a+b=p-1}, it follows that
	\[ [L(\lambda) \otimes L(\mu) : L(\nu)] \leq 1 \]
	for all weights $\nu \in X^+$ that are $p$-regular.
	Now let $\delta \in X^+$ be a $p$-singular weight and suppose that $L(\delta)$ is a composition factor of $L(\lambda) \otimes L(\mu)$.
	Again by Theorem \ref{thm:reflectionsmalltensorproduct}, the weight $\delta$ must belong to one of the walls $F_{2,3}$ and $F_{2,3a}$ of $C_2$.
	If $\delta \in F_{2,3}$ then $\lambda+\mu \in F_{2,3}$ and $\delta = \lambda+\mu - c \alpha_1$ for some $c \geq 0$ because $\mu$ is reflection small with respect to $\lambda$, and Remark \ref{rem:LeviprestrctedconditionB2} implies that $[ L(\lambda) \otimes L(\mu) : L(\delta) ] = 1$.
	Now suppose that $\delta \in F_{2,3a}$, and observe that
	\[ 1 \leq [L(\lambda)\otimes L(\mu):L(\delta)] \leq \dim L(\mu)_{\delta-\lambda} \]
	by Lemma \ref{lem:boundformultiplicityintensorproductifCR}.
	Since $\mu$ is reflection small with respect to $\lambda$, we have
	\[ b+2a'+b' \leq p-1 = (\delta,\alpha_2^\vee) = b + (\delta-\lambda,\alpha_2^\vee) , \]
	and as $\delta-\lambda$ is a weight of $L(\mu)$, it follows that $(\delta-\lambda,\alpha_2^\vee) = 2a'+b'$.
	Now Lemmas \ref{lem:boundformultiplicityintensorproductifCR} and \ref{lem:B2ondimensionalweightspacesboundary} imply that
	\[ [ L(\lambda) \otimes L(\mu) : L(\delta) ] \leq \dim L(\mu)_{\delta-\lambda} \leq 1 , \]
	and we conclude that $L(\lambda) \otimes L(\mu)$ is multiplicity free.
\end{proof}

\begin{Proposition}\label{prop:B2MFC3_2a+b=2p-2}
	Let $\lambda=a\varpi_1+b\varpi_2\in C_3 \cap X$ with $2a+b=2p-2$ and $\mu=a'\varpi_1+b'\varpi_2 \in X^+$ such that $\mu$ is reflection small with respect to $\lambda$.
	If $b'\in \{0,1\}$ or $a'=0$ then the tensor product $L(\lambda)\otimes L(\mu)$ is multiplicity free.
\end{Proposition}
\begin{proof}
	The tensor product $L(\lambda) \otimes L(\mu)$ is completely reducible by Theorem \ref{thm:reflectionsmalltensorproduct}, and with
	\[ \lambda_0 = uts \Cdot \lambda = (p-2-a) \cdot \varpi_1 \in C_0 , \]
	the tensor product $L_\C(\lambda_0) \otimes L_\C(\mu)$ is multiplicity free by Theorem \ref{thm:StembridgeB2}.
	This implies that $c_{\lambda_0,\mu}^\nu \leq 1$ for all $\nu \in C_0 \cap X$ by Lemma \ref{lem:MFboundVerlindecoeff}, and as in the proof of Proposition \ref{prop:B2MFC1_2a+b=p-1}, we conclude that
	\[ [ L(\lambda) \otimes L(\mu) : L(\nu) ] \leq 1 \]
	for all weights $\nu \in X^+$ that are $p$-regular.
	Now suppose that $\delta \in X^+$ is a $p$-singular weight such that $L(\delta)$ is a composition factor of $L(\lambda) \otimes L(\mu)$, and note that $\delta$ belongs to one of the walls $F_{3,4a}$ and $F_{3,4b}$ of $C_3$ by Theorem \ref{thm:reflectionsmalltensorproduct}.
	In either case, we can argue as in the proof of Proposition \ref{prop:B2MFC2_a+b=p-1}, using Lemmas \ref{lem:boundformultiplicityintensorproductifCR} and \ref{lem:B2ondimensionalweightspacesboundary} to show that
	\[ [ L(\lambda) \otimes L(\mu) : L(\delta) ] \leq \dim L(\mu)_{\delta-\lambda} \leq 1 , \]
	and it follows that $L(\lambda) \otimes L(\mu)$ is multiplicity free.
\end{proof}

\begin{Proposition}\label{prop:B2MFC3_2a+b=2p-1}
	Let $\lambda=a\varpi_1+b\varpi_2\in C_3 \cap X$ with $2a+b = 2p-1$ and $\mu = a'\varpi_1 \in X^+$ such that $\mu$ is reflection small with respect to $\lambda$.
	Then the tensor product $L(\lambda)\otimes L(\mu)$ is multiplicity free.
\end{Proposition}

\begin{proof}
	The tensor product $L(\lambda) \otimes L(\mu)$ is completely reducible by Theorem \ref{thm:reflectionsmalltensorproduct}, and with
	\[ \lambda_0 = uts \Cdot \lambda = \tfrac{b-3}{2} \cdot \varpi_1 + \varpi_2 \in C_0 , \]
	the tensor product $L_\C(\lambda_0) \otimes L_\C(\mu)$ is multiplicity free by Theorem \ref{thm:StembridgeB2}.
	This implies that $c_{\lambda_0,\mu}^\nu \leq 1$ for all $\nu \in C_0 \cap X$ by Lemma \ref{lem:MFboundVerlindecoeff}, and as in the proof of Proposition \ref{prop:B2MFC1_2a+b=p-1}, we conclude that
	\[ [ L(\lambda) \otimes L(\mu) : L(\nu) ] \leq 1 \]
	for all weights $\nu \in X^+$ that are $p$-regular.
	Now suppose that $\delta \in X^+$ is a $p$-singular weight such that $L(\delta)$ is a composition factor of $L(\lambda) \otimes L(\mu)$, and note that $\delta$ belongs to one of the walls $F_{3,4a}$ and $F_{3,4b}$ of $C_3$ by Theorem \ref{thm:reflectionsmalltensorproduct}.
	In either case, we can argue as in the proof of Proposition \ref{prop:B2MFC2_a+b=p-1}, using Lemmas \ref{lem:boundformultiplicityintensorproductifCR} and \ref{lem:B2ondimensionalweightspacesboundary} to show that
	\[ [ L(\lambda) \otimes L(\mu) : L(\delta) ] \leq \dim L(\mu)_{\delta-\lambda} \leq 1 , \]
	and it follows that $L(\lambda) \otimes L(\mu)$ is multiplicity free.
\end{proof}

Now we are ready to establish the sufficient conditions in Theorem \ref{thm:B2MF}, that is, we prove that for weights $\lambda,\mu \in X_1 \setminus \{ 0 \}$ that satisfy one of the conditions from Table \ref{tab:B2MF}, the tensor product $L(\lambda) \otimes L(\mu)$ is multiplicity free.

\begin{proof}[Proof of Theorem \ref{thm:B2MF}, sufficient conditions]
\label{proof:B2MFsufficient}
	Let $\lambda,\mu \in X_1 \setminus \{0\}$ and suppose that $\lambda$ and $\mu$ satisfy one of the conditions from Table \ref{tab:B2MF}.
	We consider the different conditions in turn.
	\begin{itemize}
	\item Condition 1: $L(\lambda) \otimes L(\mu)$ is multiplicity free by Proposition \ref{prop:B2MFpairsinC1}.
	\item Condition 2: $L(\lambda) \otimes L(\mu)$ is multiplicity free by Proposition \ref{prop:B2MFCRp-1_1}.
	\item Condition 3: $L(\lambda) \otimes L(\mu)$ is multiplicity free by Proposition \ref{prop:B2MFCRp-1_2}.
	\item Condition 4: $L(\lambda) \otimes L(\mu)$ is multiplicity free by Proposition \ref{prop:B2MFCRp-21}.
	\item Condition 5: $L(\lambda) \otimes L(\mu)$ is multiplicity free by Proposition \ref{prop:B2MFCRp-2_2}.
	\item Condition 6: $L(\lambda) \otimes L(\mu)$ is multiplicity free by Proposition \ref{prop:B2MFCRp-2_1}.
	\item Condition 7: $L(\lambda) \otimes L(\mu)$ is multiplicity free by Proposition \ref{prop:B2MFCRp-3_2}.
	\item Condition 8a: $L(\lambda) \otimes L(\mu)$ is multiplicity free by Proposition \ref{prop:B2MFC0C0}.
	\item Condition 8b: $L(\lambda) \otimes L(\mu)$ is multiplicity free by Proposition \ref{prop:B2MFfundamentalweight}.
	\item Condition 8c: $L(\lambda) \otimes L(\mu)$ is multiplicity free by Proposition \ref{prop:B2MFC1_b=1}.
	\item Condition 8c\om: $L(\lambda) \otimes L(\mu)$ is multiplicity free by Proposition \ref{prop:B2MFC1_2a+b=p-1}.
	\item Conditions 8d, 8d\om, 8e or 8e\om: $L(\lambda) \otimes L(\mu)$ is multiplicity free by Proposition \ref{prop:B2MFC1_b=0_2a+b=p-2}.
	\item Conditions 8f or 8f\om: $L(\lambda) \otimes L(\mu)$ is multiplicity free by Proposition \ref{prop:B2MFC1_1dimensionalweightspaces}.
	\item Conditions 8g or 8h: $L(\lambda) \otimes L(\mu)$ is multiplicity free by Proposition \ref{prop:B2MFC2_a+b=p-1}.
	\item Conditions 8i or 8j: $L(\lambda) \otimes L(\mu)$ is multiplicity free by Proposition \ref{prop:B2MFC3_2a+b=2p-2}.
	\item Condition 8k: $L(\lambda) \otimes L(\mu)$ is multiplicity free by Proposition \ref{prop:B2MFC3_2a+b=2p-1}.
	\end{itemize}
	Thus, if $\lambda$ and $\mu$ satisfy one of the conditions from Table \ref{tab:B2MF} then $L(\lambda) \otimes L(\mu)$ is multiplicity free.
\end{proof}

\subsection{Proofs: multiplicity freeness, necessary conditions}\label{subsec:B2proofsMF}
\label{subsec:B2proofsMFnecessary}

In this subsection, we prove that for all $\lambda,\mu \in X_1 \setminus \{0\}$ such that the tensor product $L(\lambda) \otimes L(\mu)$ is multiplicity free, the weights $\lambda$ and $\mu$ satisfy one of the conditions from Table \ref{tab:B2MF}, up to interchanging $\lambda$ and $\mu$ (see page \pageref{proof:B2MFnecessary}).
Recall from Lemma \ref{lem:MFimpliesCR} that $L(\lambda) \otimes L(\mu)$ is multiplicity free only if $L(\lambda) \otimes L(\mu)$ is completely reducible.
By Theorem \ref{thm:A2CR}, we may then assume that $\lambda$ and $\mu$ satisfy one of the conditions from Table \ref{tab:B2CR}, and as the conditions 1--7 in Table \ref{tab:B2CR} match the conditions 1--7 in Table \ref{tab:B2MF}, it remains to consider pairs of weights $\lambda$ and $\mu$ such that $\lambda \in C_0 \cup C_1 \cup C_2 \cup C_3$ and $\mu$ is reflection small with respect to $\lambda$ (cf.\ condition 8 in Table \ref{tab:B2CR}).
The case $\lambda \in C_0$ has already been considered in Proposition \ref{prop:B2MFC0C0}, and we consider the alcoves $C_1,C_2,C_3$ in turn, starting with $C_1$.
We will need the following preliminary results.

\begin{Lemma} \label{lem:B2_C1_reflectionsmall}
	Let $\lambda \in C_1 \cap X$ and $\mu \in X_1$ such that $\mu$ is reflection small with respect to $\lambda$.
	Then we have $\mu \in \overline{C}_0$ and $L(\mu) = \Delta(\mu)$.
\end{Lemma}
\begin{proof}
	Write $\lambda = a \varpi_1 + b \varpi_2$ and $\mu = a' \varpi_1 + b' \varpi_2$.
	The assumption that $\lambda \in C_1 \cap X$ implies that $p \geq 5$ and $2a+b \geq p-2$, and as $\mu$ is reflection small with respect to $\lambda$, we have $a + b + a' + b' \leq p-2$.
	As $p$ is odd, we must have $b \geq 1$ or $2a \geq p-1$, and in either case, it follows that $2a + 2b \geq p-1$.
	We conclude that $a'+b' \leq p-2 - (a+b) \leq p-2 - \tfrac{p-1}{2} = \tfrac{p-3}{2}$ and $2a'+b' \leq 2a'+2b' \leq p-3$.
	In particular, we have $\mu \in \overline{C}_0$ and $L(\mu) = \Delta(\mu)$ (see Subsection \ref{subsec:B2Weylmodulestiltingmodules}), as claimed.
\end{proof}

For $\lambda \in C_1 \cap X$ and $\mu \in X_1$ such that $\mu$ is reflection small with respect to $\lambda$, the composition multiplicities in $L(\lambda) \otimes L(\mu)$ are alternating sums of dimensions of weight spaces in $L(\mu) = \Delta(\mu)$ (see the preceding lemma) by Theorem \ref{thm:reflectionsmalltensorproduct}.
The task of computing the dimensions of these weight spaces in $\Delta(\mu)$ can be reduced to a computation in a Weyl module with an explicit highest weight (which can easily be done using Weyl's character formula or a computer), using a result of M.\ Cavallin which we recall below in the special case $G = \mathrm{Sp}_4(\kk)$.

\begin{Remark} \label{rem:Cavallin}
	Let $\mu = a' \varpi_1 + b' \varpi_2 \in X^+$ and $\nu = \mu - c \alpha_1 - d \alpha_2$ for some $c,d \geq 0$, so that $\nu \leq \mu$.
	Further let $\tilde a = \min\{ a' , c \}$ and $\tilde b = \min\{ b' , d \}$, and define $\tilde \mu = \tilde a \varpi_1 + \tilde b \varpi_2$ and $\tilde \nu = \nu - (\mu - \tilde \mu)$.
	Then we have $\dim \Delta(\mu)_\nu = \dim \Delta(\tilde \mu)_{\tilde \nu}$ by Proposition A in \cite{Cavallin}.
\end{Remark}

The five following propositions establish that for $\lambda \in C_1$ and $\mu \in X^+$ such that $\mu$ is reflection small with respect to $\lambda$, the tensor product $L(\lambda) \otimes L(\mu)$ is multiplicity free only if the weights $\lambda$ and $\mu$ satisfy one of the conditions 8b--8f or 8c*--8f* in Table \ref{tab:B2MF} (see Corollary \ref{cor:B2MFC1conditions}).
We say that a $G$-module \emph{has multiplicity} if it is not multiplicity free.

\begin{Proposition}\label{prop:B2_C1_multiplicity_1}
	Let $\lambda=a\varpi_1+b\varpi_2\in C_1 \cap X$ with $2a+b\geq p-1$, $b\geq 1$, and $\mu=a^\prime\varpi_1+b^\prime\varpi_2\in X^+$ such that $\mu$ is reflection small with respect to $\lambda$.
	If $a^\prime\geq 1$ and $b^\prime\geq 1$ then $L(\lambda)\otimes L(\mu)$ has multiplicity.
\end{Proposition}
\begin{proof}
	Consider the weight $\nu=\lambda+\mu-\varpi_1 = (a+a'-1)\varpi_1+(b+b')\varpi_2$, and observe that $\nu \in C_1$ because $\mu$ is reflection small with respect to $\lambda$ and $2(a+a'-1)+(b+b')>2a+b>p-3$.
	We claim that $[L(\lambda)\otimes L(\mu):L(\nu)]\geq 2$.
	Indeed, by Theorem \ref{thm:reflectionsmalltensorproduct}, we have 
	\begin{align*}
		[ L(\lambda) \otimes L(\mu) : L(\nu) ] & = \sum_{w \in W_{C_1}} (-1)^{\ell(w)} \cdot \dim L(\mu)_{w\Cdot\nu-\lambda} \\
		& = \dim L(\mu)_{\nu-\lambda} - \dim L(\mu)_{s\Cdot\nu - \lambda} - \dim L(\mu)_{u\Cdot\nu - \lambda} + \dim L(\mu)_{su\Cdot\nu - \lambda} \\
		&\geq \dim L(\mu)_{\nu-\lambda} - \dim L(\mu)_{s\Cdot\nu - \lambda} - \dim L(\mu)_{u\Cdot\nu - \lambda} .
	\end{align*}
	The weights in this formula are given by
	\begin{align*}
		\nu -\lambda & = \mu-\varpi_1 , \\
		s\Cdot\nu - \lambda & = \nu - \big( 2(a+a'-1) + (b+b') - (p-3) \big) \cdot \varpi_1 - \lambda = (p-2-b-b'-2a-a')\varpi_1+b'\varpi_2 , \\
		u\Cdot\nu - \lambda & = \nu - (b+b'+1) \cdot \alpha_2 - \lambda = (a'+b'+b)\varpi_1+(-2b-b'-2)\varpi_2 .
	\end{align*}
	We observe that
	\[ s_{\alpha_2}(u\Cdot\nu - \lambda)= (a'-b-2) \varpi_1 + (2b+b'+2) \varpi_2 = \mu + b \alpha_2 - \alpha_1 \nleq \mu \]
because $b \geq 1$, and
\[ s_{\alpha_1+\alpha_2}( \delta - \lambda ) = (a'+2a+b+2-p) \varpi_1 + b' \varpi_2 = \mu + (2a+b-(p-2)) (\alpha_1+\alpha_2) \nleq \mu \]
because $2a+b\geq p-1$, and therefore $L(\mu)_{s\Cdot\nu-\lambda} = 0 $ and $L(\mu)_{u\Cdot\nu-\lambda}=0$.
	By Lemma \ref{lem:B2_C1_reflectionsmall}, we further have $L(\mu) = \Delta(\mu)$, and as $\varpi_1 = \alpha_1+\alpha_2$, we can use Remark \ref{rem:Cavallin} to obtain
	\[ [ L(\lambda) \otimes L(\mu) : L(\nu) ] \geq \dim \Delta(\mu)_{\mu-\varpi_1} = \dim \Delta(\varpi_1+\varpi_2)_{\varpi_2} = 2 , \]
	as claimed.
\end{proof}

\begin{Proposition}\label{prop:B2_C1_multiplicity_2}
	Let $\lambda=a\varpi_1+b\varpi_2\in C_1 \cap X$ with $2a+b\geq p-1$, $b\geq 1$, and $\mu=b^\prime\varpi_2\in X^+$ such that $\mu$ is reflection small with respect to $\lambda$.
	If $b^\prime \geq 2$ then $L(\lambda) \otimes L(\mu)$ has multiplicity.
\end{Proposition}
\begin{proof}
	Consider the weight $\nu = \lambda + \mu - 2\varpi_2 = a \varpi_1 + (b+b'-2) \varpi_2$, and observe that $\nu \in C_1$ because $\mu$ is reflection small with respect to $\lambda$ and $2a+(b+b'-2)\geq 2a+b>p-3$.
	We claim that $[L(\lambda)\otimes L(\mu):L(\nu)]\geq 2$.
	Indeed, as in the proof of Proposition \ref{prop:B2_C1_multiplicity_1}, we have
	\[ [ L(\lambda) \otimes L(\mu) : L(\nu) ] \geq \dim L(\mu)_{\nu-\lambda} - \dim L(\mu)_{s\Cdot\nu - \lambda} - \dim L(\mu)_{u\Cdot\nu - \lambda} , \]
	and the weights on the right hand side are given by
	\begin{align*}
		\nu -\lambda & = \mu- 2\varpi_2 , \\
		s\Cdot\nu - \lambda & = \nu - \big( 2a + (b+b'-2) - (p-3) \big) \cdot \varpi_1 - \lambda = ( p-1 - 2a - b- b' ) \varpi_1 + (b'-2) \varpi_2 , \\
		u\Cdot\nu - \lambda & = \nu - (b+b'-1) \cdot \alpha_2 - \lambda = ( b + b' - 1) \varpi_1 + (-2b-b')\varpi_2 .
	\end{align*}
	We observe that
	\[ s_{\alpha_2}( u\Cdot\nu - \lambda ) = (a'-b-1) \varpi_1 + (2b+b') \varpi_2 = \mu + (b-1) \alpha_2 - \alpha_1 , \]
	and thus $s_{\alpha_2}(su\Cdot\delta-\lambda)\nleq \mu$ if $b\geq 2$.
	If $b=1$ then $s_{\alpha_2}(su\Cdot\delta-\lambda)=\mu-\alpha_1$ is not a weight of $\Delta(b'\varpi_2)$, and in either case, we conclude that $\dim L(\mu)_{u\Cdot\nu - \lambda}=0$.
	Further observe that
	\[ s_{\alpha_1+\alpha_2}(s\Cdot\nu-\lambda) = (2a+b+3-p) \varpi_1 + (b'-2) \varpi_2 = \mu + (2a+b-(p-2))\alpha_1+(2a+b-(p-2)-1)\alpha_2\nleq \mu\]
	because $2a+b\geq p-1$, and therefore $\dim L(\mu)_{s\Cdot\nu - \lambda}=0$.
	By Lemma \ref{lem:B2_C1_reflectionsmall}, we further have $L(\mu) = \Delta(\mu)$, and as $2 \varpi_2 = \alpha_1 + 2 \varpi_2$, we can use Remark \ref{rem:Cavallin} to obtain
	\[ [ L(\lambda) \otimes L(\mu) : L(\nu) ] \geq \dim \Delta(\mu)_{\mu - 2 \varpi_2} = \dim \Delta(2\varpi_2)_0 = 2 , \]
	as claimed.
\end{proof}

\begin{Proposition}\label{prop:B2_C1_multiplicity_3}
	Let $\lambda=a\varpi_1+b\varpi_2\in C_1 \cap X$ with $2a+b\geq p$, $b\geq 2$, and $\mu=a^\prime\varpi_1\in X^+$ such that $\mu$ is reflection small with respect to $\lambda$.
	If $a^\prime \geq 2$ then $L(\lambda)\otimes L(\mu)$ has multiplicity.
\end{Proposition}
\begin{proof}
	Consider the weight $\nu=\lambda+\mu-2\varpi_1 = (a+a'-2) \varpi_1+b\varpi_2$, and observe that $\nu \in C_1$ because $\mu$ is reflection small with respect to $\lambda$ and $2(a+a'-2)+b\geq 2a+b>p-3$.
	We claim that $[L(\lambda)\otimes L(\mu):L(\nu)]\geq 2$.
	Indeed, as in the proof of Proposition \ref{prop:B2_C1_multiplicity_1}, we have
	\[ [ L(\lambda) \otimes L(\mu) : L(\nu) ] \geq \dim L(\mu)_{\nu-\lambda} - \dim L(\mu)_{s\Cdot\nu - \lambda} - \dim L(\mu)_{u\Cdot\nu - \lambda} , \]
	and the weights on the right hand side are given by
	\begin{align*}
		\nu -\lambda & = \mu - 2\varpi_1 , \\
		s\Cdot\nu - \lambda & = \nu - \big( 2 (a+a'-2) + b - (p-3) \big) \cdot \varpi_1 - \lambda = ( p-1 - 2a - a' - b ) \varpi_1 , \\
		u\Cdot\nu - \lambda & = \nu - (b+1) \cdot \alpha_2 - \lambda = ( a' + b - 1 ) \varpi_1 + ( - 2b - 2 ) \varpi_2 .
	\end{align*}
	We observe that
	\[ s_{\alpha_2}(u\Cdot\nu - \lambda ) = (a'-b-3)\varpi_1+(2b+2)\varpi_2=\mu+(b-1)\alpha_2-2\alpha_1\nleq \mu\]
	because $b\geq 2$ and
	\[s_{\alpha_1+\alpha_2}(\delta-\lambda)=(2a+b+a'+1-p)\varpi_1=\mu+(2a+b-(p-1))(\alpha_1+\alpha_2)\nleq \mu\]
	because $2a+b\geq p$, and therefore $L(\mu)_{s\Cdot\nu-\lambda} = 0 $ and $L(\mu)_{u\Cdot\nu-\lambda}=0$.
	By Lemma \ref{lem:B2_C1_reflectionsmall}, we further have $L(\mu) = \Delta(\mu)$, and as $2 \varpi_1 = 2 \alpha_1 + 2 \alpha_2$, we can use Remark \ref{rem:Cavallin} to obtain
	\[ [ L(\lambda) \otimes L(\mu) : L(\nu) ] \geq \dim \Delta(\mu)_{\mu - 2 \varpi_1} = \dim \Delta(2\varpi_1)_0 = 2 , \]
	as claimed.
\end{proof}

\begin{Proposition}\label{prop:B2_C1_multiplicity_3a}
	Let $\lambda=a\varpi_1+b\varpi_2\in C_1 \cap X$ with $2a+b=p-2$, $b\neq 1$, and $\mu=a^\prime\varpi_1+b^\prime\varpi_2 \in X^+$ such that $\mu$ is reflection small with respect to $\lambda$.
	If $b^\prime \geq 2$ and $a^\prime \geq 1$ then $L(\lambda) \otimes L(\mu)$ has multiplicity.
\end{Proposition}
\begin{proof}
	Consider the weight $\nu=\lambda+\mu-2\varpi_2 = (a+a') \varpi_1 + (b+b'-2) \varpi_2$, and observe that $\nu \in C_1$ because $\mu$ is reflection small with respect to $\lambda$ and $2(a+a')+(b+b'-2)>2a+b>p-3$.
	We claim that $[L(\lambda)\otimes L(\mu):L(\nu)]\geq 2$.
	Indeed, as in the proof of Proposition \ref{prop:B2_C1_multiplicity_1}, we have
	\[ [ L(\lambda) \otimes L(\mu) : L(\nu) ] \geq \dim L(\mu)_{\nu-\lambda} - \dim L(\mu)_{s\Cdot\nu - \lambda} - \dim L(\mu)_{u\Cdot\nu - \lambda} , \]
	and the weights on the right hand side are given by
	\begin{align*}
		\nu -\lambda & = \mu - 2\varpi_2 , \\
		s\Cdot\nu - \lambda & = \nu - \big( 2 (a+a') + (b+b'-2) - (p-3) \big) \cdot \varpi_1 - \lambda = ( 1 - a' - b' ) \varpi_1 + ( b' - 2 ) \varpi_2 , \\
		u\Cdot\nu - \lambda & = \nu - (b+1) \cdot \alpha_2 - \lambda = ( a' +  b + b' - 1 ) \varpi_1 + ( - 2b - b' ) \varpi_2 .
	\end{align*}
	We observe that
	\[ s_{\alpha_2}(u\Cdot\nu-\lambda) = (a'-b-1) \varpi_1 + (2b+b') \varpi_2 = \mu + (b-1) \alpha_2 - \alpha_1 \nleq \mu \]
	because $b\geq 2$ and
	\[ s_{\alpha_1+\alpha_2}(s\Cdot\nu-\lambda) = (a'+1) \varpi_1 + (b'-2) \varpi_2 = \mu - \alpha_2 , \]
	and as $L(\mu) = \Delta(\mu)$ by Lemma \ref{lem:B2_C1_reflectionsmall}, it follows that $\dim L(\mu)_{u\Cdot\nu-\lambda} = 0$ and
	\[ \dim L(\mu)_{s\Cdot\nu-\lambda} = \dim L(\mu)_{\mu - \alpha_2} = \dim \Delta(\mu)_{\mu - \alpha_2} = 1 . \]
	As $2 \varpi_2 = \alpha_1 + 2 \alpha_2$, we can further use Remark \ref{rem:Cavallin} to obtain
	\[ \dim L(\mu)_{\nu - \lambda} = \dim \Delta(\mu)_{\mu-2\varpi_2} = \dim \Delta(\varpi_1+2\varpi_2)_{\varpi_1} = 3 , \]
	and we conclude that $[ L(\lambda) \otimes L(\mu) : L(\nu) ] \geq 3 - 1 = 2$, as claimed.
\end{proof}

\begin{Corollary} \label{cor:B2_C1_multiplicity_3a}
		Let $\lambda=a\varpi_1+b\varpi_2\in C_1 \cap X$ with $2a+b=p-2$, $b\neq 1$, and $\mu=a^\prime\varpi_1+b^\prime\varpi_2 \in X^+$ such that $\mu$ is reflection small with respect to $\lambda$.
	If $b^\prime \geq 2$ and $a^\prime \geq 1$ then $\mu \in C_0$ and there is a weight $\nu \in C_0 \cap X$ such that $c_{s\Cdot\lambda,\mu}^\nu \geq 2$.
\end{Corollary}
\begin{proof}
	The tensor product $L(\lambda) \otimes L(\mu)$ is completely reducible by Theorem \ref{thm:reflectionsmalltensorproduct}, and by the proof of Proposition \ref{prop:B2_C1_multiplicity_3a}, there is a $p$-regular weight $\nu \in C_0$ such that $[ L(\lambda) \otimes L(\mu) : L(s\Cdot\nu) ] \geq 2$.
	This implies that $\mu$ is $p$-regular by Lemma \ref{lem:nonCRcriterionpregularweight}, and as $\mu \in \overline{C}_0$ by Lemma \ref{lem:B2_C1_reflectionsmall}, we conclude that $\mu \in C_0$.
	Now Lemma \ref{lem:MFboundVerlindecoeff} yields $c_{s\Cdot\lambda,\mu}^\nu = [ L(\lambda) \otimes L(\mu) : L(s\Cdot\nu) ] \geq 2$, as claimed.
\end{proof}

\begin{Proposition}\label{prop:B2_C1_multiplicity_3b}
	Let $\lambda=a\varpi_1\in C_1 \cap X$ with $a\neq \frac{p-1}2$ and $\mu=a^\prime\varpi_1+b^\prime\varpi_2\in X^+$ such that $\mu$ is reflection small with respect to $\lambda$.
	If $b^\prime \geq 2$ and $a^\prime \neq 0$ then $L(\lambda)\otimes L(\mu)$ has multiplicity.
\end{Proposition}
\begin{proof}
	Let $\lambda_0 = s \Cdot \lambda = (p-3-a) \cdot \varpi_1 \in C_0$, and observe that $\mu$ is reflection small with respect to the weight $s \omega \Cdot \lambda_0 \in C_1$ by Lemma \ref{lem:omega_reflection_small}.
	Using Remark \ref{rem:B2stabC0}, we compute
	\[ \omega \Cdot \lambda_0 = (p-3-a) \cdot \varpi_1 + (2a-p+2) \cdot \varpi_2 , \qquad s \omega \Cdot \lambda_0 = (p-2-a) \cdot \varpi_1 + (2a-p+2) \cdot \varpi_2 , \]
	where $2 \cdot (p-2-a) + (2a-p+2) = p-2$ and $2a-p+2 \neq 1$.
	Now Corollary \ref{cor:B2_C1_multiplicity_3a} implies that $\mu \in C_0$ and there is a weight $\nu \in C_0 \cap X$ such that $c_{\omega\Cdot\lambda_0,\mu}^\nu \geq 2$.
	Using Lemmas \ref{lem:MFboundVerlindecoeff} and \ref{lem:VerlindecoefficientsFundamentalgroup}, we conclude that
	\[ [ L(\lambda) \otimes L(\mu) : L(s\omega\Cdot\nu) ] \geq c_{\lambda_0,\mu}^{\omega\Cdot\nu} = c_{\omega\Cdot\lambda_0,\mu}^\nu \geq 2 , \]
	as required.
\end{proof}

\begin{Corollary} \label{cor:B2MFC1conditions}
	Let $\lambda \in C_1 \cap X$ and $\mu \in X^+$ such that $\mu$ is reflection small with respect to $\lambda$.
	If the tensor product $L(\lambda) \otimes L(\mu)$ is multiplicity free then the weights $\lambda$ and $\mu$ satisfy one of the conditions 8b--8f or 8b\om--8f\om in Table \ref{tab:B2MF}.
\end{Corollary}
\begin{proof}
	Let us write $\lambda = a\varpi_1 + b\varpi_2$ and $\mu = a^\prime\varpi_1 + b^\prime\varpi_2$, with $a,b,a^\prime,b^\prime \in \Z_{\geq 0}$.
	The multiplicity freeness of $L(\lambda) \otimes L(\mu)$ imposes the following conditions on $\lambda$ and $\mu$:
	\begin{enumerate}
	\item If $2a+b \geq p-1$ and $b \geq 1$ then $a^\prime = 0$ or $b^\prime = 0$ by Proposition \ref{prop:B2_C1_multiplicity_1};
	\item If $2a+b \geq p-1$ and $b \geq 1$ then $b^\prime \leq 1$ by Proposition \ref{prop:B2_C1_multiplicity_2};
	\item If $2a+b \geq p$ and $b \geq 2$ then $a^\prime \leq 1$ by Proposition \ref{prop:B2_C1_multiplicity_3};
	\item If $2a+b = p-2$ and $b \neq 1$ then either $b^\prime \leq 1$ or $a^\prime = 0$ by Proposition \ref{prop:B2_C1_multiplicity_3a};
	\item If $b=0$ and $a \neq \frac{p-1}{2}$ then either $b^\prime \leq 1$ or $a^\prime = 0$ by Proposition \ref{prop:B2_C1_multiplicity_3b}.
	\end{enumerate}
	We consider the different values of $2a+b$ and of $b$ in turn.
	\begin{itemize}
	\item Suppose that $2a+b=p-2$.
	If $b = 1$ then we are in case 8f\om of Table \ref{tab:B2MF}, and if $b \neq 1$ then we are in one of the cases 8d\om or 8e\om of Table \ref{tab:B2MF} by (4).
	\item Suppose that $b=0$.
	If $2a = p-1$ then we are in case 8f of Table \ref{tab:B2MF}.
	If $2a \neq p-1$ then either $b^\prime  \leq 1$ or $a^\prime = 0$ by (5), so we are in one of the cases 8d or 8e of Table \ref{tab:B2MF}.
	\item Suppose that $2a+b \geq p-1$ and $b \geq 1$.
	Then by (1) and (2), we have either $b^\prime=0$ or $(a^\prime,b^\prime) = (0,1)$.
	If $(a^\prime,b^\prime) = (0,1)$ then we are in case 8b of Table \ref{tab:B2MF}, so now assume that $b^\prime = 0$.
	If $b=1$ or $2a+b=p-1$ then it follows that we are in one of the cases 8c or 8c\om of Table \ref{tab:B2MF}.
	If $b \geq 2$ and $2a+b \geq p$ then $a^\prime \leq 1$ by (3), so we are again in case 8b of Table \ref{tab:B2MF}.
	\end{itemize}
	In all of the three cases above, one of the conditions 8b--8f or 8c\om--8f\om is satisfied, as claimed.
\end{proof}

Our next goal is to prove that for $\lambda \in C_2 \cap X$ and $\mu \in X^+$ such that $\mu$ is reflection small with respect to $\lambda$, the tensor product $L(\lambda) \otimes L(\mu)$ is multiplicity free only if one of the conditions 8b, 8g or 8h from Table \ref{tab:B2MF} is satisfied.
This will follow from Propositions \ref{prop:B2_C2_multiplicity} and \ref{prop:B2_C2_multiplicity_2} below (see also Corollary \ref{cor:B2MFC2conditions}), which will be proven after some preliminary lemmas.

\begin{Lemma}\label{lem:B2_reflectionsmallC2C3}
	Let $\lambda\in ( C_2\cup C_3 ) \cap X$ and $\mu \in X^+$ such that $\mu$ is reflection small with respect to $\lambda$.
	Further let $\lambda_0 \in X^+$ be the unique weight with $\lambda_0 \in C_0 \cap W_\mathrm{aff} \Cdot \lambda$.
	Then $\mu$ is reflection small with respect to~$\lambda_0$.
\end{Lemma}
\begin{proof}
	Using the labeling of the walls from Subsection \ref{subsec:B2setup}, the walls that belong to the upper closure of the alcove $C_2 = st \Cdot C_0$ are precisely $F_{2,3} = H_{\alpha_{\mathrm{hs},2}}$ and $F_{2,3a} = H_{\alpha_2,1} = st \Cdot F_{0,1}$, where $F_{0,1} = H_{\alpha_\mathrm{hs},1}$ is the unique wall that belongs to the upper closure of $C_0$.
	This implies that for $x \in X_\R$, we have $(x+\rho,\alpha_\mathrm{hs}^\vee) \leq p$ if and only if $( st \Cdot x + \rho , \alpha_2^\vee ) \leq p$.
	Since $st\Cdot \big( \lambda_0 + W_\mathrm{fin}(\mu) \big) = st \Cdot \lambda_0 + W_\mathrm{fin}(\mu)$ for all $\mu \in X^+$, we conclude that $\mu$ is reflection small with respect to $\lambda_0$ only if $\mu$ is reflection small with respect to $st\Cdot\lambda_0$.
	Analogously, we see that $\lambda_0$ is reflection small with respect to $\mu$ only if $\mu$ is reflection small with respect to $stu\Cdot\lambda_0$, and the claim follows because we have $\lambda \in \{ st \Cdot \lambda_0 , stu \Cdot \lambda_0 \}$ by assumption.
\end{proof}

\begin{Lemma}\label{lem:B2_multiplicityC0C0}
	Let $\lambda \in C_0 \cap X$ and $\mu \in X^+$ such that $\mu$ is reflection small with respect to $\lambda$.
	If the tensor product $L_{\C}(\lambda) \otimes L_{\C}(\mu)$ of simple $G_\C$-modules has multiplicity then we have $\mu \in C_0$ and there is a weight $\nu\in C_0 \cap X$ such that $c_{\lambda,\mu}^\nu\geq 2$.
\end{Lemma}
\begin{proof}
	Suppose that $L_\C(\lambda) \otimes L_\C(\mu)$ has multiplicity.
	As $\mu$ is reflection small with respect to $\lambda$, we have $\lambda + \mu \in \overline{C}_0$ and $\mu \in \overline{C}_0$, and by Remark \ref{rem:C0reflectionsmalltiltingmultiplicity}, there exists a weight $\nu \in \overline{C}_0 \cap X$ such that
	\[ 2 \leq [ L_\C(\lambda) \otimes L_\C(\mu) : L_\C(\nu) ] = \big[ T(\lambda) \otimes T(\mu) : T(\nu)  \big]_\oplus . \]
	This implies that $\lambda \neq 0$ and $\mu \in C_0$, and we claim that $\nu \in C_0$.
	Indeed, all weights $\delta \in (\overline{C}_0 \setminus C_0) \cap X^+$ belong to the wall $F_{0,1}$ of $C_0$, so $\delta = \lambda+\mu-c\alpha_2$ for some $c \geq 0$ and
	\[ [ L_\C(\lambda) \otimes L_\C(\mu) : L_\C(\delta) ] \leq 1 . \]
	In particular, we have $\delta \neq \nu$ and so $\nu \in C_0$ and $c_{\lambda,\mu}^\nu \geq 2$.
\end{proof}

\begin{Lemma}\label{lem:B2_C2C3multiplicity}
	Let $\lambda \in ( C_2\cup C_3 ) \cap X$ and $\mu \in X^+$ such that $\mu$ is reflection small with respect to $\lambda$.
	Let $\lambda_0 \in X^+$ be the unique weight in $C_0\cap W_\mathrm{aff}\Cdot \lambda$.
	If $L_{\C}(\lambda_0)\otimes L_{\C}(\mu)$ has multiplicity then $L(\lambda)\otimes L(\mu)$ has multiplicity.
\end{Lemma}
\begin{proof}
	Suppose that $L_{\C}(\lambda_0)\otimes L_{\C}(\mu)$ has multiplicity.
	By Lemma \ref{lem:B2_reflectionsmallC2C3}, the weight $\mu$ is reflection small with respect to $\lambda_0$, and Lemma \ref{lem:B2_multiplicityC0C0} implies that there exists a weight $\nu\in C_0 \cap X$ such that $c_{\lambda_0,\mu}^\nu\geq 2$.
	Then $L(\lambda) \otimes L(\mu)$ has multiplicity by Lemma \ref{lem:MFboundVerlindecoeff}.
\end{proof}

\begin{Proposition}\label{prop:B2_C2_multiplicity}
	Let $\lambda=a\varpi_1+b\varpi_2\in C_2 \cap X$ with $a+b=p-1$ and $\mu=a^\prime\varpi_1+b^\prime\varpi_2 \in X^+$ such that $\mu$ is reflection small with respect to $\lambda$.
	If $a'\neq0$ and $b'\neq 0$ then $L(\lambda)\otimes L(\mu)$ has multiplicity.
\end{Proposition}

\begin{proof}
	Let $\lambda_0 \in X^+$ be the unique weight with $\lambda_0 \in C_0\cap W_\mathrm{aff}\Cdot \lambda$.
	By Lemma \ref{lem:B2_C2C3multiplicity}, it suffices to show that $L_{\C}(\lambda_0)\otimes L_{\C}(\mu)$ has multiplicity.
	Observe that
	\[\lambda_0=(a+b-p+1)\varpi_1+(2p-4-2a-b)\varpi_2=(p-3-a)\varpi_2.\]
	Since $\mu$ is reflection small with respect to $\lambda$, we have
	\[ 2p-3 \geq 2 (a+a^\prime) + (b+b^\prime) = (a+b) + (a+2a^\prime+b) \geq (p-1) + (a+3) , \]
	so $p-3-a\geq 2$, and the claim follows from Theorem \ref{thm:StembridgeB2}.
\end{proof}

\begin{Proposition}\label{prop:B2_C2_multiplicity_2}
	Let $\lambda=a\varpi_1+b\varpi_2\in C_2 \cap X$ with $a+b>p-1$ and $\mu=a'\varpi_1+b'\varpi_2 \in X^+$ such that $\mu$ is reflection small with respect to $\lambda$.
	If $\mu\notin\{\varpi_1,\varpi_2\}$ then $L(\lambda)\otimes L(\mu)$ has multiplicity.
\end{Proposition}

\begin{proof}
	Let $\lambda_0 \in X^+$ be the unique weight with $\lambda_0 \in C_0\cap W_\mathrm{aff}\Cdot \lambda$.
	By Lemma \ref{lem:B2_C2C3multiplicity}, it suffices to show that $L_{\C}(\lambda_0)\otimes L_{\C}(\mu)$ has multiplicity.
	Observe that \[\lambda_0=(a+b-p+1)\varpi_1+(2p-4-2a-b)\varpi_2=:c\varpi_1+d\varpi_2 , \]
	with $c=a+b-p+1\geq 1$ and $d=2p-4-2a-b$.
	If $d\geq 2$ then the claim follows from Theorem \ref{thm:StembridgeB2} because $\mu\notin\{\varpi_1,\varpi_2\}$, so now assume that $d \leq 1$.
	Then $2a+b\geq 2p-5$, and since $\mu$ is reflection small with respect to $\lambda$, we further have $2(a+a')+(b+b')\leq 2p-3$.
	Combining the two inequalities, it follows that $2a^\prime+b^\prime \leq 2$, and since $\mu\notin\{\varpi_1,\varpi_2\}$, this forces that $\mu=2\varpi_2$ and $d=1$.
	As before, Theorem \ref{thm:StembridgeB2} shows that $L_{\C}(\lambda_0)\otimes L_{\C}(\mu)$ has multiplicity.
\end{proof}

\begin{Corollary} \label{cor:B2MFC2conditions}
	Let $\lambda \in C_2 \cap X$ and $\mu \in X^+$ such that $\mu$ is reflection small with respect to $\lambda$.
	If the tensor product $L(\lambda) \otimes L(\mu)$ is multiplicity free then the weights $\lambda$ and $\mu$ satisfy one of the conditions 8b, 8g or 8h in Table \ref{tab:B2MF}.
\end{Corollary}
\begin{proof}
	Let us write $\lambda = a\varpi_1 + b\varpi_2$ and $\mu = a^\prime\varpi_1 + b^\prime\varpi_2$, with $a,b,a^\prime,b^\prime \in \Z_{\geq 0}$.
	If $a+b = p-1$ then we have either $a^\prime = 0$ or $b^\prime = 0$ by Proposition \ref{prop:B2_C2_multiplicity}, so we are in one of the cases 8g or 8h of Table \ref{tab:B2MF}.
	If $a+b > p-1$ then $\mu \in \{ \varpi_1 , \varpi_2 \}$ by Proposition \ref{prop:B2_C2_multiplicity_2}, and we are in case 8b of Table \ref{tab:B2MF}.
\end{proof}

Now it remains to show that for $\lambda \in C_3 \cap X$ and $\mu \in X^+$ such that $\mu$ is reflection small with respect to $\lambda$, the tensor product $L(\lambda) \otimes L(\mu)$ is multiplicity free only if one of the conditions 8b, 8i, 8j or 8k from Table \ref{tab:B2MF} is satisfied.
This will be proven in the three propositions below; see also Corollary \ref{cor:B2MFC3conditions}.

\begin{Proposition}\label{prop:B2_C3_multiplicity}
	Let $\lambda=a\varpi_1+b\varpi_2\in C_3 \cap X$ with $2a+b=2p-2$ and $\mu=a^\prime\varpi_1+b^\prime\varpi_2 \in X^+$ such that $\mu$ is reflection small with respect to $\lambda$.
	If $b^\prime \geq 2$ and $a^\prime \neq 0$ then $L(\lambda) \otimes L(\mu)$ has multiplicity.
\end{Proposition}
\begin{proof}
	Let $\lambda_0 \in X^+$ be the unique weight with $\lambda_0 \in C_0\cap W_\mathrm{aff}\Cdot \lambda$.
	Observe that
	\[ \lambda_0=(p-2-a)\varpi_1+(2a+b-2p+2)\varpi_2=(p-2-a)\varpi_1. \]
	Since $\mu$ is reflection small with respect to $\lambda$, we have $p-1 \geq a+a^\prime+b^\prime \geq a + 3$, and it follows that $p-a-2 \geq 2$.
	Now Theorem \ref{thm:StembridgeB2} shows that $L_{\C}(\lambda_0)\otimes L_{\C}(\mu)$ has multiplicity, and the claim follows from Lemma \ref{lem:B2_C2C3multiplicity}.
\end{proof}

\begin{Proposition}\label{prop:B2_C3_multiplicity_2}
	Let $\lambda=a\varpi_1+b\varpi_2\in C_3 \cap X$ with $2a+b=2p-1$ and $\mu=a^\prime\varpi_1+b^\prime\varpi_2 \in X^+$ such that $\mu$ is reflection small with respect to $\lambda$.
	If $b^\prime\neq 0$ and $\mu \neq \varpi_2$ then $L(\lambda)\otimes L(\mu)$ has multiplicity.
\end{Proposition}
\begin{proof}
	Let $\lambda_0 \in X^+$ be the unique weight with $\lambda_0 \in C_0\cap W_\mathrm{aff}\Cdot \lambda$.
	Observe that
	\[ \lambda_0 = (p-2-a)\varpi_1+(2a+b-2p+2)\varpi_2 = (p-2-a)\varpi_1+\varpi_2 . \]
	Since $\mu$ is reflection small with respect to $\lambda$, we have $p-1 \geq a+a^\prime+b^\prime$, and the assumptions that $b^\prime \neq 0$ and $\mu \neq \varpi_2$ imply that $a < p-2$ and $p-2-a > 0$.
	Now Theorem \ref{thm:StembridgeB2} shows that $L_{\C}(\lambda_0)\otimes L_{\C}(\mu)$ has multiplicity, and the claim follows from Lemma \ref{lem:B2_C2C3multiplicity}.
\end{proof}

\begin{Proposition}\label{prop:B2_C3_multiplicity_3}
	Let $\lambda=a\varpi_1+b\varpi_2\in C_3 \cap X$ with $2a+b\geq2p$ and $\mu=a^\prime\varpi_1+b^\prime\varpi_2 \in X^+$ such that $\mu$ is reflection small with respect to $\lambda$.
	If $\mu\notin\{\varpi_1,\varpi_2\}$ then $L(\lambda)\otimes L(\mu)$ has multiplicity.
\end{Proposition}

\begin{proof}
	Let $\lambda_0 \in X^+$ be the unique weight with $\lambda_0 \in C_0\cap W_\mathrm{aff}\Cdot \lambda$.
	Observe that
	\[ \lambda_0 = (p-2-a)\varpi_1+(2a+b-2p+2)\varpi_2 = c\varpi_1+d\varpi_2 , \]
	where $c = p-2-a$ and $d = 2a+b-2p+2 \geq 2$.
	As $\mu$ is reflection small with respect to $\lambda$, we have $p-1 \geq a + a^\prime + b^\prime$, and the assumption that $\mu \notin \{ \varpi_1 , \varpi_2 \}$ implies that $a < p-2$ and $c > 0$.
	Now Theorem~\ref{thm:StembridgeB2} shows that $L_{\C}(\lambda_0)\otimes L_{\C}(\mu)$ has multiplicity, and the claim follows from Lemma \ref{lem:B2_C2C3multiplicity}.
\end{proof}

\begin{Corollary} \label{cor:B2MFC3conditions}
	Let $\lambda \in C_3 \cap X$ and $\mu \in X^+$ such that $\mu$ is reflection small with respect to $\lambda$.
	If the tensor product $L(\lambda) \otimes L(\mu)$ is multiplicity free then the weights $\lambda$ and $\mu$ satisfy one of the conditions 8b, 8i, 8j or 8k in Table \ref{tab:B2MF}.
\end{Corollary}
\begin{proof}
	Let us write $\lambda = a\varpi_1 + b\varpi_2$ and $\mu = a^\prime\varpi_1 + b^\prime\varpi_2$, with $a,b,a^\prime,b^\prime \in \Z_{\geq 0}$.
	If $2a+b = 2p-2$ then either $a^\prime=0$ or $b^\prime \leq 1$ by Proposition \ref{prop:B2_C3_multiplicity}, and it follows that we are in one of the cases 8i or 8j of Table \ref{tab:B2MF}.
	If $2a+b = 2p-1$ then either $\mu = \varpi_2$ or $b^\prime = 0$ by Proposition \ref{prop:B2_C2_multiplicity_2}, so we are in one of the cases 8b or 8k of Table \ref{tab:B2MF}.
	Finally, if $2a+b  \geq 2p$ then $\mu \in \{ \varpi_1 , \varpi_2 \}$ by \ref{prop:B2_C3_multiplicity_3}, and we are again in case 8b of Table \ref{tab:B2MF}.
\end{proof}

Now we are ready to establish the necessary conditions in Theorem \ref{thm:B2MF}, that is, we prove that for weights $\lambda,\mu \in X_1 \setminus \{ 0 \}$ such that the tensor product $L(\lambda) \otimes L(\mu)$ is multiplicity free, the weights $\lambda$ and $\mu$ must satisfy one of the conditions from Table \ref{tab:B2MF}, up to interchanging $\lambda$ and $\mu$.

\begin{proof}[Proof of Theorem \ref{thm:B2MF}, necessary conditions]
\label{proof:B2MFnecessary}	
	Let $\lambda,\mu \in X_1 \setminus \{ 0 \}$ and suppose that the tensor product $L(\lambda) \otimes L(\mu)$ is multiplicity free.
	Then $L(\lambda) \otimes L(\mu)$ is completely reducible by Lemma \ref{lem:MFimpliesCR}, and so $\lambda$ and $\mu$ satisfy one of the conditions in Table \ref{tab:B2CR} by Theorem \ref{thm:B2CR} (up to interchanging $\lambda$ and $\mu$).
	As the cases 1--7 in Table \ref{tab:B2CR} match the cases 1--7 in Table \ref{tab:B2MF}, it remains to consider the pairs of weights $\lambda$ and $\mu$ such that $\lambda \in C_0 \cup C_1\cup C_2\cup C_3$ and $\mu$ is reflection small with respect to $\lambda$ (i.e.\ case 8 in Table \ref{tab:B2CR}).
	We consider the alcoves $C_0,C_1,C_2,C_3$ in turn.
	\begin{itemize}
		\item If $\lambda\in C_0$ then we are in case 8a of Table \ref{tab:B2MF} by Proposition \ref{prop:B2MFC0C0}.
		\item If $\lambda \in C_1$ then we are in one of the cases 8b--8f or 8c\om--8f\om of Table \ref{tab:B2MF} by Corollary \ref{cor:B2MFC1conditions}.
		\item If $\lambda\in C_2$ then we are in one of the cases 8b, 8g or 8h of Table \ref{tab:B2MF} by Corollary \ref{cor:B2MFC2conditions}.
		\item If $\lambda\in C_3$ then we are in one of the cases 8b, 8i, 8j or 8k of Table \ref{tab:B2MF} by Corollary \ref{cor:B2MFC3conditions}.
	\end{itemize}
	In all cases, the weights $\lambda$ and $\mu$ satisfy one of the conditions from Table \ref{tab:B2MF}, as claimed.
\end{proof}

\bibliographystyle{alpha}
\bibliography{tensor}

\end{document}